\documentclass[reqno,11pt]{article}
\usepackage{a4wide}
\usepackage{color,eucal,enumerate,mathrsfs,amsthm}
\usepackage[normalem]{ulem}
\usepackage{amsmath,amssymb,amsthm,epsfig,bbm}
\numberwithin{equation}{section}

\usepackage{amsfonts}
\usepackage{dsfont}
\usepackage{mathtools}
\usepackage{leftidx}
\usepackage{graphicx}
\usepackage{esint}
\usepackage{authblk}

\usepackage[pdfborder={0 0 0}]{hyperref}  

\usepackage[latin1]{inputenc}
\usepackage[english]{babel}

\newtheorem{Definition}{Definition}[section]
\newtheorem{Proposition}[Definition]{Proposition}
\newtheorem{Lemma}[Definition]{Lemma}
\newtheorem{Theorem}[Definition]{Theorem}
\newtheorem{Corollary}[Definition]{Corollary}
\theoremstyle{definition}

\newtheorem{Remark}[Definition]{Remark}
\newtheorem{Setting}[Definition]{Setting}

\allowdisplaybreaks

\renewcommand{\H}{\mathbb{H}}

\newcommand{\N}{\mathbb{N}}

\newcommand{\R}{\mathbb{R}}

\newcommand{\hh}{{\mbox{\boldmath$h$}}}

\newcommand{\mm}{{\mbox{\boldmath$m$}}}

\newcommand{\ggamma}{{\mbox{\boldmath$\gamma$}}}

\newcommand{\ppi}{{\mbox{\boldmath$\pi$}}}

\newcommand{\sggamma}{{\mbox{\scriptsize\boldmath$\gamma$}}}

\newcommand{\sfd}{{\sf d}}

\newcommand{\sfh}{{\sf h}}

\newcommand{\Kliminf}{K\kern-3pt-\kern-2pt\mathop{\rm lim\,inf}\limits}  
\newcommand{\supp}{\mathop{\rm supp}\nolimits}   
\newcommand{\Lip}{\mathop{\rm Lip}\nolimits}          
\renewcommand{\d}{{\mathrm d}}
\newcommand{\dt}{{\d t}}
\newcommand{\ddt}{{\frac \d\dt}}

\newcommand{\restr}[1]{\lower3pt\hbox{$|_{#1}$}} 

\newcommand{\la}{\left<}                  
\newcommand{\ra}{\right>}
\newcommand{\eps}{\varepsilon}  
\newcommand{\nchi}{\chi}
\newcommand{\weakto}{\rightharpoonup}

\newcommand{\limi}{\varliminf}
\newcommand{\lims}{\varlimsup}

\newcommand{\fr}{\penalty-20\null\hfill$\blacksquare$}                      

\newcommand{\adm}{{\rm{Adm}}}                   
\newcommand{\gopt}{{\rm{OptGeo}}}                   
\newcommand{\prob}[1]{\mathscr P(#1)}                   
\newcommand{\probt}[1]{\mathscr P_2(#1)}                   
\newcommand{\e}{{\rm{e}}}                           

\renewcommand{\mm}{\mathfrak m}                                

\renewenvironment{proof}{\removelastskip\par\medskip   
\noindent{\em proof} \rm}{\penalty-20\null\hfill$\square$\par\medbreak}

\newcommand{\bd}{{\mathbf\Delta}}

\newcommand{\testi}[1]{{\rm Test}^{\infty}(#1)}
\newcommand{\testl}[1]{{\rm Test}^\infty_{loc}(#1)}

\newcommand{\X}{{\rm X}}

\newcommand{\h}{{\sfh}}

\renewcommand{\ae}{{\textrm{\rm{-a.e.}}}}
\newcommand{\aeon}{{\textrm{\rm{-a.e. on }}}}

\newcommand{\CD}{{\sf CD}}
\newcommand{\RCD}{{\sf RCD}}
\newcommand{\Ggamma}{{\mathbf\Gamma}}

\newcommand{\lip}{{\rm lip}}

\newcommand{\HS}{{\lower.3ex\hbox{\scriptsize{\sf HS}}}}
\renewcommand{\H}[1]{{\rm Hess}(#1)}
\renewcommand{\div}{{\rm div}}

\newcommand{\Y}{{\rm Y}}
\newcommand{\hr}{{\sf r}}
\newcommand{\hR}{{\sf R}}

\newcommand{\I}{\mathcal I}
\newcommand{\CC}{\mathcal C}
\renewcommand{\hh}{{\sf H}}

\title{Second order differentiation formula on $\RCD^*(K,N)$ spaces}
\author{Nicola Gigli \thanks{SISSA, Trieste. email: ngigli@sissa.it} \quad Luca Tamanini \thanks{Institut f\"ur Angewandte Mathematik, Universit\"at Bonn. email: tamanini@iam.uni-bonn.de}}

\begin{document}

\maketitle

\begin{abstract}
Aim of this paper is to prove the second order differentiation formula for $H^{2,2}$ functions along geodesics in $\RCD^*(K,N)$ spaces with $N < \infty$. This formula is new even in the context of Alexandrov spaces, where second order differentiation is typically related to semiconvexity.

We establish this result by showing that $W_2$-geodesics can be approximated up to second order, in a sense which we shall make precise, by entropic interpolation. In turn this is achieved by proving new,  even in the smooth setting, estimates concerning entropic interpolations which we believe are interesting on their own. In particular we obtain:
\begin{itemize}
\item[-] equiboundedness of the densities along the entropic interpolations,
\item[-] local equi-Lipschitz continuity of the Schr\"odinger potentials,
\item[-] a uniform weighted $L^2$ control of the Hessian of such potentials.
\end{itemize}
Finally, the techniques adopted in this paper can be used to show that in the $\RCD$ setting the viscous solution of the Hamilton-Jacobi equation can be obtained  via a vanishing viscosity method, in accordance with the smooth case.

With respect to a previous version, where the space was assumed to be compact, in this paper the second order differentiation formula is proved in full generality.
\end{abstract}

\tableofcontents

\section{Introduction}

In the last ten years there has been a great interest in the study of metric measure spaces with Ricci curvature bounded from below, see for instance \cite{Lott-Villani09},  \cite{Sturm06I}, \cite{Sturm06II}, \cite{Gigli-Kuwada-Ohta10}, \cite{AmbrosioGigliSavare11}, \cite{AmbrosioGigliSavare11-2},  \cite{Gigli12}, \cite{AmbrosioGigliSavare12}, \cite{Rajala12}, \cite{RajalaSturm12}, \cite{Gigli-Mosconi12},  \cite{Gigli13}, \cite{Gigli13over}, \cite{Ketterer13}, \cite{AmbrosioMondinoSavare13}, \cite{Mondino-Naber14}, \cite{CavMon15}, \cite{CavMil16}. The starting points of this research line have been the seminal papers \cite{Lott-Villani09} and \cite{Sturm06I}, \cite{Sturm06II} which linked lower Ricci bounds on metric measure spaces to properties of entropy-like functionals in connection with $W_2$-geometry.  Later (\cite{AmbrosioGigliSavare11}) it emerged  that also Sobolev calculus is linked to $W_2$-geometry and building on top of this the original definition of $\CD$ spaces by Lott-Sturm-Villani has evolved into that of $\RCD$ spaces (\cite{AmbrosioGigliSavare11-2}, \cite{Gigli12}).

An example of link between Sobolev calculus and $W_2$-geometry is the following result (a minor variant of a statement in \cite{Gigli13}). It says that we can safely take one derivative of a $W^{1,2}(\X)$ function along an optimal geodesic test plan $\ppi$, i.e.\ a test plan satisfying
\[
\iint_0^1|\dot\gamma_t|^2\,\d t\,\d\ppi(\gamma) = W_2^2\big((\e_0)_*\ppi,(\e_1)_*\ppi\big).
\]

\begin{Theorem}[First order differentiation formula]\label{thm:1}
Let $(\X,\sfd,\mm)$ be an $\RCD(K,\infty)$ space, $\ppi$ an optimal geodesic test plan with bounded support (equivalently: such that $\{\gamma_t\,:\, t\in[0,1],\ \gamma\in\supp(\ppi)\}\subset\X$ is bounded) and $f\in W^{1,2}(\X)$. 

Then the map $[0,1] \ni t\mapsto f \circ \e_t \in L^2(\ppi)$ is in $C^1([0,1],L^2(\ppi))$ and we have
\[
\ddt \big( f\circ\e_t\big) = \la\nabla f,\nabla\phi_t\ra\circ\e_t,\\
\]
for every $t \in [0,1]$, where $\e_t : C([0,1],\X) \to \X, \gamma \mapsto \gamma_t$ is the evaluation map and $\phi_t$ is any function such that for some $s \neq t$, $s \in [0,1]$, the function $-(s-t)\phi_t$ is a Kantorovich potential from $(\e_t)_*\ppi$ to $(\e_s)_*\ppi$.
\end{Theorem}

Recall that on $\RCD(K,\infty)$ spaces every $W_2$-geodesic $(\mu_t)$ between measures with bounded density and support is such that $\mu_t \leq C\mm$ for every $t \in [0,1]$ and some $C>0$ (\cite{RajalaSturm12}), so that between two such measures there always exists a (unique) optimal geodesic test plan with bounded support. Thus the theorem also says that we can find `many' $C^1$ functions on $\RCD$ spaces. We remark that such $C^1$ regularity - which was crucial in \cite{Gigli13} - is non-trivial even if the function $f$ is assumed to be Lipschitz and that statements about $C^1$ smoothness are quite rare in metric geometry.

Furthermore, projecting from $\ppi$ to $\mu_t := (\e_t)_*\ppi$ one can see that Theorem \ref{thm:1} immediately implies
\begin{equation}\label{eq:bonn}
\frac\d{\d t}\int f\,\d\mu_t = \int\la\nabla f,\nabla\phi_t\ra\,\d\mu_t
\end{equation}
and one might think of this identity as an `integrated' version of the basic formula
\[
\frac{\d}{\d t}f(\gamma_t) = \d f(\gamma'_t)
\]
valid in the smooth framework; at the technical level the proof of the claim has to do with the fact that the geodesic $(\mu_t)$ solves the continuity equation
\begin{equation}
\label{eq:cint}
\frac\d{\d t}\mu_t + \div(\nabla\varphi_t\mu_t)=0,
\end{equation}
where the $\varphi_t$'s are appropriate choices of Kantorovich potentials (see also \cite{GigliHan13} in this direction), and with the fact that $\nabla\varphi_t = \nabla\phi_t$ (see Lemma \ref{lem:gradpot} below).

\bigskip

In \cite{Gigli14}, the first author developed a second-order calculus on $\RCD$ spaces, in particular defining the space $H^{2,2}(\X)$ and for $f \in H^{2,2}(\X)$ the Hessian $\H f$, see \cite{Gigli14} and the Appendix. It is then natural to ask whether an `integrated' version of the second order differentiation formula
\[
\frac{\d^2}{\d t^2}f(\gamma_t) = \H f(\gamma'_t,\gamma'_t) \qquad \text{for $\gamma$ geodesic}
\]
holds in this framework. In this paper we provide affirmative answer to this question, our main result being:

\begin{Theorem}[Second order differentiation formula]\label{thm:2}
Let $(\X,\sfd,\mm)$ be an $\RCD^*(K,N)$ space, $N<\infty$, $\ppi$ an optimal geodesic test plan with bounded support and $f \in H^{2,2}(\X)$. 

Then the map $[0,1] \ni t \mapsto f \circ \e_t \in L^2(\ppi)$ is in $C^2([0,1],L^2(\ppi))$ and we have
\begin{equation}
\label{eq:th1}
\frac{\d^2}{\d t^2} \big( f\circ\e_t\big) = \H{f}(\nabla\phi_t,\nabla\phi_t)\circ\e_t,\\
\end{equation}
for every $t \in [0,1]$, where $\phi_t$ is as in Theorem \ref{thm:1}. 
\end{Theorem}

An equivalent formulation, which is the one we shall actually prove (see Theorem \ref{thm:main}) and is more in the spirit of \eqref{eq:bonn}, is the following:

\begin{Theorem}[Second order differentiation formula (2nd form)]\label{thm:2bis}
Let $(\X,\sfd,\mm)$ be an $\RCD^*(K,N)$ space, $N<\infty$, $\mu_0,\mu_1 \in \probt\X$ be such that $\mu_0,\mu_1 \leq C\mm$ for some $C>0$, with compact supports and let $(\mu_t)$ be the unique $W_2$-geodesic connecting $\mu_0$ to $\mu_1$. Also, let $f \in H^{2,2}(\X)$.

Then the map 
\[
[0,1] \ni \ t\quad \mapsto\quad \int f\,\d\mu_t\ \in\R
\]
belongs to $C^2([0,1])$ and it holds
\begin{equation}\label{eq:th2}
\frac{\d^2}{\d t^2}\int f\,\d\mu_t = \int \H f(\nabla\phi_t,\nabla\phi_t)\,\d\mu_t,
\end{equation}
for every $t \in [0,1]$, where $\phi_t$ is any function such that for some $s\neq t$, $s\in[0,1]$, the function $-(s-t)\phi_t$ is a Kantorovich potential from $\mu_t$ to $\mu_s$. 
\end{Theorem}

Let us comment about the assumptions in Theorem \ref{thm:2} and Theorem \ref{thm:2bis}:
\begin{itemize}
\item[-] The first order differentiation formula is valid on general $\RCD(K,\infty)$ spaces, while for the second order one we need to assume finite dimensionality. This is due to the strategy of our proof, which among other things uses the Li-Yau inequality.
\item[-] There exist optimal geodesic test plans without bounded support (if $K=0$ or the densities of the initial and final marginals decay sufficiently fast) but in this case the functions $\phi_t$ appearing in the statement(s) are not Lipschitz. As such it seems hard to have $\H{h}(\nabla\phi_t,\nabla\phi_t)\circ\e_t\in L^1(\ppi)$ and thus we cannot really hope for anything like \eqref{eq:th1}, \eqref{eq:th2} to hold: this explains the need of the assumption on bounded supports.
\end{itemize}

Having at disposal such second order differentiation formula is interesting not only at the theoretical level, but also for applications to the study of the geometry of $\RCD$ spaces. For instance, the proofs of both the splitting theorem \cite{Gigli13} and of the `volume cone implies metric cone' \cite{DePhilippisGigli16} in this setting can be greatly simplified by using such formula (in this direction, see \cite{Tamanini17} for comments about the splitting). Also, one aspect of the theory of $\RCD$ spaces which is not yet clear is whether they have constant dimension: for Ricci-limit spaces this is known to be true by a result of Colding-Naber \cite{ColdingNaber12} which uses second order derivatives along geodesics in a crucial way. Thus our result is necessary to replicate Colding-Naber argument in the non-smooth setting (but not sufficient: they also use a calculus with Jacobi fields which as of today does not have  a non-smooth counterpart).

\bigskip

Let us discuss the strategy of the proof. Our starting point is a related second order differentiation formula obtained in \cite{Gigli14}, available under proper regularity assumptions:

\begin{Theorem}\label{thm:1i}
Let $(\mu_t)$ be a $W_2$-absolutely continuous curve solving the continuity equation
\[
\frac\d{\d t}\mu_t+\div(X_t\mu_t)=0,
\] 
for some vector fields $(X_t)\subset L^2(T\X)$ in the following sense: for every $f\in W^{1,2}(\X)$ the map $t\mapsto\int f\,\d\mu_t$  is absolutely continuous and it holds
\[
\frac\d{\d t}\int f\,\d\mu_t=\int\la\nabla f,X_t\ra\,\d\mu_t.
\]
Assume that 
\begin{itemize}
\item[(i)] $t \mapsto X_t \in L^2(T{\X})$ is absolutely continuous,
\item[(ii)] $\sup_t\{ \|X_t\|_{L^2} + \|X_t\|_{L^{\infty}}+ \|\nabla X_t\|_{L^2} \} < +\infty$.
\end{itemize} 
Then for $f\in H^{2,2}(\X)$ the map $t\mapsto\int f\,\d\mu_t$ is $C^{1,1}$ and the formula
\begin{equation}
\label{eq:secondsmooth}
\frac{\d^2}{\d t^2}\int f\d\mu_t = \int \H{f}(X_t,X_t) + \la\nabla f,\tfrac\d{\d t} X_t  + \nabla_{X_t} X_t\ra  \d\mu_t
\end{equation}
holds for a.e.\ $t \in [0,1]$.
\end{Theorem}

If the vector fields $X_t$ are of gradient type, so that $X_t = \nabla\phi_t$ for every $t$ and the `acceleration' $a_t$ is defined as
\[
\frac\d{\d t}\phi_t+\frac{|\nabla\phi_t|^2}2=:a_t
\]
then \eqref{eq:secondsmooth} reads as
\begin{equation}
\label{eq:secondsmooth2}
\frac{\d^2}{\d t^2}\int f\d\mu_t = \int \H{f}(\nabla\phi_t,\nabla\phi_t)\,\d\mu_t+\int  \la\nabla f,\nabla a_t\ra  \d\mu_t.
\end{equation}
In the case of geodesics, the functions $\varphi_t$ appearing in \eqref{eq:cint} solve (in a sense which we will not make precise here) the Hamilton-Jacobi equation
\begin{equation}
\label{eq:hji}
\frac\d{\d t}\varphi_t + \frac{|\nabla\varphi_t|^2}{2} = 0,
\end{equation}
thus in this case the acceleration $a_t$ is identically 0. Hence if the vector fields $(\nabla\varphi_t)$ satisfy the regularity requirements $(i),(ii)$ in the last theorem we would easily be able to establish Theorem \ref{thm:2}. However in general this is not the case; informally speaking this has to do with the fact that for solutions of the Hamilton-Jacobi equations we do not have sufficiently strong second order estimates.

In order to establish Theorem \ref{thm:2} it is therefore natural to look for suitable `smooth' approximation of geodesics for which we can apply Theorem \ref{thm:1i} above and then pass to the limit in formula \eqref{eq:secondsmooth}. Given that the lack of smoothness of $W_2$-geodesic is related to the lack of smoothness of solutions of \eqref{eq:hji}, also in line with the classical theory of viscous approximation for the Hamilton-Jacobi equation there is a quite natural thing to try:  solve, for $\eps>0$, the equation
\[
\frac\d{\d t}\varphi^\eps_t=\frac{|\nabla\varphi^\eps_t|^2}{2}+\frac\eps2\Delta\varphi^\eps_t,\qquad\qquad\varphi^\eps_0:=\varphi,
\]
where $\varphi$ is a given, fixed, Kantorovich potential for the geodesic $(\mu_t)$, and then solve
\[
\frac\d{\d t}\mu^\eps_t-\div(\nabla\varphi^\eps_t\mu^\eps_t)=0,\qquad\qquad\mu^\eps_0:=\mu_0.
\]
This plan can actually be pursued and following the ideas in this paper one can show that if the space $(\X,\sfd,\mm)$ is $\RCD^*(K,N)$ and the geodesic $(\mu_t)$ is made of measures with equibounded densities, then as $\eps\downarrow0$:
\begin{itemize}
\item[i)] the curves $(\mu_t^\eps)$ $W_2$-uniformly converge to the geodesic $(\mu_t)$ and the measures $\mu^\eps_t$ have equibounded densities.
\item[ii)] the functions $\varphi^\eps_t$ are equi-Lipschitz and converge both uniformly and in the $W^{1,2}$-topology to the only viscous solution $(\varphi_t)$ of \eqref{eq:hji} with $\varphi$ as initial datum; in particular the continuity equation \eqref{eq:cint} for the limit curve holds.
\end{itemize}
These convergence results are based on Hamilton's gradient estimates and the Li-Yau inequality and are sufficient to pass to the limit in the term with the Hessian in \eqref{eq:secondsmooth2}. For these curves the acceleration is given by $a_t^\eps=-\frac\eps2\Delta\varphi^\eps_t$ and thus we are left to prove  that the quantity
\[
\eps\int\la\nabla f,\nabla\Delta\varphi^\eps_t \ra\,\d\mu^\eps_t
\]
goes to 0 in some sense. However, there appears to be \emph{no hope of obtaining this by PDE estimates}. The problem is that this kind of viscous approximation can produce in the limit a curve which is not a geodesic if $\varphi$ is not $c$-concave: shortly said, this happens as  soon as a shock appears in Hamilton-Jacobi. Since there is no hope for formula \eqref{eq:th2} to be true for non-geodesics, we see that there is little chance of obtaining it via such viscous approximation.

\bigskip

We therefore use another way of approximating geodesics: the slowing down of entropic interpolations.  Let us briefly describe what this is in the familiar Euclidean setting.

Fix two probability measures $\mu_0=\rho_0\mathcal L^d$, $\mu_1=\rho_1\mathcal L^d$ on $\R^d$. The Schr\"odinger functional equations are
\begin{equation}
\label{eq:sch10}
\rho_0=f\,\h_1g\qquad\qquad\qquad\qquad\rho_1=g\,\h_1f,
\end{equation}
the unknown being the Borel functions $f,g:\R^d\to[0,\infty)$, where $\h_tf$ is the heat flow starting at $f$ evaluated at time $t$. It turns out that in great generality these equations admit a solution which is unique up to the trivial transformation $(f,g)\mapsto (cf,g/c)$ for some constant $c>0$. Such solution can be found in the following way: let $\hR$ be the measure on $(\R^d)^2$ whose density w.r.t.\ $\mathcal L^{2d}$ is given by the heat kernel $\hr_t(x,y)$ at time $t=1$ and minimize the Boltzmann-Shannon entropy $H(\ggamma\,|\,\hR)$ among all transport plans $\ggamma$ from $\mu_0$ to $\mu_1$. The Euler equation for the minimizer forces it to be of the form  $f\otimes g\,\hR$ for some Borel functions $f,g:\R^d\to[0,\infty)$, where $f\otimes g(x,y):=f(x)g(y)$ (we shall reprove this known result in Proposition \ref{pro:2}). Then the fact that  $f\otimes g\,\hR$ is a transport plan from $\mu_0$ to $\mu_1$ is equivalent to $(f,g)$ solving \eqref{eq:sch10}. 

Once we have found the solution of \eqref{eq:sch10} we can use it in conjunction with the heat flow to interpolate from $\rho_0$ to $\rho_1$ by defining
\[
\rho_t:=\h_tf\,\h_{1-t}g.
\]
This is called entropic interpolation. Now we slow down the heat flow: fix $\eps>0$ and by mimicking the above find $f^\eps,g^\eps$ such that
\[
\rho_0=f^\eps\,\h_{\eps/2}g^\eps\qquad\qquad\rho_1=g^\eps\,\h_{\eps/2}f^\eps,
\]
(the factor $1/2$ plays no special role, but is convenient in computations).
Then define 
\[
\rho^\eps_t:=\h_{t\eps/2}f^\eps\,\h_{(1-t)\eps/2}g^\eps.
\]

The  remarkable and non-trivial fact here is that as $\eps\downarrow0$ the curves of measures $(\rho^\eps_t\mathcal L^d)$ converge to the $W_2$-geodesic from $\mu_0$ to $\mu_1$. 

The first connections between Schr\"odinger equations and optimal transport have been obtained by Mikami in \cite{Mikami04} for the quadratic cost on $\R^d$; later Mikami-Thieullen \cite{MikamiThieullen08} showed that a link persists even for more general cost functions. The statement we have just made about convergence of entropic interpolations to displacement ones has been proved by L\'eonard in \cite{Leonard12}. Actually, L\'eonard worked in much higher generality: as it is perhaps clear from the presentation, the construction of entropic interpolation can be done in great generality, as only a  heat kernel is needed. He also provided a basic intuition about why such convergence is in place: the basic idea is that if the heat kernel admits the asymptotic expansion $\eps\log \hr_\eps(x,y)\sim -\frac{\sfd^2(x,y)}{2}$ (in the sense of Large Deviations), then the rescaled entropy functionals $\eps H(\cdot\,|\, \hR_\eps)$ converge to $\frac12\int \sfd^2(x,y)\,\d \cdot $ (in the sense of $\Gamma$-convergence). We refer to \cite{Leonard14} for a deeper discussion of this topic, historical remarks and much more.

\bigskip

Starting from these intuitions and results, working in the setting of $\RCD^*(K,N)$ spaces we gain new information about the convergence of entropic interpolations to displacement ones. In order to state our results, it is convenient to introduce the Schr\"odinger potentials $\varphi^\eps_t,\psi^\eps_t$ as
\[
\varphi^\eps_t:=\eps\log \h_{t\eps/2}f^\eps\qquad\qquad\qquad\qquad\psi^\eps_t:=\eps\log \h_{(1-t)\eps/2}g^\eps.
\]
In the limit $\eps\downarrow0$ these will converge to forward and backward Kantorovich potentials along the limit geodesic $(\mu_t)$ (see below). In this direction, it is worth to notice that while for $\eps>0$ there is a tight link between potentials and densities, as we trivially have
\[
\varphi^\eps_t+\psi^\eps_t=\eps\log\rho^\eps_t,
\]
in the limit this becomes the well known (weaker) relation that is in place between  forward/backward Kantorovich potentials and measures $(\mu_t)$:
\[
\begin{split}
\varphi_t+\psi_t&=0\qquad\text{on }\supp(\mu_t),\\
\varphi_t+\psi_t&\leq 0\qquad\text{on }\X,
\end{split}
\]
see e.g.\ Remark 7.37 in \cite{Villani09} (paying attention to the different sign convention). By direct computation one can verify that $(\varphi^\eps_t),(\psi^\eps_t)$ solve the Hamilton-Jacobi-Bellman equations
\begin{equation}
\label{eq:hj2}
\frac{\d}{\d t}\varphi^{\varepsilon}_t  =  \frac{1}{2}|\nabla\varphi^{\varepsilon}_t|^2 + \frac{\varepsilon}{2}\Delta\varphi^{\varepsilon}_t\qquad\qquad\qquad\qquad-\frac{\d}{\d t}\psi^{\varepsilon}_t  =  \frac{1}{2}|\nabla\psi^{\varepsilon}_t|^2 + \frac{\varepsilon}{2}\Delta\psi^{\varepsilon}_t,
\end{equation}
thus introducing the functions 
\[
\vartheta^\eps_t:=\frac{\psi^\eps_t-\varphi^\eps_t}2
\]
it is not hard to check that it holds
\begin{equation}
\label{eq:cei}
\frac\d{\d t}\rho^\eps_t+\rm{div}(\nabla\vartheta^\eps_t\,\rho^\eps_t)=0
\end{equation}
and
\[
\frac\d{\d t}\vartheta^\eps_t+\frac{|\nabla\vartheta^\eps_t|^2}2=a^\eps_t,\qquad\qquad\text{where}\qquad a^\eps_t:= -\frac{\varepsilon^2}{8}\Big(2\Delta\log\rho^{\varepsilon}_t + |\nabla\log\rho^{\varepsilon}_t|^2\Big).
\]
With this said, our main results about entropic interpolations can be summarized  as follows. Under the assumptions that the metric measure space $(\X,\sfd,\mm)$ is $\RCD^*(K,N)$, $N<\infty$, and that $\rho_0,\rho_1$ belong to $L^\infty(\X)$ with bounded supports it holds:
\begin{itemize}
\item[-]\noindent\underline{Zeroth order} 
\begin{itemize}
\item[--]\emph{bound} For some $C>0$ we have $\rho^\eps_t\leq C$ for every $\eps\in(0,1)$ and $t\in[0,1]$.
\item[--]\emph{convergence} The curves  $(\rho^\eps_t\mm)$ $W_2$-uniformly converge to the unique $W_2$-geodesic $(\mu_t)$ from $\mu_0$ to $\mu_1$ and setting $\rho_t := \frac{\d\mu_t}{\d\mm}$ it holds $\rho_t^\eps \stackrel{\ast}{\rightharpoonup} \rho_t$ in $L^\infty(\X)$ for all $t \in [0,1]$.
\end{itemize}

\item[-]\noindent\underline{First order}
\begin{itemize}
\item[--]\emph{bound} For any $t \in (0,1]$ the functions $\{\varphi^\eps_t\}_{\eps\in(0,1)}$ are locally equi-Lipschitz. Similarly for the $\psi$'s. 
\item[--]
\emph{convergence} For every sequence $\eps_n\downarrow0$ there is a subsequence - not relabeled - such that for any $t\in(0,1]$ the functions $\varphi^\eps_t$ converge both locally uniformly and in $W^{1,2}_{loc}(\X)$ to a function $\varphi_t$ such that $-t\varphi_t$ is a Kantorovich potential from $\mu_t$ to $\mu_0$. Similarly for the $\psi$'s. 
\end{itemize}

\item[-]\noindent\underline{Second order} For every $\delta\in(0,1/2)$ we have
\begin{itemize}
\item[--]\emph{bound}
\begin{equation}
\label{eq:2i}
\begin{split}
&\sup_{\eps\in(0,1)}\iint_\delta^{1-\delta} \big(|\H{\vartheta^\eps_t}|_\HS^2+\eps^2|\H{\log\rho^\eps_t}|_\HS^2\big)\rho^\eps_t\,\d t\,\d\mm<\infty,\\
&\sup_{\eps\in(0,1)}\iint_\delta^{1-\delta} \big(|\Delta{\vartheta^\eps_t}|^2+\eps^2|\Delta{\log\rho^\eps_t}|^2\big)\rho^\eps_t\,\d t\,\d\mm<\infty.
\end{split}
\end{equation}
Notice that since in general the Laplacian is not the trace of the Hessian, there is no direct link between these two bounds.
\item[--]\emph{convergence} For every function $h\in W^{1,2}(\X)$ with $\Delta h\in L^\infty(\X)$ it holds
\begin{equation}
\label{eq:2ii}
\lim_{\eps\downarrow0}\iint_\delta^{1-\delta}\la\nabla h,\nabla a^\eps_t\ra \rho^\eps_t\,\d t\,\d\mm=0.
\end{equation}
\end{itemize}
\end{itemize}
\bigskip

With the exception of the convergence $\rho^\eps_t\mm\to \mu_t$, all these results are new even on  smooth manifolds (in fact, even on $\R^d$). 

The zeroth and first order bounds are both consequences of the Hamilton-Jacobi-Bellman equations \eqref{eq:hj2} satisfied by the $\varphi$'s and $\psi$'s and can be obtained from Hamilton's gradient estimate and the Li-Yau inequality. The facts that the limit curve is the $W_2$-geodesic and that the limit potentials are Kantorovich potentials are consequence of the fact that we can pass to the limit in the continuity equation \eqref{eq:cei} and that the limit potentials satisfy the Hamilton-Jacobi equation. In this regard it is key that we approximate  at the same time both the `forward' potentials $\psi$ and the `backward' one $\varphi$: see the proof of Proposition \ref{pro:9} and recall that the simple viscous approximation may converge to curves which are not $W_2$-geodesics. 

Notice that these zeroth and first order convergences are sufficient to pass to the limit in the term with the Hessian in \eqref{eq:secondsmooth2}. As said, also the viscous approximation could produce the same kind of convergence. 

The crucial advantage of dealing with entropic interpolations (which has no counterpart in viscous approximation) is thus in the second order bounds and convergence results which show that the term with the acceleration in \eqref{eq:secondsmooth2} vanishes in the limit and thus eventually allows us to prove our main result Theorem \ref{thm:2}. In this direction, we informally point out  that being the geodesic equation a second order one, in searching for an approximation procedure it is natural to look for one producing some sort of second order convergence.

The limiting property \eqref{eq:2ii} is mostly a consequence - although perhaps non-trivial - of the bound \eqref{eq:2i} (see in particular Lemma \ref{lem:vanish} and the proof of Theorem \ref{thm:main}), thus let us focus on how to get \eqref{eq:2i}. The starting point here is a formula due to L\'eonard \cite{Leonard13}, who realized that there is a connection between entropic interpolation and lower Ricci bounds: he computed the second order derivative of the entropy along entropic interpolations and in this direction our contribution has been the rigorous proof in the $\RCD$ framework of his formal computations, thus getting
\begin{equation}
\label{eq:leo}
\frac{\d^2}{\d t^2}H(\mu^{\varepsilon}_t \,|\, \mm) =\int\rho^\eps_t\,\d\big(\Ggamma_2(\vartheta^\eps_t)+\tfrac{\eps^2}4\Ggamma_2(\log(\rho^\eps_t))\big)= \frac{1}{2}\int \rho^\eps_t\,\d\big(\Ggamma_2(\varphi^{\varepsilon}_t) + \Ggamma_2(\psi^{\varepsilon}_t)\big),
\end{equation}
where $\Gamma_2$ is the `iterated carr\'e du champ' operator defined as
\[
\Ggamma_2(f) := \bd\frac{|\nabla f|^2}2-\la\nabla f,\nabla\Delta f\ra
\]
(in the setting of $\RCD$ spaces some care is needed when handling this object, but let us neglect this issue here).

Observe that if $h : [0,1] \to \R^+$ is a convex function, then $-\frac{h(0)}{t} \leq h'(t) \leq \frac{h(1)}{1-t}$ for any $t \in (0,1)$ and thus
\begin{equation}
\label{eq:intderd}
\int_\delta^{1-\delta}h''(t)\,\d t = h'(1-\delta)-h'(\delta)\leq \frac{h(1)}{1-\delta}+\frac{h(0)}{\delta}.
\end{equation}
If we assume for simplicity that $K=0$ we have $\Ggamma_2\geq 0$, so that \eqref{eq:leo} tells in particular that $t\mapsto H(\mu^{\varepsilon}_t \,|\, \mm)$ is convex for any $\eps>0$, and if we also assume that $\mm(\X)=1$ such function is non-negative. Therefore \eqref{eq:intderd} gives that for any $\delta\in(0,1/2)$ it holds
\begin{equation}
\label{eq:per2i}
\sup_{\eps\in(0,1)}\int_\delta^{1-\delta}\int \rho^\eps_t\,\d\big(\Ggamma_2(\vartheta^\eps_t)+\tfrac{\eps^2}4\Ggamma_2(\log(\rho^\eps_t))\big)\,\d t\leq \frac{H(\mu_1 \,|\, \mm) }{1-\delta}+\frac{H(\mu_0 \,|\, \mm) }{\delta}<\infty.
\end{equation}
Recalling the Bochner inequalities (\cite{Erbar-Kuwada-Sturm13}, \cite{AmbrosioMondinoSavare13}, \cite{Gigli14})
\[
\begin{split}
\Ggamma_2(\eta)&\geq|\H\eta|_\HS^2\mm,\qquad\qquad\qquad\qquad\Ggamma_2(\eta)\geq\frac{(\Delta\eta)^2}N \mm,
\end{split}
\]
we see that \eqref{eq:2i} follows from \eqref{eq:per2i}. Then with some work (see Lemma \ref{lem:vanish} and Theorem \ref{thm:main} for the details) starting from \eqref{eq:2i} we can  deduce \eqref{eq:2ii} which in turn ensures that the term with the acceleration in \eqref{eq:secondsmooth2} vanishes in the limit $\eps\downarrow0$, thus leading to our main result Theorem \ref{thm:2}.

\bigskip

\noindent{\bf Acknowledgements} 

The authors wish to thank C. L\'eonard for the numerous inspiring conversations about entropic interpolation.

This research has been supported by the MIUR SIR-grant `Nonsmooth Differential Geometry' (RBSI147UG4). The second author is also grateful to the UFI/UIF for the financial support of the Vinci Programme.

\section{The Schr\"odinger problem}\label{sec:3}

Let $(\X,\tau)$ be a Polish space, $\mu_0,\mu_1 \in \prob\X$ and $\hR$ be a non-negative Radon measure on $\X^2$. Recall that $\ggamma\in\prob{\X^2}$ is called transport plan  for $\mu_0,\mu_1$ provided $\pi^0_*\ggamma=\mu_0$ and $\pi^1_*\ggamma=\mu_1$, where $\pi^0,\pi^1:\X^2\to \X$ are the canonical projections. We are interested in finding  a transport plan of the form
\[
\ggamma=f\otimes g\,\hR
\]
for certain Borel functions $f,g:\X\to[0,\infty)$, where $f\otimes g(x,y):=f(x)g(y)$. As we shall see in this short section, in great generality this problem can be solved in a unique way and the plan $\ggamma$ can be found as the minimum of
\[
\ggamma'\quad\mapsto\quad H(\ggamma'\,|\, \hR)
\]
among all transport plans from $\mu_0$ to $\mu_1$, where $H(\,\cdot\,|\,\cdot)$ is the Boltzmann-Shannon entropy. For appropriate choice of the reference measure $\hR$ (which will also be our choice in the following), this minimization problem is called Schr\"odinger problem, we refer to \cite{Leonard14} for a survey on the topic.

Let us first recall the definition of the relative entropy functional in the case of a reference measure with possibly infinite mass (see \cite{Leonard14b} for more details). Given a $\sigma$-finite measure $\nu$ on a Polish space $(\Y,\tau')$, there exists a measurable function $W : \Y \to [0,\infty)$ such that
\[
z_W := \int e^{-W}\d\nu < +\infty.
\]
Introducing the probability measure $\nu_W := z_W^{-1}e^{-W}\nu$, for any $\sigma \in \prob\Y$ such that $\int W\d\sigma < +\infty$ the Boltzmann-Shannon entropy is defined as
\begin{equation}\label{eq:entdef}
H(\sigma\,|\,\nu) := H(\sigma\,|\,\nu_W) - \int W\d\sigma - \log z_W
\end{equation}
where $H(\sigma\,|\,\nu_W)$ is in turn defined as
\[
H(\sigma\,|\,\tilde{\nu}):=\left\{\begin{array}{ll}
\displaystyle{\int\rho\log(\rho)\,\d\tilde{\nu}} & \qquad \text{ if } \sigma = \rho\tilde{\nu}\\
+\infty & \qquad \text{ if } \sigma \not\ll \tilde{\nu}
\end{array}\right.
\]
for all $\tilde{\nu} \in \prob\Y$; notice that Jensen's inequality  and the fact that $\tilde\nu\in \prob\Y$ grant that $\int\rho\log(\rho)\,\d\tilde{\nu}$ is well defined and non-negative, in particular the definition makes sense. The definition is meaningful, because if $\int W'\d\sigma < +\infty$ for another function $W'$ such that $z_{W'} < +\infty$, then
\[
H(\sigma\,|\,\nu_W) - \int W\d\sigma - \log z_W = H(\sigma\,|\,\nu_{W'}) - \int W'\d\sigma - \log z_{W'}.
\]
Hence $H(\,\cdot\,|\,\nu)$ is well defined for all $\sigma \in \prob\Y$ such that $\int W\d\sigma < +\infty$ for some non-negative measurable function $W$ with $z_W < +\infty$.

The following proposition collects the basic properties of the minimizer of the Schr\"odinger problem; we emphasize that point $(i)$  of the statement is already  known in the literature on the subject (see in particular \cite{Leonard01}, \cite{BLN94} and \cite{RuscThom93}) and there are similarities between point $(ii)$ and some results in \cite{BLN94}. A complete proof has already been presented in \cite{GigTam17} for the compact case; here we adapt the  arguments to our more general case. Notice that Radon measures on Polish spaces are always $\sigma$-finite, hence the above discussion about the Boltzmann-Shannon entropy applies.

\begin{Proposition}\label{pro:2}
Let $(\X,\tau,\mm)$ be a Polish space equipped with a non-negative Radon measure $\mm$ and let $\hR$ be a non-negative Radon measure on $\X^2$ such that $\pi^0_*\hR =\pi^1_* \hR = \mm$ and
\[
\mm\otimes\mm\ll \hR\ll\mm\otimes\mm. 
\]
Let $\mu_0 = \rho_0\mm$ and $\mu_1 = \rho_1\mm$ be Borel probability measures and assume that there exists a Borel function $B : \X \to [0,\infty)$ such that
\begin{equation}\label{eq:paris}
\int_{\X^2} e^{-B(x)-B(y)}\d\hR(x,y) < \infty \qquad \int B\,\d\mu_0 < \infty \qquad \int B\,\d\mu_1 < \infty.
\end{equation}
Then the following holds.
\begin{itemize}
\item[i)] Assume that 
\begin{equation}\label{eq:finent}
H(\mu_0 \otimes \mu_1 \,|\, \hR) < +\infty.
\end{equation}
Then:
\begin{itemize}
\item[i-a)]There exists a unique minimizer $\ggamma$ of $H(\,\cdot\,|\,\hR)$ among all transport plans from $\mu_0$ to $\mu_1$.
\item[i-b)] $\ggamma=f\otimes g\hR$ for appropriate Borel functions $f,g:\X\to[0,\infty)$ which are $\mm$-a.e. unique up to the trivial transformation $(f,g)\to (cf,g/c)$ for some $c>0$.
\end{itemize}
\item[ii)] Assume that  $\rho_0,\rho_1 \in L^\infty(\X,\mm)$ and that for some $c>0$ it holds
\begin{equation}
\label{eq:k8}
\hR \geq c\mm \otimes \mm\qquad\text{in $P_0 \times P_1$},
\end{equation}
where $P_0 := \{\rho_0>0\}$ and $P_1 := \{\rho_1>0\}$. Then:
\begin{itemize}
\item[ii-a)] The bound \eqref{eq:finent} holds.
\item[ii-b)] The functions $f,g$ given by point $(i\text{-}b)$ above are in $L^1\cap L^\infty(\mm)$ with
\begin{equation}
\label{eq:l1lifg}
\|f\|_{L^\infty(\mm)}\|g\|_{L^1(\mm)}\leq \frac{\|\rho_0\|_{L^\infty(\mm)}}c\quad\text{and}\quad\|f\|_{L^1(\mm)}\|g\|_{L^\infty(\mm)}\leq \frac{\|\rho_1\|_{L^\infty(\mm)}}c\qquad
\end{equation} and $\ggamma$ is the only transport plan which can be written as $f'\otimes g' \hR$ for $f',g':\X\to[0,\infty)$ Borel.
\end{itemize}
\end{itemize}
\end{Proposition}
\begin{proof}

\noindent{\bf{(i-a)}} Existence follows by the direct method of calculus of variations: the class of transport plans is not empty, narrowly compact (see e.g.\ \cite{AmbrosioGigliSavare08}) and $H(\,\cdot\,|\,\hR)$ is well defined therein; indeed by assumption $\int W\d\sigma < +\infty$ with $W(x,y) := B(x) + B(y)$ for all transport plan $\sigma$. Moreover by \eqref{eq:entdef} we have that
\[
H(\sigma\,|\,\hR) = H(\sigma\,|\,\hR_W) - \int B\,\d\mu_0 - \int B\,\d\mu_1 - \log z_W,
\]
so that $H(\,\cdot\,|\,\hR)$ is narrowly lower semicontinuous on the class of transport plans.

Since $H(\,\cdot\,|\,\hR)$ is strictly convex, uniqueness is equivalent to the existence of a $\ggamma \in \adm(\mu_0,\mu_1)$ with finite entropy w.r.t.\ $\hR$ and by \eqref{eq:finent} we conclude.

\medskip

\noindent{\bf{(i-b)}} The uniqueness part of the claim is trivial, so we concentrate on existence. Finiteness of the entropy in particular grants that $\ggamma\ll \hR$. Put $p:=\frac{\d\sggamma}{\d \hR}$ and let $P_0:=\{\rho_0>0\}$, $P_1:=\{\rho_1>0\}$. We start claiming that 
\begin{equation}
\label{eq:denspos}
p>0\quad \mm\otimes\mm\aeon P_0\times P_1.
\end{equation}
Notice that since $\mm\otimes\mm$ and $\hR$ are mutually absolutely continuous, the claim makes sense and arguing by contradiction we shall assume that $\hR(Z)>0$, where  $Z:=(P_0\times P_1)\cap \{p=0\}$.

Let $s:=\frac{\d(\mu_0\otimes\mu_1)}{\d \hR }$ and  for $\lambda\in(0,1)$ let us define $\Phi(\lambda):\X^2\to\R$ by
\[
\Phi(\lambda) := \frac{u(p + \lambda(s-p)) - u(p)}{\lambda},\qquad\text{where}\quad u(z):=z\log(z).
\]
The convexity of $u$ grants that $\Phi(\lambda)\leq u(s) - u(p)\in L^1(\X^2,\hR)$ (recall \eqref{eq:finent}) and   that  $\Phi(\lambda)$ is monotone decreasing as $\lambda \downarrow 0$. Moreover, on  $Z$ we have $\Phi(\lambda) \downarrow -\infty$ $\hR$-a.e.\ as $\lambda \downarrow 0$, thus  the monotone convergence theorem ensures that
\[
\lim_{\lambda \downarrow 0} \frac{H(\ggamma + \lambda(\mu_0\otimes\mu_1 - \ggamma) \,|\, \hR) - H(\ggamma \,|\, \hR)}{\lambda} = -\infty.
\]
Since $\ggamma + \lambda(\mu_0\otimes\mu_1 - \ggamma)$ is a transport plan from $\mu_0$ to $\mu_1$ for $\lambda\in(0,1)$, this is in contradiction with the minimality of $\ggamma$ which grants that the left-hand side is non-negative, hence $Z$ is $\hR$-negligible, as desired.

Let us now pick $h\in L^\infty(\X^2,\ggamma)$  such that $\pi^0_*( h\ggamma)=\pi^1_*(h\ggamma)=0$ and  $\eps\in(0, \|h\|^{-1}_{L^\infty(\sggamma)})$. Then $(1+\eps h)\ggamma$ is a transport plan from $\mu_0$ to $\mu_1$ and noticing that $hp$ is well defined $\hR$-a.e.\  we have
\[
\begin{split}
\|u((1+\eps h)p)\|_{L^1(\hR)}& =  \int| (1 + \varepsilon h)p\log((1 + \varepsilon h)p)|\,\d \hR  \\ 
& \leq  \int (1 + \varepsilon h)p\,|\log p|\,\d \hR + \int (1 + \varepsilon h)\,|\log(1 + \varepsilon h)|\,\d \ggamma \\ 
& \leq  \|1 + \varepsilon h\|_{L^{\infty}(\sggamma)} \|p\log p\|_{L^1(\hR)} + \|(1 + \varepsilon h)\log(1 + \varepsilon h)\|_{L^{\infty}(\sggamma)} ,
\end{split}
\]
so that $u((1+\eps h)p)\in L^1(\hR)$. Then again by the monotone convergence theorem we get
\[
\lim_{\eps\downarrow 0}\frac{H((1 + \varepsilon h)\ggamma \,|\, \hR)-H(\ggamma \,|\, \hR)}{\eps}=\int \lim_{{\eps\downarrow 0}}\frac{u((1+\eps h)p)-u(p)}\eps\,\d \hR=\int hp(\log p+1)\,\d \hR.
\] 
By the minimality of $\ggamma$ we know that the left-hand side in this last identity is non-negative, thus after running the same computation with $-h$ in place of $h$ and noticing that the choice of $h$ grants that $\int hp\,\d \hR=\int h\,\d\ggamma=0$ we obtain
\begin{equation}
\label{eq:kyoto1}
\int hp\log(p)\,\d \hR=0\qquad\forall h\in L^\infty(\ggamma)\ \text{such that}\ \pi^0_*( h\ggamma)=\pi^1_*(h\ggamma)=0.
\end{equation}
The rest of the argument is better understood by introducing the spaces $V,{}^\perp W\subset L^1(\ggamma)$ and $V^\perp,W\subset L^\infty(\ggamma)$ as follows
\[
\begin{split}
V&:=\big\{f\in L^1(\ggamma)\ :\ f=\varphi\oplus \psi\text{ for some }\varphi\in L^0(\mm\restr{P_0}),\ \psi\in L^0(\mm\restr{P_1})\big\},\\
W&:=\big\{h\in L^\infty(\ggamma)\ :\ \pi^0_*( h\ggamma)=\pi^1_*(h\ggamma)=0\big\},\\
V^\perp&:=\big\{h\in L^\infty(\ggamma)\ :\ \int fh\,\d\ggamma=0\ \forall f\in V\big\},\\
{}^\perp W&:=\big\{f\in L^1(\ggamma):\int fh\,\d\ggamma=0\ \forall h\in W\big\},
\end{split}
\]
where here and in the following the function $\varphi\oplus\psi$ is defined as $\varphi\oplus\psi(x,y):=\varphi(x)+\psi(y)$. Notice that the Euler equation \eqref{eq:kyoto1} reads as $\log(p)\in {}^\perp W$ and our thesis as $\log(p)\in V$; hence to conclude it is sufficient to show that ${}^\perp W\subset V$.

\noindent{\underline{Claim 1}: $V$ is a closed subspace of $L^1(\ggamma)$}.

We start claiming that  $f\in V$ if and only if $f\in L^1(\ggamma)$ and
\begin{equation}
\label{eq:kyoto2}
f(x,y)+f(x',y')=f(x,y')+f(x',y)\quad \mm\otimes\mm\otimes\mm\otimes\mm\ae\ (x,x',y,y')\in P_0^2\times P_1^2.
\end{equation}
Indeed the `only if' follows trivially from $\ggamma\ll\mm\otimes\mm$ and the definition of $V$. For the `if' we apply Fubini's theorem to get the existence of $x'\in P_0$ and $y'\in P_1$ such that
\[
f(x,y)+f(x',y')=f(x,y')+f(x',y)\quad \mm\otimes\mm\ae\ x,y\in P_0\times P_1.
\]
Thus $f = f(\cdot,y')\oplus(f(x',\cdot)-f(x',y'))$, as desired. 

Now notice that since \eqref{eq:denspos} grants that $(\mm\otimes\mm)\restr{P_0\times P_1} \ll \ggamma$, we see that the condition \eqref{eq:kyoto2} is closed w.r.t.\ $L^1(\ggamma)$-convergence.

\noindent{\underline{Claim 2}: $V^\perp\subset W$}.

Let $h\in L^\infty(\ggamma)\setminus W$, so that either the first or second marginal of $h\ggamma$ is non-zero. Say the first. Thus since $\pi^0_*\ggamma=\mu_0$ we have $\pi^0_*(h\ggamma)=f_0\mu_0$ for some $f_0\in L^\infty(\mu_0)\setminus\{0\}$. Then the function $f:=f_0\oplus 0=f_0\circ\pi^0$ belongs to $V$ and we have
\[
\int hf\,\d\ggamma=\int f_0\circ\pi^0\,\d(h\ggamma)=\int f_0\,\d\pi^0_*(h\ggamma)=\int f_0^2\,\d\mu_0>0,
\]
so that $h\notin V^\perp$.

\noindent{\underline{Claim 3}: ${}^\perp W\subset V$}.

Let $f\in L^1(\ggamma)\setminus V$, use the fact that  $V$ is closed and  the Hahn-Banach theorem to find $h\in L^\infty(\ggamma)\sim L^1(\ggamma)^*$ such that $\int fh\,\d\ggamma\neq 0$ and $\int \tilde fh\,\d\ggamma=0$ for every $\tilde f\in V$. Thus $h\in V^\perp$ and hence by the previous step $h\in W$. The fact that $\int fh\,\d\ggamma\neq 0$ shows that $f\notin {}^\perp W$, as desired.

\medskip

\noindent{\bf (ii-a)} The bounds \eqref{eq:paris} and \eqref{eq:k8} give that $\int e^{-B(x)-B(y)}\,\d(\mm\otimes\mm)\restr{P_0\times P_1}<\infty$, which together with \eqref{eq:paris} again grants that $
H(\mu_0\otimes\mu_1 \,|\, (\mm\otimes\mm)\restr{P_0\times P_1})$ is well defined. The assumption that $\rho_0,\rho_1\in L^\infty(\mm)$ then ensures that $
H(\mu_0\otimes\mu_1 \,|\, (\mm\otimes\mm)\restr{P_0\times P_1})$ is finite, hence the claim follows by direct computations:
\[
\begin{split}
H(\mu_0\otimes\mu_1 \,|\, \hR) & =  
H\big(\mu_0\otimes\mu_1 \,|\, (\mm\otimes\mm)\restr{P_0\times P_1}\big) + \int \log\bigg(\frac{\d\big((\mm \otimes \mm)\restr{P_0\times P_1}\big)}{\d \hR}\bigg)\rho_0\otimes \rho_1\d (\mm \otimes \mm)  \\ 
& \leq H\big(\mu_0\otimes\mu_1 \,|\, (\mm\otimes\mm)\restr{P_0\times P_1}\big) - \log(c)<\infty.
\end{split}
\]
\noindent{\bf (ii-b)} 
Then let $\sigma$ be a transport plan from $\mu_0$ to $\mu_1$ such that $\sigma = f' \otimes g' \hR$ for suitable non-negative Borel functions $f',g'$. We claim that in this case it holds $f',g'\in L^\infty(\mm)$, leading in particular to the claim in the statement about $\ggamma$.

By disintegrating $\hR$ w.r.t.\ $\pi^0$, from $\pi^0_*(f'\otimes g'\hR)=\rho_0\mm$ and $\hR_0 = \mm$ we get that
\begin{equation}\label{eq:44}
f'(x)\int g'(y)\,\d \hR_x(y) = \rho_0(x) < +\infty, \quad \textrm{for }\mm\ae\ x 
\end{equation}
whence $g' \in L^1(\hR_x)$ for $\mm$-a.e.\ $x$. Notice then that the sets where $f'$ and $g'$ are positive must coincide with $P_0$ and $P_1$ respectively, up to $\mm$-negligible sets, so that nothing changes in \eqref{eq:44} if we restrict the integral to $P_1$. Moreover, since from \eqref{eq:k8} we have that $\hR_x\geq c\mm$ in $P_1$ for $\mm$-a.e.\ $x \in P_0$, we see that $g'\in L^1(\mm)$ with 
\[
c \|g'\|_{L^1(\mm)}\leq \int g'(y)\,\d \hR_x( y)\quad \textrm{for }\mm\ae\ x \in P_0
\]
and thus \eqref{eq:44} yields
\[
f' \leq \frac{\|\rho_0\|_{L^{\infty}(\mm)}}{c\|g'\|_{L^1(\mm)}}, \quad \mm\ae\textrm{ in }P_0,
\]
which is the first in \eqref{eq:l1lifg}, because in $\X \setminus P_0$ we already know that $f'$ vanishes $\mm$-a.e. By interchanging the roles of $f'$ and $g'$, the same conclusion follows for $g'$.

For the uniqueness of $\ggamma$, put $\varphi:=\log f'$, $\psi:=\log g'$ and notice that, by what we have just proved, they are bounded from above. On the other hand
\[
\int \varphi\oplus\psi \,\d\sigma = H(\sigma\,|\,\hR) > -\infty
\]
because, as already remarked in the proof of $(i)$, \eqref{eq:paris} implies that $H(\cdot\,|\,\hR)$ is well defined on $\adm(\mu_0,\mu_1)$. From these two facts we infer that
\begin{equation}
\label{eq:kyoto3}
\varphi\circ\pi^0,\psi\circ\pi^1\in L^1(\sigma).
\end{equation}
Putting for brevity $p':=f'\otimes g'$ and arguing as before to justify the passage to the limit inside the integral we get
\[
\begin{split}
\frac{\d}{\d\lambda} H\big((1-\lambda)\sigma+\lambda\ggamma\,|\,\hR\big)\restr{\lambda=0^+}&=\int(p-p')\log(p')\,\d \hR \\
&=\int\varphi\oplus\psi\,\d(\ggamma-\sigma)\\
(\text{by }\eqref{eq:kyoto3})\qquad&=\int\varphi\,\d\pi^0_*(\ggamma-\sigma)+\int\psi\,\d\pi^1_*(\ggamma-\sigma)\\
(\text{because }\sigma\text{ and }\ggamma\text{ have the same marginals})\qquad&=0.
\end{split}
\]
This equality and the convexity of $H(\,\cdot\,|\,\hR)$ yield $H(\sigma\,|\,\hR)\leq H(\ggamma\,|\,\hR)$ and being $\ggamma$ the unique minimum of $H(\,\cdot\,|\,\hR)$ among transport plans from $\mu_0$ to $\mu_1$, we conclude that $\sigma=\ggamma$.
\end{proof}

The above result is valid in the very general framework of Polish spaces. We shall now restate it in the form we shall need in the context of $\RCD$ spaces.

Recall that on a finite-dimensional $\RCD^*(K,N)$ space $(\X,\sfd,\mm)$, $\mm$ satisfies the volume growth condition \eqref{eq:volgrowth}, so that we can choose $W = \sfd^2(\cdot,\bar{x})$ for any $\bar{x} \in \X$ in \eqref{eq:entdef}. Setting $z := \int e^{-\sfd^2(\cdot,\bar{x})}\d\mm$ and
\[
\tilde{\mm} := z^{-1}e^{-\sfd^2(\cdot,\bar{x})}\mm,
\]
the definition \eqref{eq:entdef}  becomes
\begin{equation}\label{eq:entdef2}
H(\mu\,|\,\mm) = H(\mu\,|\,\tilde{\mm}) - \int \sfd^2(\cdot,\bar{x})\d\mu - \log z
\end{equation}
and this shows that $H(\,\cdot\,|\,\mm)$ is well defined on $\probt\X$ and $W_2$-lower semincontinuous. Let us also remind that on $\RCD$ spaces there is a well defined heat kernel $\hr_\eps[x](y)$ (see \eqref{eq:hk} and \eqref{eq:rapprform}). The choice of working with $\hr_{\eps/2}$ in the following statement is convenient for the computations we will do later on.

\begin{Theorem}\label{thm:5}
Let $(\X,\sfd,\mm)$ be an $\RCD^*(K,N)$ space with $K \in \mathbb{R}$ and $N \in [1,\infty)$ endowed with a non-negative Radon measure $\mm$. For $\eps>0$ define $\hR^{\eps/2}\in\prob{\X^2}$ as
\[
\d \hR^{\eps/2}(x,y):= \hr_{\eps/2}[x](y)\,\d\mm(x)\,\d\mm(y).
\]
Also, let $\mu_0,\mu_1 \in \prob\X$ be Borel probability measures with bounded densities and supports.

Then there exist and are uniquely $\mm$-a.e.\ determined (up to multiplicative constants) two Borel non-negative functions $f^\eps, g^\eps : \X \to [0,\infty)$ such that $f^\eps\otimes g^\eps\hR^{\eps/2}$ is a transport plan from $\mu_0$ to $\mu_1$. In addition, $f^\eps,g^\eps$ belong to $L^{\infty}(\X)$ and their supports are included in $\supp(\mu_0)$ and $\supp(\mu_1)$ respectively.
\end{Theorem}

\begin{proof}
Start observing that $\hR^{\eps/2}_0 = \hR^{\eps/2}_1 = \mm$ and if we set $B := \sfd^2(\cdot,\bar{x})$ with $\bar{x} \in \X$ arbitrarily chosen, then the second and third in \eqref{eq:paris} are authomatically satisfied; for the first one notice that
\[
\int_{\X^2} e^{-B \oplus B}\d\hR^{\eps/2} = \int\bigg(e^{-\sfd^2(y,\bar{x})}\d\hR^{\eps/2}_x(y)\bigg)e^{-\sfd^2(x,\bar{x})} \d\mm(x),
\]
$e^{-\sfd^2(y,\bar{x})} \leq 1$, $\hR^{\eps/2}_x$ is a probability measure and recall \eqref{eq:volgrowth}. Hence Proposition \ref{pro:2}, the fact that the Gaussian estimates \eqref{eq:gaussest} on the heat kernel grant that there are constants $0 < c_{\varepsilon} \leq C_{\varepsilon} < +\infty$ such that
\[
c_{\varepsilon} \mm \otimes \mm \leq \hR^{\eps/2} \leq C_{\varepsilon} \mm \otimes \mm
\]
in $P_0 \times P_1$ and the fact that $f^\eps \otimes g^\eps \hR^{\eps/2}$ is a transport plan from $\mu_0$ to $\mu_1$ provide us with the conclusion.
\end{proof}

\section{Hamilton's and Li-Yau's estimates}\label{sec:4}
Here we recall Hamilton's gradient estimate and Li-Yau Laplacian estimates for $\log\h_tu$, where $u$ is a non-negative function. 

Let us start with the following result,  which we shall frequently use later on without explicit mention:
\begin{Proposition}\label{pro:testloc}
Let $(\X,\sfd,\mm)$ be an $\RCD^*(K,N)$ space with $K \in \R$ and $N \in [1,\infty)$, $t>0$ and  $u_0 \in L^2 \cap L^{\infty}(\X)$ be non-negative and not identically zero. Put $u_t := \h_t u_0$. 

Then $\log u_t \in \testl\X$.
\end{Proposition}
\begin{proof}
By \eqref{eq:regflow} $u_t \in \testi\X$ and by \eqref{eq:gaussest} $u_t$ is locally bounded away from 0. Taking into account the fact that $\log$ is smooth on $(0,\infty)$, the conclusion easily follows from \eqref{eq:stabcomp}.
\end{proof}
We now  recall Hamilton's gradient estimate on $\RCD(K,\infty)$ spaces, which is known to be true from \cite{JiangZhang16}:
\begin{Theorem}[Hamilton's gradient estimate]\label{pro:4}
Let $(\X,\sfd,\mm)$ be an $\RCD^*(K,\infty)$ space with $K \in \R$ and let $u_0 \in L^p \cap L^{\infty}(\X)$ be positive with $p \in [1,\infty)$. Put $u_t := \h_t u_0$ for all $t > 0$. Then
\[
t|\nabla\log u_t|^2 \leq (1 + 2K^-t)\log\bigg(\frac{\|u_0\|_{L^{\infty}(\X)}}{u_t}\bigg), \quad \mm\ae
\]
for all $t > 0$, where $K^-:=\max\{0,-K\}$.
\end{Theorem}
\begin{proof}
In  \cite{JiangZhang16} this result has been  stated for proper $\RCD(K,\infty)$ spaces; still, the assumption that bounded sets are relatively compact is never used so that the proof works in general $\RCD$ spaces. We remark that  in  \cite{JiangZhang16} the authors refer to  \cite{Gigli12}, \cite{AmbrosioGigliMondinoRajala15}, \cite{AmbrosioGigliSavare11-2} and \cite{Savare13} for the various calculus rules and that in these latter references no properness assumption is made.
\end{proof}

In the finite-dimensional case, thanks to the Gaussian estimates for the heat kernel we can easily obtain a bound independent of the $L^\infty$ norm of the initial datum:
\begin{Theorem}\label{cor:1}
Let $(\X,\sfd,\mm)$ be an $\RCD^*(K,N)$ space with $K \in \R$ and $N \in [1,\infty)$. Then there is a constant $C$ depending on $K,N$ only such that for any $u_0 \in L^1(\X)$ non-negative, not identically 0 and with bounded support the inequality
\begin{equation}\label{eq:23}
|\nabla\log (u_t)|^2 \leq C\bigg(1 + \frac{1}{t}\bigg)\bigg(1 + t + \frac{D_0^2(x)}{t}\bigg), \quad \mm\ae
\end{equation}
holds for all $t > 0$, where $u_t:=\h_t u$ and
\[
D_0(x) := \sup_{y \in \supp(u_0)} \sfd(x,y).
\]
In particular, for every $0 < \delta \leq T < \infty$ and $\bar{x} \in \X$ there is a constant $C_{\delta,T}>0$ depending on $K,N,\delta,T,\bar{x}$ and the diameter of $\supp(u_0)$ such that for every $\eps\in(0,1)$ it holds
\begin{equation}
\label{eq:ham}
\eps |\nabla\log(u_{\eps t})| \leq C_{\delta,T}\big(1 + \sfd(\cdot,\bar{x})\big) \qquad\forall t \in [\delta,T].
\end{equation}
\end{Theorem}

\begin{proof}
Recall the representation formula \eqref{eq:rapprform}
\[
u_t(x) = \int u_0(y)r_t[y](x)\,\d\mm(y) = \int_{\supp(u_0)} u_0(y)r_t[y](x)\,\d\mm(y) \qquad\forall x\in \X
\]
and that for the transition probability densities $r_t[y](x)$ we have the Gaussian estimates \eqref{eq:gaussest}, which can be simplified as
\[
\frac{C_0}{\mm(B_{\sqrt t}(y))}\exp\Big(-\frac{\sfd^2(x,y)}{3t} - C_2 t\Big)\leq \hr_t[x](y) \leq \frac{C_1}{\mm(B_{\sqrt t}(y))}e^{C_2 t}\qquad\forall x,y\in \X,
\]
for appropriate constants $C_0,C_1,C_2$ depending only on $K,N$. Therefore, we have
\[
\begin{split}
\|u_t\|_{L^\infty}&=\sup_x u_t(x)\leq C_1e^{C_2 t}\int_{\supp(u_0)}\frac{u(y)}{\mm(B_{\sqrt t}(y))}\,\d\mm(y),\\
\inf_x u_{2t}(x)&\geq C_0e^{-2C_2 t}e^{-\frac{D_0^2(x)}{t}}\int_{\supp(u_0)}\frac{u(y)}{\mm(B_{\sqrt {2t}}(y))}\,\d\mm(y)>0.
\end{split}
\]
By the fact that $\mm$ is uniformly locally doubling we know that it holds
\[
\mm(B_{\sqrt{2t}}(y)) \leq \mm(B_{\sqrt{t}}(y))C_3 e^{C_4\sqrt{t}} \qquad\forall y \in \X,\ t>0,
\]
where $C_3,C_4$ only depend on $K,N$. As a consequence, the above yields
\[
\frac{\|u_t\|_{L^\infty}}{u_{2t}(x)}\leq C_5 e^{3C_2 t + C_4\sqrt{t} + \frac{D_0^2(x)}{t}}\qquad\forall x\in\X,\ t>0.
\]
We now apply Proposition \ref{pro:4} with $u_t$ in place of $u_0$ (notice that the assumptions are fulfilled) to get
\[
t|\nabla\log(u_{2t})|^2 \leq (1+2K^-t)\log\Big(\frac{\|u_t\|_{L^\infty}}{u_{2t}}\Big) \leq (1+2K^-t)\Big(\log C_5 + 3C_2t + C_4\sqrt{t} + \frac{D_0^2(x)}{t}\Big)
\]
$\mm$-a.e., which is (equivalent to) the bound \eqref{eq:23}. The last statement is now obvious, noticing that
\[
D_0(x) \leq D_0(\bar{x}) + \sfd(x,\bar{x})
\]
for any $\bar{x} \in \supp(u_0)$.
\end{proof}

A further result that we shall need soon is the Li-Yau inequality in the form proved by Baudoin and Garofalo (see \cite{GarofaloMondino14} for the case of finite mass and \cite{Jiang15} for the general one). 

\begin{Theorem}[Li-Yau inequality]
Let $(\X,\sfd,\mm)$ be an $\RCD^*(K,N)$ space with $K \in \R$ and $N \in [1,\infty)$ and let $u_0 \in L^p({\X})$ for some $p \in [1,\infty)$ be non-negative. Put $u_t := \h_t u_0$ for all $t > 0$. Then
\begin{equation}
\label{eq:bgineq}
|\nabla \log u_t|^2 \leq e^{-2Kt/3}\frac{\Delta u_t}{u_t} + \frac{NK}{3}\frac{e^{-4Kt/3}}{1-e^{-2Kt/3}}, \qquad \mm\ae
\end{equation}
for all $t > 0$, where $\frac{NK}{3}\frac{e^{-4Kt/3}}{1-e^{-2Kt/3}}$ is understood as $\frac{N}{2t}$ when $K=0$.
\end{Theorem}

We restate such inequality  in the form that we shall use:
\begin{Theorem}
Let $(\X,\sfd,\mm)$ be an $\RCD^*(K,N)$ space with $K \in \R$ and $N \in [1,\infty)$. Then for every $0 < \delta \leq T < \infty$ and $\bar{x} \in \X$ there exists a constant $C_{\delta,T}>0$ depending on $K,N,\delta,T,\bar{x}$ and the diameter of $\supp(u_0)$ such that the following holds.

For any $u_0 \in L^1(\X)$ non-negative, not identically zero and with bounded support and for any $\eps\in(0,1)$ it holds
\begin{equation}
\label{eq:liyau}
\eps\Delta \log(\h_{\eps t}(u_0))\geq -C_{\delta,T}\big(1 + \sfd^2(\cdot,\bar{x})\big) \qquad\forall t \in [\delta,T].
\end{equation}
\end{Theorem}
\begin{proof}
Rewrite the Li-Yau inequality \eqref{eq:bgineq} as
\[
e^{-2Kt/3}\bigg(\underbrace{\frac{\Delta u_t}{u_t} - |\nabla\log u_t|^2}_{=\Delta\log u_t}\bigg) \geq (1-e^{-2Kt/3})|\nabla\log u_t|^2 - \frac{NK}{3}\frac{e^{-4Kt/3}}{1-e^{-2Kt/3}}
\]
and use Hamilton's gradient estimate  \eqref{eq:ham} to control $|\nabla\log u_t|^2$ in the right-hand side.
\end{proof}

\section{The Schr\"odinger problem: properties of the solutions}\label{sec:5}

\subsection{The setting}

Let us fix once for all the assumptions and notations which we shall use from now on.

\begin{Setting}\label{set}
$(\X,\d,\mm)$ is an $\RCD^*(K,N)$ space with $K \in \R$ and $N \in [1,\infty)$ and $\mu_0=\rho_0\mm$, $\mu_1=\rho_1\mm$ are two absolutely continuous Borel probability measures with bounded densities and supports.
 
For any $\eps>0$ we consider the couple $(f^{\varepsilon},g^{\varepsilon})$ given by Theorem \ref{thm:5} normalized in such a way that 
\begin{equation}
\label{eq:normalization}
\int\log(\h_{\frac\eps2} f^\eps)\rho_1\,\d\mm=0,
\end{equation}
then we set $\rho^\eps_0:=\rho_0$, $\rho^\eps_1:=\rho_1$, $\mu^\eps_0:=\mu_0$, $\mu^\eps_1:=\mu_1$ and 
\[
\left\{\begin{array}{l}
f^{\varepsilon}_t := \h_{\varepsilon t/2}f^{\varepsilon} \\
\\
\varphi_t^{\varepsilon} := \varepsilon\log f_t^{\varepsilon}\\
\\
\text{for }t\in(0,1]
\end{array}
\right.\qquad\qquad
\left\{\begin{array}{l}
g^{\varepsilon}_t := \h_{\varepsilon(1-t)/2}g^{\varepsilon} \\
\\
\psi_t^{\varepsilon} := \varepsilon\log g_t^{\varepsilon}\\
\\
\text{for }t\in[0,1)
\end{array}
\right.\qquad\qquad
\left\{\begin{array}{l}
\rho^{\varepsilon}_t := f^{\varepsilon}_t g^{\varepsilon}_t \\
\\
\mu^{\varepsilon}_t := \rho^{\varepsilon}_t\mm\\
\\
\vartheta^{\varepsilon}_t := \frac12({\psi^{\varepsilon}_t - \varphi^{\varepsilon}_t})\\
\\
\text{for }t\in(0,1)
\end{array}
\right.
\]
\end{Setting}

In order to investigate the time behaviour of the functions just defined, let us introduce the weighted $L^2$ and $W^{1,2}$ spaces. The weight we will always consider is $e^{-V}$ with $V = M\sfd^2(\cdot,\bar{x})$; because of \eqref{eq:volgrowth}, $e^{-V}\mm$ has finite mass for every $M>0$. For $L^2(\X,e^{-V}\mm)$ no comments are required. The weighted Sobolev space is denoted and defined as
\[
W^{1,2}(\X,e^{-V}\mm) := \{ f \in W^{1,2}_{loc}(\X) \,:\, f,|Df| \in L^2(\X,e^{-V}\mm) \}
\]
where $|Df|$ is the local minimal weak upper gradient already introduced. Since $V$ is locally bounded, $W^{1,2}(\X,e^{-V}\mm)$ turns out to coincide with the Sobolev space built over the metric measure space $(\X,\sfd,e^{-V}\mm)$, thus motivating the choice of the notation. The advantage of dealing with $L^2(\X,e^{-V}\mm)$ and $W^{1,2}(\X,e^{-V}\mm)$ is the fact that they are Hilbert spaces, unlike $L^2_{loc}(\X)$ and $W^{1,2}_{loc}(\X)$.

\medskip

After this premise, let us begin with a couple of quantitative estimates for $f_t^\eps$, $g_t^\eps$ and $\rho_t^\eps$.

\begin{Lemma}
With the same assumptions and notation as in Setting \ref{set} and defining
\begin{equation}\label{eq:vV}
v_s := \inf_{y \in \supp(\rho_0) \cup \supp(\rho_1)} \mm(B_{\sqrt{s}}(y)) \qquad V_s := \sup_{y \in \supp(\rho_0) \cup \supp(\rho_1)} \mm(B_{\sqrt{s}}(y)),
\end{equation}
for any $\bar{x} \in \X$ there exist positive constants $C_1,...,C_{9}$ depending on $K$, $N$, $\rho_0$, $\rho_1$, $\bar{x}$ only such that the following bounds hold:
\begin{itemize}
\item[i)] For any $\eps > 0$ and $t \in (0,1]$ we have
\begin{equation}\label{eq:weightbound}
\!\!\!\!\!\frac{C_1}{V_{\eps t/2}}\|f^\eps\|_{L^1(\mm)}\exp\Big(-\frac{C_2\sfd^2(\cdot,\bar{x})}{\eps t} -\frac{ C_3}{\eps t}\Big) \leq f^\eps_t \leq \frac{C_4}{v_{\eps t/2}}\|f^\eps\|_{L^1(\mm)}\exp\Big(-\frac{C_5\sfd^2(\cdot,\bar{x})}{\eps t} + \frac{ C_6}{\eps t}\Big)
\end{equation}
and analogously for $g_t^\eps$ and $t\in[0,1)$.
\item[ii)] For any $\eps\in(0,1)$ and $t\in[0,1]$ we have
\begin{equation}
\label{eq:2}
\rho^\eps_t \leq \frac{C_7}{\eps^{N/2} } \exp\bigg(\frac{C_8-C_9\sfd^2(\cdot,\bar{x})}{\eps}\bigg).
\end{equation}
\end{itemize}
\end{Lemma}
\begin{proof}

\noindent{\bf{(i)}} Direct consequence of the representation formula \eqref{eq:rapprform}, the Gaussian estimates \eqref{eq:gaussest} and the fact that $\rho_0$ and $f^\eps$ have the same support

\noindent{\bf{(ii)}} We shall indicate by $C$ a constant depending only on $K$, $N$, $\rho_0$, $\rho_1$, $\bar{x}$ whose value might change in the various instances it appears. Start from  $\rho^\eps_t=f^\eps_tg^\eps_t\leq f^\eps_t\|g^\eps_t\|_{L^\infty}\leq f^\eps_t\|g^\eps\|_{L^\infty}$, then use the bounds \eqref{eq:weightbound}, \eqref{eq:l1lifg} and notice that the constant $c$ appearing in \eqref{eq:l1lifg} is $\geq \frac{C}{V_{\eps/2}}e^{-\frac C\eps}$  to obtain
\[
\rho^\eps_t\leq \frac{CV_{\eps/2}}{v_{\eps t/2}}\exp\Big(\frac {C-C\sfd^2(\cdot,\bar x)}{\eps t}\Big).
\]
Noticing that the  the Bishop-Gromov inequality \eqref{eq:bishop} ensures that for every $s\in[0,1]$ it holds $V_s\leq C\mm(B_1(\bar x))$ and $v_s\geq C\mm(B_1(\bar x))s^{N/2}$, we obtain the claim for $t\in[1/2,1]$. The case  $t\in[0,1/2]$ follows by a symmetric argument. 
\end{proof}

The following proposition collects the basic properties of the functions defined in Setting \ref{set} and the respective `PDEs' solved:

\begin{Proposition}\label{pro:7}
With the same assumptions and notation as in Setting \ref{set}, the following holds.

All the functions are well defined and for any $\eps > 0$:
\begin{itemize}
\item[a)] $f^\eps_t,g^\eps_t,\rho^\eps_t$ belong to $\testi \X$ for all $t \in \I$, where $\I$ is the respective domain of definition (for $(\rho^\eps_t)$ we pick $\I=(0,1)$);
\item[b)] $\varphi^\eps_t,\psi^\eps_t,\vartheta^\eps_t$ belong to $\testl \X$ for all $t \in \I$, where $\I$ is the respective domain of definition.
\end{itemize}
For any $\eps>0$, $\CC \subset \I$ compact and $\bar{x} \in \X$ there exists $M = M(K,N,\rho_0,\rho_1,\CC,\bar{x}) > 0$ such that all the curves $(f^\eps_t),(g^\eps_t),(\rho^\eps_t)$ belong to $AC(\CC,W^{1,2}(\X))$ and $(\varphi_t^\eps),(\psi_t^\eps),(\vartheta_t^\eps)$ to $AC(\CC,W^{1,2}(\X,e^{-V}\mm))$, where $\I$ is the respective domain of definition (for $(\rho^\eps_t)$ we pick $\I=(0,1)$) and $V = M\sfd^2(\cdot,\bar{x})$; their time derivatives are given by the following expressions for a.e.\ $t\in[0,1]$:
\begin{align*}
\frac{\d}{\d t}f^{\varepsilon}_t & = \frac{\varepsilon}{2}\Delta f^{\varepsilon}_t && \qquad & \frac{\d}{\d t}g^{\varepsilon}_t & = -\frac{\varepsilon}{2}\Delta g^{\varepsilon}_t\\
\frac{\d}{\d t}\varphi^{\varepsilon}_t & =  \frac{1}{2}|\nabla\varphi^{\varepsilon}_t|^2 + \frac{\varepsilon}{2}\Delta\varphi^{\varepsilon}_t && \qquad & -\frac{\d}{\d t}\psi^{\varepsilon}_t & =  \frac{1}{2}|\nabla\psi^{\varepsilon}_t|^2 + \frac{\varepsilon}{2}\Delta\psi^{\varepsilon}_t\\
\frac{\d}{\d t}\rho^{\varepsilon}_t & + {\rm div}(\rho^{\varepsilon}_t\nabla\vartheta^{\varepsilon}_t) = 0 && \qquad & \frac{\d}{\d t}\vartheta^{\varepsilon}_t & + \frac{|\nabla\vartheta^{\varepsilon}_t|^2}{2} = -\frac{\varepsilon^2}{8}\Big(2\Delta\log\rho^{\varepsilon}_t + |\nabla\log\rho^{\varepsilon}_t|^2\Big).
\end{align*}
Moreover, for every $\eps>0$ we have:
\begin{itemize}
\item[i)]
\begin{equation}
\label{eq:boundbase}
\sup_{t\in \CC}\|h^\eps_t\|_{L^\infty(\X)}+\Lip(h^\eps_t)+\|\Delta h^\eps_t\|_{W^{1,2}(\X)} < \infty
\end{equation}
if $(h^\eps_t)$ is equal to any of $(f^\eps_t),(g^\eps_t),(\rho^\eps_t)$ and
\begin{equation}
\label{eq:boundbase2}
\sup_{t\in \CC}\|e^{-V}h^\eps_t\|_{L^\infty(\X)}+\|e^{-V}\lip(h^\eps_t)\|_{L^\infty(\X)}+\|\Delta h^\eps_t\|_{W^{1,2}(\X,e^{-V}\mm)} < \infty
\end{equation}
if $(h^\eps_t)$ is equal to any of $(\varphi^\eps_t),(\psi^\eps_t),(\vartheta^\eps_t)$; in both cases, $\CC$ is a compact subset of the respective domain of definition $\I$ (for $(\rho^\eps_t)$ we pick $\I=(0,1)$),
\item[ii)] $\mu^\eps_t \in \probt \X$ for every $t\in[0,1]$ and $(\rho^\eps_t)\in C([0,1],L^2(\X))$,
\item[iii)] we have $f^\eps_t\to f^\eps$ and $g^\eps_t\to g^\eps$ in $L^2(\X)$ as $t\downarrow0$ and $t\uparrow 1$ respectively.
\end{itemize}
\end{Proposition}

\begin{proof}

\noindent{\bf Properties of $(f^\eps_t),(g^\eps_t)$.} Recalling \eqref{eq:regflow} we see that $f^\eps_{t_0}\in\testi \X$ for any $t_0>0$. Then the maximum principle for the heat flow, the fact that it is a contraction in $W^{1,2}(\X)$ and the Bakry-\'Emery gradient estimates \eqref{eq:be} together with the Sobolev-to-Lipschitz property grant that \eqref{eq:boundbase} holds for $(f_t^\eps)$. The fact that $(f^\eps_t)\in AC(\mathcal C,W^{1,2}(\X))$ and that it solves the stated scaled heat equation is trivial. The fact that $f^\eps_t\to f^\eps$ in $L^2(\X)$ as $t\downarrow 0$ follows from  the $L^2$-continuity of the heat flow.

\noindent{\bf Properties of $(\varphi^\eps_t),(\psi^\eps_t)$.} By Proposition \ref{pro:testloc} we know that $\varphi^\eps_t\in\testl \X$ and from the chain and Leibniz rules we see that
\[
\begin{split}
\nabla\varphi^\eps_t &= \eps\frac{\nabla f^\eps_t}{f^\eps_t}, \\
\Delta\varphi^\eps_t &= \eps\bigg(\frac{\Delta f^\eps_t}{f^\eps_t} - \frac{|\nabla f^\eps_t|^2}{(f^\eps_t)^2}\bigg),\\
 \nabla\Delta\varphi^\eps_t&=\eps\bigg( \frac{\nabla\Delta f^\eps_t}{f^\eps_t}- \frac{\Delta f^\eps_t \nabla f^\eps_t}{(f^\eps_t)^2}- \frac{\nabla |\nabla f^\eps_t|^2}{(f^\eps_t)^2} + \frac{2|\nabla f^\eps_t|^2\nabla f^\eps_t}{(f^\eps_t)^3}\bigg).
\end{split}
\]
These identities, \eqref{eq:boundbase} for $(f^\eps_t)$, estimate \eqref{eq:reggrad} and \eqref{eq:weightbound} imply that for any $\bar x$ there is $M>0$ such that for $V:=M\sfd^2(\cdot,\bar x)$ the bound \eqref{eq:boundbase2} for $(\varphi^\eps_t)$ holds, as claimed. Similarly, we see that $|\nabla\varphi^\eps_t|^2\in L^2_{loc}((0,1],W^{1,2}(\X,e^{-V}\mm))$.

The expressions for $\nabla\varphi^\eps_t,\Delta\varphi^\eps_t$ and the equation for $(f^\eps_t)$ also grant that $\mm$-a.e.\ it holds
\begin{equation}
\label{eq:pdefi}
\frac{\d}{\d t}\varphi^{\varepsilon}_t  =  \frac{1}{2}|\nabla\varphi^{\varepsilon}_t|^2 + \frac{\varepsilon}{2}\Delta\varphi^{\varepsilon}_t 
\end{equation}
for a.e.\ $t$ and since we have seen that the right-hand side belongs to $L^2_{loc}((0,1],W^{1,2}(\X,e^{-V}\mm))$, this shows at once that $(\varphi^\eps_t)\in AC_{loc}((0,1],W^{1,2}(\X,e^{-V}\mm))$ and that \eqref{eq:pdefi} holds when the left-hand side is intended as limit of the difference quotients in $W^{1,2}(\X,e^{-V}\mm))$, as claimed.

The same arguments apply to $\psi^\eps_t$.

\noindent{\bf Properties of $(\rho^\eps_t),(\vartheta^\eps_t)$.}  The bound \eqref{eq:boundbase} for $(f^\eps_t),(g^\eps_t)$ and the Leibniz rules for the gradient and Laplacian give the bound \eqref{eq:boundbase} for $(\rho^\eps_t)$ and also grant that $(\rho^\eps_t)\in AC_{loc}((0,1),L^{2}(\X))$. To see that this curve is absolutely continuous with values in $W^{1,2}(\X)$ notice that 
\[
\frac\d{\d t}\rho^\eps_t  = \frac{\eps}{2} \Big(g^\eps_t\Delta f^\eps_t - f^\eps_t \Delta g^\eps_t\Big)
\]
and recall \eqref{eq:boundbase} for $f^\eps_t,g^\eps_t$. The stated equation for $(\rho^\eps_t)$ is now a matter of direct computation:
\[
\begin{split}
\frac{\eps}{2} \Big(g^\eps_t\Delta f^\eps_t - f^\eps_t \Delta g^\eps_t\Big) &{=} \frac{\eps}{2}\rho_t^\eps\Big(\Delta\log f_t^\eps + |\nabla\log f_t^\eps|^2 - \Delta\log g_t^\eps - |\nabla\log g_t^\eps|^2\Big) \\ & = \rho^\eps_t\frac1\eps\Big(\frac{|\nabla\varphi^\eps_t|^2}{2}-\frac{|\nabla\psi^\eps_t|^2}{2}+\frac\eps2\Delta\varphi^\eps_t-\frac\eps2\Delta\psi^\eps_t\Big)\\
& = \rho^\eps_t\Big(-\la\nabla\vartheta^\eps_t,\nabla\log\rho^\eps_t\ra-\Delta\vartheta^\eps_t\Big) \\ & = -\la\nabla\vartheta^\eps_t,\nabla\rho^\eps_t\ra-\rho^\eps_t\Delta\vartheta^\eps_t=-{\rm div}(\rho^\eps_t\nabla\vartheta^\eps_t).
\end{split}
\]

It is clear that $\rho^\eps_t\geq 0$ for every $\eps,t$, hence the identity
\[
\int\rho^\eps_t\,\d\mm = \int \h_{\varepsilon t/2}f^\eps\h_{\varepsilon(1-t)/2}g^{\varepsilon}\,\d\mm=\int f^\eps\h_{\eps/2} g^\eps\,\d\mm=\int \rho^\eps_0\,\d\mm=1
\]
shows that $\mu^\eps_t\in\prob \X$. The fact that $\mu^\eps_t$ has finite second moment is a direct consequence of the Gaussian bound \eqref{eq:2}  and the volume growth estimate \eqref{eq:volgrowth}.

For the $L^2$-continuity of $\rho^\eps_t$ in $t=0,1$, by the $L^2$-continuity of the heat flow and the fact that $f^\eps,g^\eps\in L^\infty$ (Theorem \ref{thm:5}) we see that $\rho^\eps_t\to f^\eps \h_{\eps/2}g^\eps$ and  $\rho^\eps_t\to \h_{\eps/2}f^\eps g^\eps$ as $t\to0,1$ respectively. Hence all we have to check is that 
\begin{equation}
\label{eq:contin0}
\rho_0 = f^\eps\h_{\eps/2}g^\eps\qquad\qquad\qquad\rho_1 = g^\eps\h_{\eps/2}f^\eps,
\end{equation}
but as already noticed in the proof of Theorem \ref{thm:5}, these are equivalent to the fact that $f^\eps\otimes g^\eps\,\hR^{\eps/2}$ is a transport plan from $\mu_0$ to $\mu_1$; hence, \eqref{eq:contin0} holds by the very choice of $(f^\eps,g^\eps)$ made.

Finally, the fact that $(\vartheta^\eps_t)$ belongs to $AC_{loc}((0,1),W^{1,2}(\X,e^{-V}\mm))$ and satisfies the bound \eqref{eq:boundbase2} is a direct consequence of the analogous property for $(\varphi^\eps_t),(\psi^\eps_t)$. The equation for its time derivative comes by direct computation:
\[
\begin{split}
\frac\d{\d t}\vartheta^\eps_t+\frac{|\nabla\vartheta^\eps_t|^2}2&=-\frac{|\nabla\psi^\eps_t|^2}{4}-\frac{\eps }4\Delta\psi^\eps_t-\frac{|\nabla\varphi^\eps_t|^2}4-\frac{\eps}4\Delta\varphi^\eps_t+\frac{|\nabla\psi^\eps_t|^2}8+\frac{|\nabla\varphi^\eps_t|^2}8-\frac{\la\nabla\psi^\eps_t,\nabla\varphi^\eps_t\ra}4\\
&=-\frac{\eps^2}4\Delta\log\rho^\eps_t-\frac18\Big(|\nabla\psi^\eps_t|^2+|\nabla\varphi^\eps_t|^2+2\la\nabla\varphi^\eps_t,\nabla\psi^\eps_t\ra\Big)\\
&=-\frac{\varepsilon^2}{8}\Big(2\Delta\log\rho^{\varepsilon}_t + |\nabla\log\rho^{\varepsilon}_t|^2\Big),
\end{split}
\]
hence the proof is complete.
\end{proof}

Using the terminology adopted in the literature (see \cite{Leonard14}) we shall refer to:
\begin{itemize}
\item $\varphi^{\varepsilon}_t$ and $\psi^{\varepsilon}_t$ as Schr\"odinger potentials, in connection with Kantorovich ones;
\item $(\mu^{\varepsilon}_t)_{t \in [0,1]}$ as entropic interpolation, in analogy with displacement one.
\end{itemize}

\subsection{Uniform estimates for the densities and the potentials}

We start collecting information about quantities which remain bounded as $\eps\downarrow0$.

\begin{Proposition}[Locally uniform Lipschitz and Laplacian controls for the potentials]\label{pro:1}
With the same assumptions and notations as in Setting \ref{set} the following holds.

For all $\delta \in (0,1)$ and $\bar{x} \in \X$ there exists $C>0$ which only depends on $K,N,\delta,\bar{x}$ such that
\begin{subequations}
\begin{align}
\label{eq:lipcontr}
& \lip(\varphi^\eps_t) \leq C\big(1 + \sfd(\cdot,\bar{x})\big), \qquad \mm\ae \\
\label{eq:lapcontr}
& \Delta\varphi^\eps_t \geq -C\big(1 + \sfd^2(\cdot,\bar{x})\big), \qquad \mm\ae
\end{align}
\end{subequations}
for every $t \in [\delta,1]$ and $\eps \in(0,1)$.

Furthermore, for all $M > 0$ there exists $C'>0$ which only depends on $K,N,\delta,\bar{x},M$ such that
\begin{equation}\label{eq:lapcontr2}
\int |\Delta\varphi^\eps_t| e^{-M\sfd^2(\cdot,\bar{x})}\d\mm \leq C'
\end{equation}
for every $t\in[\delta,1]$ and $\eps\in(0,1)$. Analogous bounds hold for the $\psi_t^\eps$'s in the time interval $[0,1-\delta]$.
\end{Proposition}

\begin{proof}
Fix $\delta \in (0,1)$ and $\bar{x} \in \X$ as in the statement and notice that the bound \eqref{eq:ham} yields
\[
|\nabla\varphi^\eps_t| = \eps |\nabla\log\h_{\frac{\eps t}2}f^\eps| \leq C\big(1 + \sfd(\cdot,\bar{x})\big)\qquad\forall t\in[\delta,1],\ \eps\in(0,1).
\]
Thus recalling the Sobolev-to-Lipschitz property \eqref{eq:sobtolip} we obtain the bound \eqref{eq:lipcontr}. The bound \eqref{eq:lapcontr} is a restatement of \eqref{eq:liyau}. Finally, let $M > 0$ and $\nchi$ a 1-Lipschitz cut-off function with bounded support; notice that $|h| = h + 2h^-$ whence
\[
\int \nchi e^{-M\sfd^2(\cdot,\bar{x})} |\Delta\varphi^\eps_t|\,\d\mm = \int \nchi e^{-M\sfd^2(\cdot,\bar{x})} \Delta\varphi^\eps_t\,\d\mm + 2\int \nchi e^{-M\sfd^2(\cdot,\bar{x})}(\Delta\varphi^\eps_t)^-\,\d\mm.
\]
Integration by parts, the fact that $|\nabla\sfd^2(\cdot,\bar{x})| = 2\sfd(\cdot,\bar{x})$, $|\nabla\nchi|\leq 1$ and $0 \leq \nchi \leq 1$ then imply that
\[
\begin{split}
\int \nchi e^{-M\sfd^2(\cdot,\bar{x})} |\Delta\varphi^\eps_t|\,\d\mm & \leq \int e^{-M\sfd^2(\cdot,\bar{x})} |\nabla\varphi_t^\eps|\,\d\mm + 2M\int\sfd(\cdot,\bar{x})e^{-M\sfd^2(\cdot,\bar{x})}|\nabla\varphi_t^\eps|\,\d\mm \\ & \, \quad + 2\int e^{-M\sfd^2(\cdot,\bar{x})}(\Delta\varphi^\eps_t)^-\,\d\mm
\end{split}
\]
and taking into account \eqref{eq:lipcontr} and \eqref{eq:lapcontr}, the bound \eqref{eq:lapcontr2} follows.

For $\psi^\eps_t$ the argument is the same.
\end{proof}
The gradient estimates that we just obtained together with the Gaussian bounds on $f^\eps_t,g^\eps_t,\rho^\eps_t$ that we previously proved have the following direct implication, which we shall frequently use later on to justify our computations:

\begin{Lemma}\label{lem:contprov2}
With the same assumptions and notation as in Setting \ref{set}, the following holds.

For any $\eps \in(0,1)$ and $t\in(0,1)$ let  $h_t^\eps$ denote any of $\varphi_t^\eps,\psi_t^\eps,\vartheta_t^\eps,\log\rho_t^\eps$ and, for any $n\in\N$, let $H^\eps_t$ denote any of the functions. 
\begin{equation}
\label{eq:Heps}
\rho^\eps_t|\nabla h^\eps_t|^n,\quad \rho^\eps_t\log\rho^\eps_t |\nabla h^\eps_t|^n,\quad  |\nabla \rho^\eps_t||\nabla h^\eps_t|^n,\quad \Delta\rho^\eps_t|\nabla h^\eps_t|^n\quad \rho^\eps_t\la\nabla h^\eps_t,\nabla \Delta h^\eps_t\ra.
\end{equation}
Then $H^\eps_t\in L^1(\X)$ for every $\eps,t\in (0,1)$. Moreover for every $\delta\in(0,1/2)$ we have
\begin{equation}
\label{eq:unifint}
\lim_{R\to\infty}\sup_{t\in[\delta,1-\delta]}\int_{\X\setminus B_R(\bar x)}|H^\eps_t|\,\d\mm=0\qquad\forall \eps\in(0,1),\ \bar x\in\X.
\end{equation}
Finally, $(0,1)\ni t\mapsto \int H^\eps_t\,\d\mm$ is continuous.
\end{Lemma}
 \begin{proof}
 
 \noindent{\bf General considerations} We shall repeatedly use the fact that if $h_1$ has Gaussian decay and $h_2$ polynomial growth, i.e.
\[
h_1\leq c_1\exp(-c_2\sfd^2(\cdot,\bar x)),\qquad\qquad h_2\leq c_3(1+\sfd^{c_4}(\cdot,\bar x))
\] 
for some $c_1,\ldots,c_4>0$, $\bar x\in \X$, then their product $h_1h_2$ belongs to $L^1\cap L^\infty(\X)$: the $L^\infty$ bound is obvious, the one for the $L^1$-norm is a direct consequence of the volume growth \eqref{eq:volgrowth} and explicit computations.
   
For what concerns the continuity of   $(0,1)\ni t\mapsto \int H^\eps_t\,\d\mm$,  notice that Proposition \ref{pro:7} yields that all the maps $(0,1)\ni t\mapsto |\nabla h^\eps_t|\in L^2(\X,e^{-V}\mm)$ and $(0,1)\ni t\mapsto \rho^\eps_t,|\nabla\rho^\eps_t|,\Delta\rho^\eps_t\in L^2(\X)$ are continuous (for $\Delta\rho^\eps_t$ use the fact that $\Delta\rho^\eps_t=g^\eps_t\Delta f^\eps_t+f^\eps_t\Delta g^\eps_t+2\la\nabla f^\eps_t,\nabla g^\eps_t\ra$ and the continuity of $(0,1)\ni t\mapsto\Delta f^\eps_t,\Delta g^\eps_t\in L^2(\X)$). Hence all the functions  in \eqref{eq:Heps}, with the possible exception of the last one, are continuous from $(0,1)$ to $L^0(\X)$ equipped with the topology of convergence in measure on bounded sets. Therefore the continuity of  $(0,1)\ni t\mapsto \int H^\eps_t\,\d\mm$ for these maps  will follow as soon as we show that they are, locally in $t\in(0,1)$, uniformly dominated by an $L^1(\X)$ function. Given that such domination also gives \eqref{eq:unifint}, we shall focus on proving it.

Finally, we shall consider only the case   $h^\eps_t=\varphi^\eps_t$, as the estimates for  $\psi^\eps_t$ can be obtained by symmetric arguments and the ones for $\vartheta^\eps_t,\log\rho^\eps_t$ follow from the identities $\vartheta^\eps_t=\frac{\psi^\eps_t-\varphi^\eps_t}2$, $\eps\log\rho^\eps_t=\varphi^\eps_t+\psi^\eps_t$.

 \noindent{{\bf Study of $\rho^\eps_t|\nabla h^\eps_t|^n$.}} By \eqref{eq:lipcontr} we know that  $|\nabla\varphi^\eps_t|$ has linear growth locally uniform in $t\in(0,1)$; hence $|\nabla\varphi^\eps_t|^n$ has polynomial growth  locally uniform in $t\in(0,1)$. Since $\rho^\eps_t$ has Gaussian bounds by \eqref{eq:2}, we deduce that $\rho^\eps_t|\nabla h^\eps_t|^n$ is, locally in $t\in(0,1)$, uniformly dominated. 
 
 \noindent{{\bf Study of $\rho^\eps_t\log\rho^\eps_t|\nabla h^\eps_t|^n$.}}  Writing $\log\rho^\eps_t=\log f^\eps_t+\log g^\eps_t$ and using \eqref{eq:weightbound} we see that $|\log\rho^\eps_t|$ has quadratic growth locally uniform in $t\in(0,1)$. Thus the claim follows as before.
 
\noindent{{\bf Study of $|\nabla\rho^\eps_t||\nabla h^\eps_t|^n$.}}   Notice that $|\nabla\rho^\eps_t|=\rho^\eps_t|\nabla \log\rho^\eps_t|$ and observe that from  $\eps\log\rho^\eps_t=\varphi^\eps_t+\psi^\eps_t$ and \eqref{eq:lipcontr} we have that $|\nabla \log\rho^\eps_t|$ has linear growth locally uniform in $t\in(0,1)$.

\noindent{{\bf Study of $|\Delta\rho^\eps_t||\nabla h^\eps_t|^n$.}} Write 
\[
|\Delta\rho^\eps_t|\leq f^\eps_t|\Delta g^\eps_t|+g^\eps_t|\Delta f^\eps_t|+2\eps^{-2}\rho^\eps_t|\nabla\varphi^\eps_t||\nabla \psi^\eps_t|
\]
and notice that the term $\rho^\eps_t|\nabla\varphi^\eps_t||\nabla \psi^\eps_t|$ can be handled as before and that by \eqref{eq:regflow} and the maximum principle for the heat flow we have that
\begin{equation}
\label{eq:uniflap}
\text{$\Delta f^\eps_t,\Delta g^\eps_t$ are bounded in $L^\infty(\X)$ locally uniformly in $t\in(0,1)$. }
\end{equation} 
Hence the conclusion follows from the Gaussian bounds \eqref{eq:weightbound}.

 \noindent{{\bf Study of $\rho^\eps_t\la\nabla h^\eps_t,\nabla \Delta h^\eps_t\ra$.}} Notice that 
\[
\begin{split}
|\nabla\Delta\varphi_t^\eps |&  \leq \eps\Big(\frac{|\nabla\Delta f^\eps_t|}{f^\eps_t} + \frac{|\Delta f^\eps_t| |\nabla f^\eps_t|}{(f^\eps_t)^2} + \frac{|\nabla |\nabla f^\eps_t|^2|}{(f^\eps_t)^2} + \frac{2|\nabla f^\eps_t|^3}{(f^\eps_t)^3}\Big) \\ 
& \leq \frac{1}{f^\eps_t}\Big( \eps |\nabla\Delta f^\eps_t| +   |\Delta f^\eps_t| |\nabla \varphi^\eps_t|+ 2 |\nabla\varphi_t^\eps||\H{f_t^\eps}|_{\HS} + 2\eps^{-2} f^\eps_t |\nabla\varphi_t^\eps|^3 \Big)\\
\end{split}
\]
and therefore, using also  $2g^\eps_t|\nabla\varphi^\eps_t|^2|\H{f^\eps_t}|_\HS\leq g^\eps_t|\nabla\varphi_t^\eps|^4 + g^\eps_t |\H{f_t^\eps}|^2_{\HS} $ we get
\[
\begin{split}
\rho_t^\eps  |\nabla\varphi_t^\eps|&|\nabla\Delta\varphi_t^\eps |\\
& \leq  \eps g^\eps_t |\nabla\varphi^\eps_t||\nabla\Delta f^\eps_t| +    g^\eps_t |\Delta f^\eps_t| |\nabla \varphi^\eps_t|^2+   g^\eps_t|\nabla\varphi_t^\eps|^4 + g^\eps_t |\H{f_t^\eps}|^2_{\HS} +2\eps^{-2} \rho^\eps_t |\nabla\varphi_t^\eps|^4.
\end{split}
\]
The last term in the right-hand side is dominated locally uniformly in  $t\in(0,1)$ by what already proved. Similarly, the term $g^\eps_t|\nabla\varphi_t^\eps|^4$ is, locally in $t$, dominated thanks to the Gaussian bounds on $g^\eps_t$, a domination for  $g^\eps_t |\Delta f^\eps_t| |\nabla \varphi^\eps_t|^2$ then follows using \eqref{eq:uniflap}. Writing $\nabla\Delta f^\eps_t=\nabla\h_{t-\delta}\Delta f^\eps_\delta$ for any $t\geq \delta>0$ and using \eqref{eq:regflow} and the Bakry-\'Emery estimates \eqref{eq:be} we see that 
\begin{equation}
\label{eq:unif3}
\text{$|\nabla\Delta f^\eps_t|$ is, locally in $t$, uniformly bounded in $L^\infty(\X)$,}
\end{equation} 
thus a local uniform domination for $ \eps g^\eps_t |\nabla\varphi^\eps_t||\nabla\Delta f^\eps_t|$ follows. 

It remains to consider the term $g^\eps_t |\H{f_t^\eps}|^2_{\HS} $: we know from \eqref{eq:heslap} that $ |\H{f_t^\eps}|_{\HS}\in L^2(\X)$ and from \eqref{eq:weightbound} that $g^\eps_t\in L^\infty(\X)$. This is sufficient to conclude that $\rho^\eps_t\la\nabla h^\eps_t,\nabla \Delta h^\eps_t\ra\in L^1(\X)$. To prove \eqref{eq:unifint}, thanks to the dominations previously obtained, it is enough to prove that 
\begin{equation}
\label{eq:unifhess}
\lim_{R\to\infty}\int(1-\nchi_R)g^\eps_t|\H{f_t^\eps}|^2_{\HS}\,\d\mm=0\qquad\text{locally uniformly in }t\in(0,1),
\end{equation}
where for any $R>0$ the function $\nchi_R$ is a cut-off given by Lemma  \ref{lem:cutoff}. From \eqref{eq:bochhess} we have
\[
\begin{split}
\int(1-\nchi_R)g^\eps_t|\H{f_t^\eps}|^2_{\HS}\,\d\mm\leq\int& \Delta((1-\nchi_R)g^\eps_t)\frac{|\nabla f^\eps_t|^2}2\\
&+(1-\nchi_R)g^\eps_t\big(\la\nabla f^\eps_t,\nabla\Delta f^\eps_t\ra-K|\nabla f^\eps_t|^2\big)\,\d\mm.
\end{split}
\]
By \eqref{eq:unif3}, the already noticed fact that $|\nabla f^\eps_t|$ is also uniformly bounded in $L^\infty(\X)$ locally in $t\in(0,1)$ and the Gaussian bounds \eqref{eq:weightbound} on $g^\eps_t$ we see that the second addend in the last integral is, locally in $t\in(0,1)$, dominated by an $L^1(\X)$ function.

For the first addend we write
\[
\begin{split}
 \Delta((1-\nchi_R)g^\eps_t)&= -g^\eps_t\Delta \nchi_R-2\la\nabla\nchi_R,\nabla g^\eps_t\ra+(1-\nchi_R)\Delta g^\eps_t\qquad\qquad |\nabla f^\eps_t|=\eps^{-1}f^\eps_t |\nabla \varphi^\eps_t|
\\
\end{split}
\]
and use the properties of $\nchi_R$ given by Lemma \ref{lem:cutoff} and those of $g^\eps_t,f^\eps_t,|\nabla\varphi^\eps_t|$ that we already mentioned to deduce that $ \Delta((1-\nchi_R)g^\eps_t)$ is bounded in $L^\infty(\X)$ and $|\nabla f^\eps_t|^2$ dominated in $L^1(\X)$, both locally uniformly in $t\in(0,1)$. Hence \eqref{eq:unifhess} follows from the fact that  $(1-\nchi_R),|\nabla\nchi_R|,\Delta \nchi_R$ are  identically 0 on $B_R(\bar x)$.

It remains to prove that $t\mapsto\int \rho^\eps_t\la\nabla \varphi^\eps_t,\nabla \Delta \varphi^\eps_t\ra\,\d\mm $ is continuous and thanks to \eqref{eq:unifint} to this aim it is sufficient to show that for any $R>0$ the map  $t\mapsto\int \nchi_R\rho^\eps_t\la\nabla \varphi^\eps_t,\nabla \Delta \varphi^\eps_t\ra\,\d\mm $ is continuous. To see this, notice that
\begin{equation}
\label{eq:ultcont}
\int \nchi_R\rho^\eps_t\la\nabla\varphi^\eps_t,\nabla \Delta \varphi^\eps_t\ra\,\d\mm=-\int\big( \nchi_R\la\nabla \rho^\eps_t,\nabla \varphi^\eps_t\ra + \rho^\eps_t\la\nabla \nchi_R ,\nabla \varphi^\eps_t\ra \big) \Delta \varphi^\eps_t+\nchi_R\rho^\eps_t|\Delta\varphi^\eps_t|^2\,\d\mm,
\end{equation}
and that  the maps $t\mapsto \rho^\eps_t,\varphi^\eps_t$ are continuous with values in $W^{1,2}(\X),W^{1,2}(\X,e^{-V}\mm)$ respectively. Also, writing $\Delta\varphi^\eps_t=\eps\frac{\Delta f^\eps_t}{f^\eps_t}-\eps|\nabla \varphi^\eps_t|^2$, using the continuity of $t\mapsto f^\eps_t,\Delta f^\eps_t\in L^2(\X)$, the bound \eqref{eq:uniflap}, the fact that $f^\eps_t$ is bounded from below on $\supp(\nchi_R)$ by a positive constant depending continuously on $t$ also taking into account what previously proved we see that the integrand in the right-hand side of \eqref{eq:ultcont} is continuous as map with values in $L^0(\X,\mm\restr{\supp(\nchi_R)})$ and, locally in $t$, uniformly dominated by an $L^1(\X)$-function. This is sufficient to conclude.
 \end{proof}

\begin{Proposition}[Uniform $L^{\infty}$ bound on the densities]\label{pro:6}
With the same assumptions and notations as in Setting \ref{set} the following holds.

For every $\bar{x} \in \X$ there exist constants $C,C' > 0$ which depends on $K,N,\bar{x},\rho_0,\rho_1$ such that
\begin{equation}
\label{eq:tail}
\rho_t^\eps \leq C e^{-C'\sfd^2(\cdot,\bar{x})} \qquad \mm\ae
\end{equation}
for every $t \in [0,1]$ and for every $\eps \in (0,1)$.
\end{Proposition}
\begin{proof} From \eqref{eq:2} and direct manipulation we see that there are constants $c,c',r>0$ depending on $K,N,\bar{x},\rho_0,\rho_1$ only such that
\begin{equation}
\label{eq:gauss1}
\rho^\eps_t(x)\leq ce^{-c'\sfd^2(x,\bar x)}\qquad\forall x\notin B_r(\bar x),\ \eps\in(0,1),\ t\in[0,1],
\end{equation}
hence to conclude it is sufficient to show that there exists  a constant $M > 0$ depending on  $K,N,\bar{x},\rho_0,\rho_1$ only   such that
\begin{equation}
\label{eq:linftybound}
\|\rho_t^\eps\|_{L^{\infty}(\X)} \leq M\qquad\forall \eps\in(0,1),\ t\in[0,1].
\end{equation}
For later purposes it will be useful to observe that from \eqref{eq:gauss1} and the volume growth estimate \eqref{eq:volgrowth} it follows  that there is $R>r$ such that 
\begin{equation}
\label{eq:dagauss1}
\int_{\X\setminus B_R(\bar x)}(\rho^\eps_t)^p \sfd^2(\cdot,\bar{x})\d\mm\leq 1\qquad\forall \eps\in(0,1),\ t\in[0,1],\ p\geq 2.
\end{equation}
Now fix $\eps > 0$. We know from Proposition \ref{pro:7} that $(\rho_t^\eps)\in C([0,1],L^2(\X))\cap AC_{loc}((0,1),L^2(\X))$ and by the maximum principle for the heat equation $\rho_t^\eps \leq C_{\eps}$ for all $t \in [0,1]$, thus for any $p>2$ the function $E_p : [0,1] \to[0,\infty)$ defined by
\[
E_p(t) := \int (\rho^\eps_t)^p\,\d\mm,
\]
belongs to $C([0,1])\cap AC_{loc}((0,1))$. An application of the dominated convergence theorem grants that its derivative can be computed passing the limit inside the integral, obtaining
\[
\frac{\d}{\d t}E_p(t) = p\int (\rho^\eps_t)^{p-1}\frac{\d}{\d t}\rho^\eps_t\d\mm = -p\int (\rho^\eps_t)^{p-1}{\rm div}(\rho^\eps_t\nabla\vartheta^\eps_t)\,\d\mm.
\]
Then the definition of $\vartheta_t^\eps$, \eqref{eq:lipcontr},  \eqref{eq:lapcontr} and \eqref{eq:2} allow to justify the integration by parts,  whence
\[
\begin{split}
\frac{\d}{\d t}E_p(t)&=p(p-1)\int(\rho^\eps_t)^{p-1}\la\nabla\rho^\eps_t,\nabla\vartheta^\eps_t\ra\d\mm\\
&=(p-1)\int\la\nabla(\rho^\eps_t)^p,\nabla\vartheta^\eps_t\ra\d\mm = -(p-1)\int (\rho^\eps_t)^{p}\Delta\vartheta^\eps_t \,\d\mm
\end{split}
\]
and recalling that  $\vartheta^\eps_t = \psi^\eps_t - \frac{\varepsilon}{2}\log\rho^\eps_t$ we obtain (for the same reasons as above, the integrals are well defined) that
\begin{equation}\label{eq:18}
\ddt E_p(t) = -(p-1)\int(\rho_t^\eps)^p\Delta\psi^\eps_t\d\mm + \frac{\varepsilon}{2}(p-1) \int (\rho^\eps_t)^p\Delta\log\rho^\eps_t\d\mm.
\end{equation}
Now notice  (the same arguments as above  justify integration by parts) that
\[
\int (\rho^\eps_t)^p\Delta\log\rho^\eps_t\d\mm = -p\int (\rho^\eps_t)^{p-1}\la\nabla\rho^\eps_t,\nabla\log\rho^\eps_t\ra\d\mm=-p\int  (\rho^\eps_t)^{p-2}|\nabla\rho^\eps_t|^2\,\d\mm\leq 0
\]
and choose $\delta:=\frac12$ and $T:=1$ in \eqref{eq:liyau} to get the existence of a constant $c''>0$ depending on $K,N,\bar{x}$ and the diameters of the supports of $\rho_0,\rho_1$ such that $\Delta\psi^\eps_t\geq -c''(1 + \sfd^2(\cdot,\bar{x}))$ for any $t \in [0,1/2]$. Thus  from \eqref{eq:18} we have
\[
\ddt E_p(t)  \leq c''(p-1)E_p(t) + c''(p-1)\int(\rho^\eps_t)^p \sfd^2(\cdot,\bar{x})\d\mm\qquad a.e.\ t\in[0,1/2]
\]
and recalling \eqref{eq:dagauss1}  we get
\[
\ddt E_p(t)  \leq c''(p-1)E_p(t)+ c''(p-1)\int_{B_R(\bar x)}(\rho^\eps_t)^p \sfd^2(\cdot,\bar{x})\d\mm+1\leq  c'''(p-1)E_p(t)+1
\]
for a.e.\ $t\in[0,1/2]$. Then Gronwall's lemma gives
\[
E_p(t) \leq \big(E_p(0) + \frac{1}{c'''(p-1)}\big)e^{c'''(p-1)}\qquad\forall t\in[0,1/2].
\]
Passing to the $p$-th roots, writing $E_p(t) = \|\rho_t^\eps\|^{p-1}_{L^{p-1}(\mu_t^\eps)}$ and observing that, being $\mu_t^\eps$ a probability measure, we have $\|h\|_{L^p(\mu_t^\eps)}\uparrow\|h\|_{L^\infty(\mu_t^\eps)}$ as $p\to\infty$, we obtain
\[
\|\rho^\eps_t\|_{L^{\infty}} \leq e^{c'''}\|\rho_0\|_{L^{\infty}}, \quad \forall t \in [0,1/2].
\]
Switching the roles of $\rho_0$ and $\rho_1$ we get the analogous control for $t\in[\frac12,1]$, whence the claim \eqref{eq:linftybound} with $M := e^{c'''}\max\{\|\rho_0\|_{L^{\infty}},\|\rho_1\|_{L^{\infty}}\}$. 
\end{proof}

\subsection{The entropy along entropic interpolations}

In \cite{Leonard13} L\'eonard computed the first and second derivatives of the relative entropy along entropic interpolations: here we are going to show that his computations are fully justifiable in our setting. As we shall see later on, these formulas will be the crucial tool for showing that the acceleration of the entropic interpolation goes to 0 in the suitable weak sense.

We start by noticing that a form of  Bochner inequality  for the Schr\"odinger potentials can be deduced. Observe that in general the object $\Ggamma_2(\varphi^\eps_t)$ is not a well defined measure, because in some sense it can have both infinite positive mass and infinite negative mass, nevertheless, thanks to Lemma \ref{lem:contprov2}, its action on $\rho^\eps_t$ can still be defined: we will put
\[
\Big\langle \Gamma_2(h^\eps_t),\rho^\eps_t \Big\rangle:=\int\Big(\frac{1}{2}\Delta\rho_t^\eps |\nabla h_t^\eps|^2 - \rho_t^\eps\langle\nabla h_t^\eps,\nabla\Delta h_t^\eps \rangle\Big)\d\mm,
\]
where $h_t^\eps$ is equal to any of $\varphi_t^\eps,\psi_t^\eps,\vartheta_t^\eps,\log\rho_t^\eps$, and notice that  Lemma \ref{lem:contprov2} ensures that the integral in the right-hand side is well defined and finite. We then have:
\begin{Lemma}\label{lem:gamma2}
With the same assumptions and notations as in Setting \ref{set}, for any $\eps > 0$ and $t \in (0,1)$ we have
\begin{subequations}
\begin{align}
\label{eq:wbochhess}
\Big\langle \Gamma_2(h^\eps_t),\rho^\eps_t \Big\rangle & \geq \int \big(|\H{h_t^\eps}|^2_\HS + K|\nabla h_t^\eps|^2\big)\rho_t^\eps\d\mm \\
\label{eq:wbochlap}
\Big\langle \Gamma_2(h^\eps_t),\rho^\eps_t \Big\rangle& \geq \int \Big( \frac{(\Delta h_t^\eps)^2}N + K|\nabla h_t^\eps|^2\Big)\rho_t^\eps\d\mm
\end{align}
\end{subequations}
where $h_t^\eps$ is equal to any of $\varphi_t^\eps,\psi_t^\eps,\vartheta_t^\eps,\log\rho_t^\eps$.
\end{Lemma}

\begin{proof} 
Fix $\eps > 0$, $t \in (0,1)$ and, for given $\bar{x} \in \X$ and $R>0$, let $\nchi_R \in \testi\X$ be a cut-off function with support in $B_{R+1}(\bar{x})$ and such that $\nchi_R \equiv 1$ in $B_R(\bar{x})$. Then we know that   ${\nchi}_{R+1}\varphi_t^\eps \in \testi\X$ and thus \eqref{eq:bochhess} holds for it, namely
\[
\Ggamma_2({\nchi}_{R+1}\varphi_t^\eps) \geq \big(|\H{{\nchi}_{R+1}\varphi_t^\eps}|^2_\HS + K|\nabla({\nchi}_{R+1}\varphi_t^\eps)|^2\big)\mm.
\]
Multiplying both sides of the inequality by $\nchi_R\rho_t^\eps$, integrating over $\X$ and using the locality of the various differential operators appearing we obtain
\begin{equation}\label{eq:intbochner}
\int\Big(\frac{1}{2}\Delta(\nchi_R\rho_t^\eps) |\nabla\varphi_t^\eps|^2 - \nchi_R\rho_t^\eps\langle\nabla\varphi_t^\eps,\nabla\Delta\varphi_t^\eps\rangle\Big)\d\mm \geq \int \nchi_R\rho_t^\eps \big(|\H{\varphi_t^\eps}|^2_\HS + K|\nabla\varphi_t^\eps|^2\big)\d\mm.
\end{equation}
By monotone convergence we have that
\[
\begin{split}
\lim_{R \to \infty}\int \nchi_R\rho_t^\eps |\nabla\varphi_t^\eps|^2\d\mm& = \int \rho_t^\eps |\nabla\varphi_t^\eps|^2\d\mm,\\
\lim_{R \to \infty}\int \nchi_R\rho_t^\eps |\H{\varphi_t^\eps}|^2_\HS\d\mm& = \int \rho_t^\eps |\H{\varphi_t^\eps}|^2_\HS\d\mm.
\end{split}
\]
and thus the right-hand side of \eqref{eq:intbochner} goes to the right-hand side of \eqref{eq:wbochhess}. Now notice that
\[
\Delta(\nchi_R\rho^\eps_t)=\nchi_R\Delta \rho^\eps_t+2\la\nabla\rho^\eps_t,\nabla\nchi_R\ra+\rho^\eps_t\Delta\nchi_R
\]
and that the choice of $\nchi_R$ grants that $|\nchi_R|,|\nabla\nchi_R|,|\Delta\nchi_R|$ are uniformly bounded and $\mm$-a.e.\ converge to $1,0,0$ respectively as $R\to\infty$. Hence Lemma \ref{lem:contprov2} and the dominated convergence theorem give that the left-hand side of \eqref{eq:intbochner} goes to $\big\langle \Gamma_2(h^\eps_t),\rho^\eps_t \big\rangle$, thus settling the proof of \eqref{eq:wbochhess} for $h^\eps_t=\varphi^\eps_t$. The other claims follow by similar means taking \eqref{eq:bochlap} into account.
\end{proof}

Now we are in  position for motivating L\'eonard's computations, thus getting the formulas for the first and second derivative of the entropy along entropic interpolations.

\begin{Proposition}\label{pro:5}
With the same assumptions and notations as in Setting \ref{set} the following holds.

For any $\eps>0$ the map $t\mapsto H(\mu^{\varepsilon}_t \,|\, \mm)$ belongs to $C([0,1])\cap C^2(0,1)$ and for every $t\in(0,1)$ it holds
\begin{subequations}
\begin{align}
\label{eq:firstder}
\frac{\d}{\d t}H(\mu^{\varepsilon}_t \,|\, \mm) &= \int\la\nabla \rho^\eps_t,\nabla\vartheta^\eps_t\ra\,\d\mm=\frac1{2\eps}\int\big(|\nabla\psi^\eps_t|^2-|\nabla\varphi^\eps_t|^2\big)\rho^\eps_t\,\d\mm,\\
\label{eq:secondder}
\frac{\d^2}{\d t^2}H(\mu^{\varepsilon}_t \,|\, \mm) & = \Big\langle\Gamma_2(\vartheta^{\varepsilon}_t),\rho^\eps_t\Big\rangle+\frac{\eps^2}4\Big\langle\Gamma_2(\log\rho^{\varepsilon}_t),\rho^\eps_t\Big\rangle= \frac{1}{2}\Big\langle\Gamma_2(\varphi^{\varepsilon}_t),\rho^\eps_t\Big\rangle+ \frac{1}{2}\Big\langle\Gamma_2(\psi^{\varepsilon}_t),\rho^\eps_t\Big\rangle.
\end{align}
\end{subequations}
\end{Proposition}
\begin{proof} By Lemma \ref{lem:contprov2} we know that the central and right-hand sides in \eqref{eq:firstder} and \eqref{eq:secondder} exist, are  finite and continuously depend on $t\in(0,1)$. Also, the equality between the central and right-hand sides follows trivially from the relations $\vartheta^\eps_t=\frac{\psi^\eps_t-\varphi^\eps_t}2$ and $\eps\log\rho^\eps_t=\varphi^\eps_t+\psi^\eps_t$.

Thus to conclude it is sufficient to show that $t\mapsto H(\mu^{\varepsilon}_t \,|\, \mm)$ is in $C([0,1])$ and that the identities  \eqref{eq:firstder} and \eqref{eq:secondder} hold for a.e.\ $t\in(0,1)$.

Lemma \ref{lem:contprov2} ensures that $t\mapsto H(\mu^{\varepsilon}_t \,|\, \mm)$ is continuous in $(0,1)$. To check continuity in $t=0,1$, thanks to the fact that $(\rho^\eps_t)\in C([0,1],L^2(\X))$ by Proposition \ref{pro:7} and arguing as in the proof of Lemma  \ref{lem:contprov2}, it is sufficient to show that $\rho^\eps_t\log\rho^\eps_t$ is dominated by an $L^1(\X)$ function. To see this, write
\[
\rho^\eps_t\log\rho^\eps_t=g^\eps_tf^\eps_t\log f^\eps_t+f^\eps_tg^\eps_t\log g^\eps_t
\]
and notice that for $t\in[0,1/2]$ the bound \eqref{eq:weightbound} ensures that the function $g^\eps_t$ is uniformly bounded above by a Gaussian and that $\log g^\eps_t$ has a quadratic growth. On the other hand, we know by Theorem \ref{thm:5} that $f^\eps_0=f^\eps$ is in $L^\infty$, thus the maximum principle for the heat flow and the fact that $z\mapsto z\log z$ is bounded from below give that the $L^\infty$ norms of $f^\eps_t,f^\eps_t\log f^\eps_t$ are uniformly bounded in $t\in[0,1/2]$. As discussed in the proof of Lemma \ref{lem:contprov2}, this is sufficient to conclude and a similar arguments yields the desired bound for $t\in[1/2,1]$.

Now fix $\eps>0$ and for $R>0$ let $\nchi_R \in \testi\X$ be a cut-off function as given by Lemma \ref{lem:cutoff}. Notice that Lemma \ref{lem:contprov2} grants that
\begin{equation}
\label{eq:conventr}
\int\nchi_R\rho^\eps_t\log\rho^\eps_t\,\d\mm\quad\to\quad\int \rho_t^\eps\log\rho^\eps_t\,\d\mm\qquad\text{as $R\to\infty$ for every  $t\in(0,1)$.}
\end{equation}
Also,  Proposition \ref{pro:7} tells that $(\rho^\eps_t)\in AC_{loc}((0,1),L^2(\X))$ and that it is, locally in $t\in(0,1)$ and in space, uniformly bounded away from 0 and $\infty$. Therefore, for  $u(z):=z\log z$ we have that $(0,1)\ni t\mapsto \nchi_R u(\rho^\eps_t)\in L^{2}(\X)$ is absolutely continuous. In particular, so is $\int \nchi_R u(\rho^\eps_t)\,\d\mm$ and it is then clear that
\[
\ddt \int \nchi_R u(\rho^\eps_t)\d\mm = \int \nchi_R(\log(\rho^\eps_t)+1)\ddt\rho_t^\eps\,\d\mm,\qquad{\rm a.e.}\ t.
\]
Using the formula for $\ddt\rho^\eps_t$ provided by Proposition \ref{pro:7} we then get
\[
\begin{split}
\ddt \int \nchi_R u(\rho^\eps_t)\,\d\mm & = -\int \nchi_R(\log(\rho^\eps_t)+1){\rm div}(\rho_t^\eps \nabla\vartheta^\eps_t)\d\mm = \int\langle\nabla(\nchi_R(\log(\rho^\eps_t)+1),\nabla \vartheta^\eps_t \rangle \rho^\eps_t\d\mm \\ & = \int\nchi_R\la\nabla\rho^\eps_t,\nabla \vartheta^\eps_t \ra\d\mm + \int\la\nabla\nchi_R, \nabla\vartheta_t^\eps\ra(\log\rho_t^\eps + 1)\rho_t^\eps\d\mm.
\end{split}
\]
Since $|\nabla\nchi_R|$ is uniformly bounded and identically 0 on $B_R(\bar x)$, Lemma \ref{lem:contprov2} grants that the last expression in the above identity converges to $\int\la\nabla\rho^\eps_t,\nabla\vartheta^\eps_t\ra\,\d\mm$ as $R\to\infty$ locally uniformly in $t\in(0,1)$. This fact, \eqref{eq:conventr} and the initial discussion give $C^1((0,1))$ regularity for $t\mapsto H(\mu^{\varepsilon}_t \,|\, \mm)$ and  \eqref{eq:firstder}.

For \eqref{eq:secondder}, notice that from Proposition \ref{pro:7} we know that $(\rho_t^\eps) \in AC_{loc}((0,1),W^{1,2}(\X))$ and $(\vartheta_t^\eps) \in AC_{loc}((0,1),W^{1,2}(\X,e^{-V}\mm))$ with $V = M\sfd^2(\cdot,\bar{x})$, for some $\bar{x} \in \X$ and $M > 0$ sufficiently large. Hence $(0,1) \ni t \mapsto \nchi_R \la\nabla\rho_t^\eps,\nabla\vartheta_t^\eps\ra \in L^2(\X)$ is absolutely continuous. In particular, so is $\int\nchi_R\la\nabla\rho_t^\eps,\nabla\vartheta_t^\eps\ra\d\mm$ and
\[
\ddt\int\nchi_R \langle\nabla\rho_t^\eps,\nabla\vartheta_t^\eps\rangle\d\mm = \int \nchi_R \Big(\langle\nabla\ddt \rho^\eps_t,\nabla\vartheta^\eps_t\rangle + \langle\nabla \rho^\eps_t,\nabla\ddt\vartheta^\eps_t\rangle\Big)\d\mm, \qquad{\rm a.e.}\ t.
\]
Thus from the formulas for $\ddt\rho^\eps_t,\ddt\vartheta^\eps_t$ provided in Proposition \ref{pro:7}  we obtain
\[
\begin{split}
\ddt\int\nchi_R \langle\nabla\rho_t^\eps,\nabla\vartheta_t^\eps\rangle\d\mm =&\underbrace{\int -\nchi_R\la\nabla({\rm div}(\rho^\eps_t\nabla\vartheta^\eps_t)),\nabla\vartheta^\eps_t\ra\d\mm}_{A_t(R)} \\
&\quad+ \underbrace{\int\nchi_R\langle\nabla\rho^\eps_t,\nabla\big(-\tfrac12|\nabla\vartheta^\eps_t|^2-\tfrac{\eps^2}{4}\Delta\log(\rho^\eps_t)-\tfrac{\eps^2}8|\nabla\log(\rho^\eps_t)|^2\big)\rangle\d\mm}_{B_t(R)}.
\end{split}
\]
Now notice that a few integration by parts and the Leibniz rule give
\[
\begin{split}
A_t(R) = & \int{\rm div}(\rho_t^\eps\nabla\vartheta_t^\eps)\langle\nabla\nchi_R,\nabla\vartheta_t^\eps \rangle\d\mm + \int \nchi_R {\rm div}(\rho_t^\eps\nabla\vartheta_t^\eps)\Delta\vartheta_t^\eps \d\mm \\
= & \int\langle\nabla\rho_t^\eps,\nabla\vartheta_t^\eps\rangle\langle\nabla\nchi_R,\nabla\vartheta_t^\eps\rangle\d\mm -\int \nchi_R\rho_t^\eps\langle\nabla\vartheta_t^\eps,\nabla\Delta\vartheta_t^\eps\rangle \d\mm
\end{split}
\]
and
\[
\begin{split}
B_t(R) = & \int\frac{1}{2}|\nabla\vartheta_t^\eps|^2 {\rm div}(\nchi_R\nabla\rho_t^\eps) - \frac{\eps^2}{4}\nchi_R \langle\nabla\rho_t^\eps,\nabla\Delta\log\rho_t^\eps\rangle + \frac{\eps^2}{8}|\nabla\log\rho_t^\eps|^2 {\rm div}(\nchi_R\nabla\rho_t^\eps) \d\mm  \\ 
= & \int \frac{1}{2}\nchi_R\Delta\rho_t^\eps |\nabla\vartheta_t^\eps|^2 - \frac{\eps^2}{4}\nchi_R\rho_t^\eps \langle\nabla\log\rho_t^\eps,\nabla\Delta\log\rho_t^\eps\rangle + \frac{\eps^2}{8}\nchi_R\Delta\rho_t^\eps|\nabla\log\rho_t^\eps|^2 \d\mm \\ 
& \qquad + \int\frac{1}{2}|\nabla\vartheta_t^\eps|^2 \langle \nabla\nchi_R,\nabla\rho_t^\eps \rangle + \frac{\eps^2}{8}|\nabla\log\rho_t^\eps|^2 \langle\nabla\nchi_R,\nabla\rho_t^\eps\rangle \d\mm.
\end{split}
\]
Since $|\nabla\nchi_R|$ is uniformly bounded and identically 0 on $B_R(\bar x)$, Lemma \ref{lem:contprov2} gives that
\[
A_t(R)+B_t(R)\quad\to\quad \Big\langle\Gamma_2(\vartheta^{\varepsilon}_t),\rho^\eps_t\Big\rangle+\tfrac{\eps^2}4\Big\langle\Gamma_2(\log\rho^{\varepsilon}_t),\rho^\eps_t\Big\rangle\qquad\text{ as $R\to\infty$}
\] 
locally uniformly in $t\in(0,1)$. 

This fact, the convergence of $\int \nchi_R \langle\nabla\rho_t^\eps,\nabla\vartheta_t^\eps\rangle\d\mm $ to $\int \langle\nabla\rho_t^\eps,\nabla\vartheta_t^\eps\rangle\d\mm $ as $R\to\infty$ (which is also consequence of Lemma \ref{lem:contprov2}) and the initial discussion  give the conclusion.
\end{proof}

As a first consequence of the formulas just obtained, we show that some quantities remain bounded as $\eps\downarrow0$:

\begin{Lemma}[Bounded quantities]\label{lem:7}
With the same assumptions and notations of Setting \ref{set}, for any $\bar{x} \in \X$ we have
\begin{subequations}
\begin{align}
\label{eq:boundmomento}
& \sup_{\eps \in (0,1),\, t \in [0,1]}\int \sfd^2(\cdot,\bar{x})\rho^\eps_t\d\mm < \infty, \\
\label{eq:boundentropia}
& \sup_{\eps \in (0,1),\, t \in [0,1]}|H(\mu^\eps_t \,|\, \mm)| < \infty, \\
\label{eq:boundenergia}
& \sup_{\eps\in(0,1)}\iint_0^1\Big(|\nabla\vartheta^\eps_t|^2+{\eps^2}|\nabla\log\rho^\eps_t|^2\Big)\rho^\eps_t\,\d t\,\d\mm < \infty,
\end{align}
\end{subequations}
and for any $\delta\in(0,\frac12)$
\begin{subequations}
\begin{align}
\label{eq:bhess}
& \sup_{\eps \in (0,1)} \iint_\delta^{1-\delta}\Big(|\H{\vartheta^\eps_t}|_\HS^2+  {\eps^2} |\H{\log\rho_t^{\varepsilon}}|^2_{\HS}\Big)\rho^{\varepsilon}_t\,\d t\,\d\mm < \infty, \\
\label{eq:blap}
&\sup_{\eps \in (0,1)} \iint_\delta^{1-\delta}\Big(|\Delta{\vartheta^\eps_t}|^2+ {\eps^2} |\Delta{\log\rho_t^{\varepsilon}}|^2\Big)\rho^{\varepsilon}_t\,\d t\,\d\mm < \infty.  
\end{align}
\end{subequations}
\end{Lemma}

\begin{proof} \eqref{eq:boundmomento}   follows from \eqref{eq:tail} and the volume growth \eqref{eq:volgrowth}.  As regards \eqref{eq:boundentropia}, notice that \eqref{eq:boundmomento} and \eqref{eq:entdef2} give a uniform lower bound on  $H(\mu_t^\eps\,|\,\mm)$; for the upper bound notice that  \eqref{eq:tail} implies a uniform quadratic growth on $\log\rho^\eps_t$.

Let us now pass to \eqref{eq:boundenergia} and observe that Proposition \ref{pro:1} together with \eqref{eq:boundmomento} grants
\begin{equation}
\label{eq:partefacile}
\sup_{\eps\in(0,1)}\iint_{\frac12}^1|\nabla\varphi^\eps_t|^2\rho^\eps_t\,\d t\,\d\mm+\iint_0^{\frac12}|\nabla\psi^\eps_t|^2\rho^\eps_t\,\d t\,\d\mm<\infty.
\end{equation}
As a second step, notice that \eqref{eq:firstder} gives
\[
\begin{split}
\iint_0^{\frac12}|\nabla\varphi^\eps_t|^2\rho^\eps_t\,\d t\,\d\mm&=\iint_0^{\frac12}|\nabla\psi^\eps_t|^2\rho^\eps_t\,\d t\,\d\mm-2\eps\int_0^{\frac12}\ddt H(\mu^\eps_t\,|\,\mm)\,\d t\\
&=\iint_0^{\frac12}|\nabla\psi^\eps_t|^2\rho^\eps_t\,\d t\,\d\mm+2\eps\big(H(\mu_0\,|\,\mm)-H(\mu^\eps_t\,|\,\mm)\big)
\end{split}
\]
so that taking into account \eqref{eq:boundentropia} and \eqref{eq:partefacile} we see that the right-hand side is uniformly bounded for $\eps\in(0,1)$. Using again \eqref{eq:partefacile} we deduce that
\[
\sup_{\eps\in(0,1)}\iint_0^{1}|\nabla\varphi^\eps_t|^2\rho^\eps_t\,\d t\,\d\mm<\infty.
\]
A symmetric argument provides the analogous bound for $(\psi^\eps_t)$ and thus recalling that $\vartheta^\eps_t=\frac12(\psi^\eps_t-\varphi^\eps_t)$ and $\eps\log\rho^\eps_t=\psi^\eps_t+\varphi^\eps_t$ we obtain \eqref{eq:boundenergia}.

Now use the fact that $\vartheta^{\varepsilon}_t = -\varphi^{\varepsilon}_t + \frac{\varepsilon}{2}\log\rho^{\varepsilon}_t$ in conjunction with \eqref{eq:firstder} to get
\[
\begin{split}
\ddt H(\mu^\eps_t\,|\,\mm)\restr{t=\delta}&=- \int  \langle \nabla\rho_{\delta}^\eps, \nabla\varphi_{\delta}^{\eps} \rangle \d\mm + \frac{\eps}{2} \int \langle \nabla\rho_{\delta}^\eps, \nabla\log\rho_{\delta}^{\eps} \rangle \d\mm\\
&=\int  \rho^\eps_{\delta} \Delta\varphi_{\delta}^{\eps}  \d\mm + \frac{\eps}{2} \int\frac{| \nabla\rho_{\delta}^\eps|^2}{\rho^\eps_\delta} \d\mm\geq \int  \rho_{\delta}^\eps \Delta\varphi_{\delta}^{\varepsilon} \d\mm.
\end{split}
\]
Recalling the lower bound \eqref{eq:lapcontr} and \eqref{eq:boundmomento}, we get that for some constant $C_\delta$ independent on $\eps$ it holds
\[
\ddt H(\mu^\eps_t\,|\,\mm)\restr{t=\delta}\geq -C_\delta\qquad\forall \eps\in(0,1)
\]
and an analogous argument starting from  $\vartheta^\eps_t=\psi^\eps_t-\frac\eps2\log\rho^\eps_t$  yields
$
\ddt H(\mu^\eps_t\,|\,\mm)\restr{t=1-\delta}\leq C_\delta$ for every $\eps\in(0,1)$.
Therefore
\[
\sup_{\eps\in(0,1)}\int_\delta^{1-\delta}\frac{\d^2}{\d t^2} H(\mu^\eps_t\,|\,\mm)=\sup_{\eps\in(0,1)}\bigg(\ddt H(\mu^\eps_t\,|\,\mm)\restr{t=1-\delta}-\ddt H(\mu^\eps_t\,|\,\mm)\restr{t=\delta}\bigg)<\infty.
\]
The bounds \eqref{eq:bhess} and \eqref{eq:blap} then come from this last inequality used in conjunction with \eqref{eq:secondder}, \eqref{eq:boundenergia} and the weighted Bochner inequalities \eqref{eq:wbochhess} and \eqref{eq:wbochlap} respectively.
\end{proof}

With the help of the previous lemma we can now prove that some crucial quantities vanish in the limit $\eps\downarrow0$; as we shall see in the proof of our main Theorem \ref{thm:main}, this is what we will need to prove that the acceleration of the entropic interpolations goes to 0 as $\eps$ goes to zero.

\begin{Lemma}[Vanishing quantities]\label{lem:vanish}
With the same assumptions and notations of Setting \ref{set}, for any $\delta\in(0,\frac12)$ we have
\begin{subequations}
\begin{align}
\label{eq:b1}
&\lim_{\varepsilon \downarrow 0}\varepsilon^2 \iint_\delta^{1-\delta} \rho^{\varepsilon}_t |\Delta\log\rho^{\varepsilon}_t|  \,\d t\,  \d\mm= 0, \\
\label{eq:b2}
&\lim_{\varepsilon \downarrow 0}\varepsilon^2 \int_\delta^{1-\delta} \rho^{\varepsilon}_t |\nabla\log\rho^{\varepsilon}_t|^2 \,\d t\,  \d\mm = 0, \\
\label{eq:b3}
&\lim_{\varepsilon \downarrow 0}\varepsilon^2 \iint_\delta^{1-\delta} \rho^{\varepsilon}_t |\Delta\log\rho^{\varepsilon}_t| |\nabla\log\rho^{\varepsilon}_t| \,\d t\,  \d\mm = 0, \\
\label{eq:b4}
 &\lim_{\varepsilon \downarrow 0} \varepsilon^2 \iint_\delta^{1-\delta}\rho^{\varepsilon}_t |\nabla\log\rho^{\varepsilon}_t|^3 \,\d t\,  \d\mm= 0.
\end{align}
\end{subequations}
\end{Lemma}

\begin{proof}
For \eqref{eq:b1} we notice that
\[
\varepsilon^2 \iint_\delta^{1-\delta} \rho^{\varepsilon}_t |\Delta\log\rho^{\varepsilon}_t|\,\d t\,  \d\mm\leq\eps\sqrt{1-2\delta}\sqrt{\varepsilon^2 \iint_\delta^{1-\delta} \rho^{\varepsilon}_t |\Delta\log\rho^{\varepsilon}_t|^2  \,\d t\,  \d\mm}
\]
and the fact that, by \eqref{eq:blap}, the last square root is uniformly bounded in $\eps\in(0,1)$. 

For \eqref{eq:b2} we start observing that Lemma \ref{lem:1} below applies to $\rho_t^\eps$, because by Proposition \ref{pro:7} $\rho_t^\eps \in \testi\X \cap L^1(\X)$ and
\[
\Delta\rho_t^\eps = f_t^\eps \Delta g_t^\eps + g_t^\eps \Delta f_t^\eps + 2\langle \nabla f_t^\eps,\nabla g_t^\eps \rangle \in L^1(\X).
\]
Hence, from the identity $\rho^{\varepsilon}_t|\nabla\log\rho^{\varepsilon}_t|^2 = -\rho^{\varepsilon}_t \Delta \log \rho^{\varepsilon}_t + \Delta\rho^{\varepsilon}_t$ and the fact that $\int\Delta\rho^{\varepsilon}_t\,\d\mm=0$ we get
\[
\varepsilon^2\iint_{\delta}^{1-\delta}\rho^{\varepsilon}_t|\nabla\log\rho^{\varepsilon}_t|^2 \,\d t\, \d\mm = -\varepsilon^2\iint_{\delta}^{1-\delta}\rho^{\varepsilon}_t \Delta\log\rho^{\varepsilon}_t\,\d t\, \d\mm \leq \varepsilon^2 \iint_{\delta}^{1-\delta}  \rho^{\varepsilon}_t|\Delta\log\rho^{\varepsilon}_t|\,\d t\, \d\mm
\]
and then conclude by \eqref{eq:b1}. 

For \eqref{eq:b3} we observe that
\[
\begin{split}
\varepsilon^2 \iint_\delta^{1-\delta} \rho^{\varepsilon}_t &|\Delta\log\rho^{\varepsilon}_t| |\nabla\log\rho^{\varepsilon}_t| \,\d t\,  \d\mm\\
&\leq \sqrt{\varepsilon^2 \iint_\delta^{1-\delta} \rho^{\varepsilon}_t |\Delta\log\rho^{\varepsilon}_t|^2 \,\d t\,  \d\mm}\sqrt{\varepsilon^2 \iint_\delta^{1-\delta} \rho^{\varepsilon}_t |\nabla\log\rho^{\varepsilon}_t|^2 \,\d t\,  \d\mm},
\end{split}
\]
and use the fact that the first square root in the right-hand side is bounded (by \eqref{eq:blap}) and the second one goes to 0 (by \eqref{eq:b2}). 

To prove \eqref{eq:b4} we start again from   the identity $\rho^{\varepsilon}_t|\nabla\log\rho^{\varepsilon}_t|^2 = -\rho^{\varepsilon}_t \Delta \log \rho^{\varepsilon}_t + \Delta\rho^{\varepsilon}_t$ to get
\[
 \iint_\delta^{1-\delta} \rho^{\varepsilon}_t |\nabla\log\rho^{\varepsilon}_t|^3 \,\d t\,\d\mm=-\iint_\delta^{1-\delta} \rho^{\varepsilon}_t\Delta\log\rho^{\varepsilon}_t |\nabla\log\rho^{\varepsilon}_t| \,\d t\,\d\mm +  \iint_\delta^{1-\delta} \Delta\rho^{\varepsilon}_t |\nabla\log\rho^{\varepsilon}_t|\,\d t\,\d\mm.
\]
After a multiplication by $\eps^2$ we see that the first integral on the right-hand side vanishes as $\varepsilon \downarrow 0$ thanks to \eqref{eq:b3}. For the second we start noticing that an application of the dominated convergence theorem ensures that
\begin{equation}
\label{eq:inpiu}
\iint_\delta^{1-\delta}\Delta\rho^{\varepsilon}_t |\nabla\log\rho^{\varepsilon}_t| \,\d t\,  \d\mm =\lim_{\eta\downarrow 0}\iint_\delta^{1-\delta}\Delta\rho^{\varepsilon}_t \sqrt{\eta+|\nabla\log\rho^{\varepsilon}_t|^2} \,\d t\,  \d\mm,
\end{equation}
then we observe that for every $\eta>0$ the map $z\mapsto \sqrt{\eta+z}$ is in $C^1\cap{\rm LIP}([0,\infty))$ and since $|\nabla\log\rho^{\varepsilon}_t|^2\in W^{1,2}(\X,e^{-V})$ for $V=M\sfd^2(\cdot,\bar x)$ and suitable $\bar x,M$ (recall Proposition \ref{pro:7}) we deduce that $\sqrt{\eta+|\nabla\log\rho^{\varepsilon}_t|^2} \in W^{1,2}(\X,e^{-V})$ as well. Thus by the chain rule for gradients, the Leibniz rule \eqref{eq:leibh} and also using a cut-off argument in conjunction with Lemma \ref{lem:contprov2} to justify integration by parts, we see that it holds
\[
\begin{split}
\bigg|\iint_\delta^{1-\delta}\Delta\rho^{\varepsilon}_t& \sqrt{\eta+|\nabla\log\rho^{\varepsilon}_t|^2} \,\d t\,  \d\mm\bigg|\\
&=\bigg|\iint_\delta^{1-\delta}\frac{\rho^\eps_t}{2\sqrt{\eta+|\nabla\log\rho^\eps_t|^2}}\langle\nabla\log\rho^\eps_t,\nabla|\nabla\log\rho^\eps_t|^2\rangle \,\d t\,  \d\mm\bigg|\\
&=\bigg|\iint_\delta^{1-\delta}\frac{\rho^\eps_t}{\sqrt{\eta+|\nabla\log\rho^\eps_t|^2}}\H{\log\rho^{\varepsilon}_t}(\nabla\log\rho^{\varepsilon}_t,\nabla\log\rho^{\varepsilon}_t) \,\d t\,  \d\mm\bigg|\\
&\leq\iint_\delta^{1-\delta}\rho^\eps_t|\H{\log\rho^{\varepsilon}_t}|_{\HS}\,|\nabla\log\rho^{\varepsilon}_t|\,\d t\,  \d\mm
\end{split}
\]
and being this true for any $\eta > 0$, from \eqref{eq:inpiu} we obtain
\[
\begin{split}
\varepsilon^2\bigg|\iint_\delta^{1-\delta}\Delta\rho^{\varepsilon}_t |\nabla\log\rho^{\varepsilon}_t| \,\d t\,  \d\mm\bigg|  
& \leq  \varepsilon^2\iint_\delta^{1-\delta} \rho^{\varepsilon}_t |\H{\log\rho^{\varepsilon}_t}|_{\HS}|\nabla\log\rho^{\varepsilon}_t|\,\d t\,  \d\mm \\ 
& \leq  \sqrt{\varepsilon^2\iint_\delta^{1-\delta} \rho^{\varepsilon}_t |\H{\log\rho^{\varepsilon}_t}|^2_{\HS} \,\d t\,  \d\mm }\\
& \qquad\qquad\qquad\qquad\times \sqrt{\varepsilon^2\iint_\delta^{1-\delta} \rho^{\varepsilon}_t |\nabla\log\rho^{\varepsilon}_t|^2 \,\d t\,  \d\mm} .
\end{split}
\]
In this last expression the first square root is uniformly bounded in $\varepsilon\in(0,1)$ by \eqref{eq:bhess}, while the second one vanishes as $\varepsilon \downarrow 0$ thanks to \eqref{eq:b2}.

The last claim follows from the fact that under the stated additional regularity properties of $\rho_0,\rho_1$ we can take $\delta=0$ in \eqref{eq:bhess}, \eqref{eq:blap}. Then we argue as before.
\end{proof}

\begin{Lemma}\label{lem:1}
Let $(\X,\sfd,\mm)$ be an $\RCD^*(K,\infty)$ space with $K \in \R$ endowed with a Borel non-negative measure $\mm$ which is finite on bounded sets and $h \in D(\Delta) \cap L^1(\X)$ with $\Delta h \in L^1(\X)$. Then
\[
\int \Delta h \d\mm = 0
\]
\end{Lemma}

\begin{proof}
Let $\bar{x} \in \X$, $R > 0$ and $\nchi_R \in \testi\X$ be a cut-off function  as given by Lemma \ref{lem:cutoff}. Then
\[
\bigg| \int \nchi_R \Delta h \d\mm \bigg| = \bigg| \int \Delta\nchi_R h\d\mm \bigg| = \bigg| \int_{\X \setminus B_{R+1}(\bar{x})} \Delta\nchi_R h \d\mm \bigg| z\leq \|\Delta\nchi_R\|_{L^{\infty}(\X)}\int_{\X \setminus B_{R+1}(\bar{x})} h\d\mm.
\]
Since  Lemma \ref{lem:cutoff} ensures that  $\|\Delta\nchi_R\|_{L^{\infty}(\X)}$ is uniformly bounded in $R$, the conclusion follows letting $R\to\infty$ in the above.
\end{proof}

\section{From entropic to displacement interpolations}\label{sec:6}

\subsection{Compactness}

Starting from the uniform estimates discussed in Section \ref{sec:5}, let us first prove that when we pass to the limit as $\varepsilon \downarrow 0$, up to subsequences Schr\"odinger potentials and entropic interpolations converge in a suitable sense to limit potentials and interpolations.

To formulate the result we need to introduce the Banach space $(C(\X,e^{-V}),\|\cdot\|_{C(\X,e^{-V})})$, where $V = M\sfd^2(\cdot,\bar{x})$ for some $\bar{x} \in \X$ and $M>0$: the norm $\|\cdot\|_{C(\X,e^{-V})}$ is defined as
\[
\|f\|_{C(\X,e^{-V})} := \sup_{x \in \X} |f(x)|e^{-V(x)}
\]
and $C(\X,e^{-V}) := \{f \in C(\X) \,:\, \|f\|_{C(\X,e^{-V})} < \infty \}$.

\begin{Proposition}[Compactness for measures]\label{lem:61} 
With the same assumptions and notations as in Setting \ref{set} the following holds.

For any sequence $\eps_n\downarrow0$ there exists a subsequence, not relabeled, such that  the curves $(\mu^{\eps_n}_t)$ uniformly converge in $(\probt{\X},W_2)$ to a limit curve $(\mu_t)$  belonging to $AC([0,1],(\probt{\X},W_2))$. Moreover, there is $M>0$ so that
\begin{equation}
\label{eq:bdens}
\mu_t\leq M\mm\qquad\forall t\in[0,1]
\end{equation}
and setting $\rho_t:=\frac{\d\mu_t}{\d\mm}$ it holds
\begin{equation}\label{eq:66}
\rho^{\varepsilon_n}_t \stackrel{\ast}{\rightharpoonup} \rho_t \quad \textrm{in } L^{\infty}(\X)\qquad\forall t\in[0,1].
\end{equation}
\end{Proposition}
\begin{proof} Fix $\eps\in(0,1)$; we want to apply Theorem \ref{thm:GH} to   $(\mu^\eps_t)$ and $(\nabla\vartheta^\eps_t)$. The continuity of $t\mapsto\rho^\eps_t\in L^2(\X)$ granted by Proposition \ref{pro:7} yields weak continuity of $(\mu_t)$ and \eqref{eq:bih} is a consequence of \eqref{eq:tail}. From the bound \eqref{eq:boundenergia} it follows \eqref{eq:bkh} and from the formula for $\frac\d{\d t}\rho^\eps_t$ given in Proposition \ref{pro:7} and again the $L^2$-continuity of $(\rho^\eps_t)$ on $[0,1]$ it easily follows that $(\mu_t)$ and $(\vartheta^\eps_t)$ solve the continuity equation in the sense of Theorem \ref{thm:GH}. The conclusion of such theorem ensures that $(\mu^\eps_t)$ is $W_2$-absolutely continuous with 
\[
\int_0^1|\dot\mu^\eps_t|^2\,\d t=\iint_0^1|\nabla\vartheta^\eps_t|^2\rho^\eps_t\,\d t\,\d\mm.
\]
The bound \eqref{eq:boundenergia} grants that the right-hand side is uniformly bounded in $\eps\in(0,1)$ and since $\{(\mu^\eps_t)\}_\eps$ is tight and 2-uniformly integrable by \eqref{eq:tail} (hence $W_2$-compact), this is sufficient to ensure the compactness of the family $\{(\mu^\eps_t)\}_\eps$ in $C([0,1],(\probt{\X},W_2))$ and, by the lower semicontinuity of the kinetic energy, the fact that any limit curve $(\mu_t) $ is absolutely continuous. The bound \eqref{eq:bdens} is then a direct consequence of the uniform bound \eqref{eq:linftybound} and the convergence property \eqref{eq:66} comes from the weak convergence of the measures and the uniform bound on the densities.
\end{proof}
\begin{Proposition}[Compactness for potentials]\label{lem:62} 
With the same assumptions and notations as in Setting \ref{set} the following holds.

For any sequence $\eps_n\downarrow0$ there exists a subsequence, not relabeled, such that for all $\bar{x} \in \X$ and $M>0$, putting $V := M\sfd^2(\cdot,\bar{x})$ we have:
\begin{itemize}
\item[i)] For every $\delta\in(0,1)$ there exists $C>0$ which only depends on $K,N,\delta,\bar{x}$ such that
\begin{equation}
\label{eq:contquadr}
|\varphi^\eps_t|,|\psi^\eps_{1-t}|\leq C(1+\sfd^2(\cdot,\bar x))\qquad\forall t\in[\delta,1],\ \eps\in(0,1)
\end{equation}
\item[ii)]The curves $(\varphi^{\eps_n}_t),(\psi^{\eps_n}_t)$ converge locally uniformly on $\mathcal I$ with values in $L^1(\X,e^{-V}\mm)$ to limit curves $(\varphi_t),(\psi_t) \in AC_{loc}(\mathcal I,L^1(\X,e^{-V}\mm))$ respectively, where  $\mathcal I:=(0,1]$ for the $\varphi$'s, $\mathcal I:=[0,1)$ for the $\psi$'s;
\item[iii)] For all $t \in I$, the functions $\varphi^{\eps_n}_t,\psi^{\eps_n}_t$ also converge in $C(\X,e^{-V})$ to $\varphi_t,\psi_t$. 
\item[iv)] For any $\delta \in (0,1)$ there exists $C>0$ which only depends on $K,N,\delta,\bar{x}$ such that
\begin{equation}\label{eq:lipcontr2}
\sup_{t \in [\delta,1]}\lip(\varphi_t) + \sup_{t \in [0,1-\delta]}\lip(\psi_t) \leq C\big(1+\sfd(\cdot,\bar{x})\big), \qquad \mm\ae;
\end{equation}
\item[v)]
Finally, up to pass to a suitable subsequence to obtain the existence of limit measures $\mu_t$ as in Proposition \ref{lem:61} above, for every $t \in (0,1)$ it holds 
\begin{equation}\label{eq:63}
\begin{split}
\varphi_t+\psi_t&\leq 0\qquad \text{ on }\X,\\
\varphi_t +\psi_t&=0 \qquad \text{ on }\supp(\mu_t).
\end{split}
\end{equation}
\end{itemize}
Similarly, the curves $(\vartheta^{\eps_n}_t)$ and functions $\vartheta^{\eps_n}_t$ converge in $(0,1)$ to the limit curve $t \mapsto \vartheta_t := \frac12(\psi_t-\varphi_t)$ and functions $\vartheta_t$ in the same sense as above.
\end{Proposition}

\begin{proof}

\noindent{\bf (i)} We start claiming that for all $\eps > 0$ and $t,s \in (0,1]$ with $t<s$ it holds
\begin{equation}\label{eq:endenich}
\|\varphi^{\varepsilon}_t - \varphi^{\varepsilon}_s\|_{L^1(\X,e^{-V}\mm)} \leq \iint_t^s e^{-V} \bigg( \frac{|\nabla\varphi^{\varepsilon}_r|^2}{2} + \frac{\varepsilon}{2}|\Delta\varphi^{\varepsilon}_r| \bigg)\d r\,\d\mm.
\end{equation}
Indeed, by Proposition \ref{pro:7} we know that $(\varphi^\eps_t)\in AC([\delta,1],W^{1,2}(\X,e^{-V'}\mm))$ with $V' := M'\sfd^2(\cdot,\bar{x})$ and $M'=M'(\delta)$ sufficiently large, for any $\delta\in(0,1)$.  Thus for any cut-off function $\nchi_R \in \testi\X$ with $\nchi_R \equiv 1$ on $B_R(\bar{x})$ and support in $B_{R+1}(\bar{x})$, we have $(\nchi_R\varphi^\eps_t) \in AC_{loc}((0,1],W^{1,2}(\X,e^{-V}\mm))$ and since $W^{1,2}(\X,e^{-V}\mm)\subset L^1(\X,\mm)$ (because $e^{-V}\mm$  is a finite measure)  a fortiori this is true for $(\nchi_R e^{-V}\varphi^\eps_t)$. From the formula for $\ddt\varphi^\eps_t$ (Proposition \ref{pro:7}) this implies
\[
\|\nchi_R(\varphi^{\varepsilon}_t - \varphi^{\varepsilon}_s)\|_{L^1(\X,e^{-V}\mm)} \leq \iint_t^s \nchi_R e^{-V} \bigg( \frac{|\nabla\varphi^{\varepsilon}_r|^2}{2} + \frac{\varepsilon}{2}|\Delta\varphi^{\varepsilon}_r| \bigg)\d r\,\d\mm,
\]
so that the claim \eqref{eq:endenich} follows by letting $R\to\infty$ and using the monotone convergence theorem. Denoting by $C_\delta$ a constant depending on  $K,N,\rho_0,\rho_1,\bar x,\delta$, but independent of $\eps,t$, whose value might change on the various occurrences it appears,  estimates  \eqref{eq:lipcontr} and \eqref{eq:lapcontr2}  give
\begin{equation}
\label{eq:lipl1}
\|\varphi^{\varepsilon}_t - \varphi^{\varepsilon}_s\|_{L^1(\X,e^{-V}\mm)} \leq C_\delta|s-t |\qquad\forall\eps\in(0,1),\ \forall t,s\in[\delta,1].
\end{equation}
Now we observe that from \eqref{eq:lipcontr} and the fact that $\X$ is a geodesic space it follows that
\begin{equation}
\label{eq:percontquadr}
|\varphi^\eps_t(x)-\varphi^\eps_t(\bar x)|\leq C_\delta\sfd(x,\bar x)(1+\sfd(x,\bar x))\leq C_\delta(1+\sfd^2(x,\bar x))\qquad\forall \eps\in(0,1),\ t\in[\delta,1].
\end{equation}
Which already tells that $\varphi^\eps_t$ has quadratic growth (with constants possibly depending on $t,\eps$). For $\mu$ with finite second moment, integrate  \eqref{eq:percontquadr} w.r.t.\ $\mu$ in the $x$ variable to get
\begin{equation}
\label{eq:perbarx}
|\varphi^\eps_t(\bar x)-\int \varphi^\eps_t\,\d\mu|\leq \int|\varphi^\eps_t(x)-\varphi^\eps_t(\bar x)|\,\d\mu(x)\leq C_\delta\int (1+\sfd^2(x,\bar x))\,\d\mu(x)
\end{equation}
then pick $t:=1$, $\mu:=\mu_1$ and recall that the normalization chosen for $(f^\eps,g^\eps)$ in  Setting \ref{set} reads as $\int \varphi^\eps_1\,\d\mu_1 = 0$ to deduce that $\sup_{\eps\in(0,1)}|\varphi^\eps_1(\bar x)|<\infty$ and thus \eqref{eq:percontquadr} gives
\[
|\varphi^\eps_1|\leq C_\delta(1+\sfd^2(x,\bar x))\qquad\forall \eps\in(0,1)
\]
which in turn implies  $\sup_{\eps\in(0,1)}\|\varphi^\eps_1\|_{L^1(\X,e^{-V}\mm)}<\infty$. This bound in conjunction with \eqref{eq:lipl1} gives
\[
\|\varphi^\eps_t\|_{L^1(\X,e^{-V}\mm)}\leq C_\delta\qquad \forall\eps\in(0,1),\ \forall t\in[\delta,1],
\]
so that picking $\mu:=e^{-V}\mm$ in \eqref{eq:perbarx} we see that $|\varphi^\eps_t(\bar x)|\leq C_\delta$ for every $\eps\in(0,1)$, $t\in[\delta,1]$ and in conclusion \eqref{eq:percontquadr} gives \eqref{eq:contquadr} for the $\varphi^\eps_t$'s.

Following the same lines of thought, the bound \eqref{eq:contquadr} for  $\psi^\eps_t$ will  follow provided we are able to show that  for some measures $\mu^\eps$ with uniformly bounded second moment it holds 
\[
\sup_{t\in[0,1-\delta]}\sup_{\eps\in(0,1)}\big|\int \psi^\eps_t\,\d\mu^\eps\big|<\infty
\]
for any $\delta\in(0,1)$. We pick $\mu^\eps:=\mu^\eps_{1/2}$: \eqref{eq:boundmomento} gives the uniform bound on the second moment, while multiplying by $\rho^\eps_{1/2}$  the identity $\varphi^\eps_{1/2}+\psi^\eps_{1/2}=\eps\log\rho^\eps_{1/2}$ and integrating we get
\[
\int\varphi^\eps_{1/2}\,\d\mu^\eps_{1/2}+\int\psi^\eps_{1/2}\,\d\mu^\eps_{1/2}=H(\mu^\eps_{1/2}\,|\,\mm)
\]
and the conclusion follows from \eqref{eq:boundentropia} and \eqref{eq:contquadr} in conjunction with \eqref{eq:boundmomento}.

\noindent{\bf (ii)}  By Ascoli-Arzel\`a's theorem, for given $\bar x \in \X$ and $C>0$ the set of functions $\varphi$ on $\X$ such that 
\[
|\varphi|\leq  C(1+\sfd^2(\cdot,\bar x))\qquad\qquad \lip(\varphi)\leq C(1+\sfd(\cdot,\bar x))
\]
is a compact subset of $C(\X,e^{-V}\mm)$. Thus for any $\delta\in(0,1)$ the estimates \eqref{eq:contquadr} and \eqref{eq:lipcontr} give that $\{\varphi^\eps_t:\eps\in(0,1), t\in[\delta,1]\}$ is compact in  $C(\X,e^{-V}\mm)$ and thus a fortiori also compact in $L^1(\X,e^{-V}\mm)$. This fact, \eqref{eq:lipl1}, the arbitrariness of $\delta\in(0,1)$ and again Ascoli-Arzel\`a's theorem give the claim. Similarly for the $\psi$'s.

\noindent{\bf (iii)} We  know that for any $t\in(0,1)$ we have  $\varphi^{\eps_n}_t\to \varphi_t$ in $L^1(\X,e^{-V}\mm)$. We also noticed that for any $t\in(0,1)$ the family $\{\varphi^{\eps_n}_t\}_n$ is compact in $C(\X,e^{-V}\mm)$, thus the claim follows.  Similarly for the $\psi$'s.

\noindent{\bf (iv)} We know that for any $x\in\X$ it holds $\lip\, \varphi_t(x)\leq \limi_{r\downarrow0}\Lip(\varphi_t\restr{B_r(x)})$ and, since $\X$ is geodesic, that $\Lip(\varphi^\eps_t\restr{B_r(x)})=\sup_{B_r(x)}\lip\varphi^\eps_t$. Thus the claim follows from the bound \eqref{eq:lipcontr} and the fact that $\Lip(\varphi_t\restr{B_r(x)})\leq \limi_{n\to\infty}\Lip(\varphi_t^{\eps_n}\restr{B_r(x)})$, which in turn is a trivial consequence of the local uniform convergence we already proved.  Similarly for the $\psi$'s.

\noindent{\bf (v)} For the first in \eqref{eq:63} we pass to the limit in the identity 
\begin{equation}
\label{eq:k4}
\varphi^\eps_t+\psi^\eps_t=\eps\log\rho^\eps_t
\end{equation} 
recalling the uniform bound \eqref{eq:linftybound}. To get the second in \eqref{eq:63} we multiply both sides of \eqref{eq:k4} by $\rho^\eps_t$ and integrate to obtain
\[
\int (\varphi^\eps_t+\psi^\eps_t)\rho^\eps_t\,\d\mm=\eps H(\mu^\eps_t \,|\, \mm).
\]
Letting $\eps=\eps_n\downarrow0$ we see that the right-hand side goes to 0 by \eqref{eq:boundentropia}; then we use the fact that $W_2(\mu^{\eps_n}_t,\mu_t)\to0$, that the functions $\varphi^\eps_t,\psi^\eps_t$ have  uniform quadratic growth and converge locally uniformly to $\varphi_t,\psi_t$ respectively to obtain that  the left-hand side goes to $\int\varphi_t+\psi_t\,\d\mu_t$. This is sufficient to conclude.
\end{proof}

\subsection{Identification of the limit curve and potentials}

We now show that the limit interpolation is the geodesic from $\mu_0$ to $\mu_1$ and the limit potentials are Kantorovich potentials. We shall make use of the following simple lemma valid on general metric measure spaces:

\begin{Lemma}\label{lem:dermiste}
Let $(\Y,\sfd_\Y,\mm_\Y)$ be a complete separable metric measure space endowed with a non-negative measure $\mm_\Y$ which is finite on bounded sets and assume that $W^{1,2}(\Y)$ is separable. Let $\ppi$ be a test plan and $f\in W^{1,2}(\Y)$. Then $t\mapsto \int f\circ\e_t\,\d\ppi$ is absolutely continuous and 
\begin{equation}
\label{eq:bder1}
\Big|\frac\d{\d t}\int f\circ\e_t\,\d\ppi \Big|\leq \int |\d f|(\gamma_t)|\dot\gamma_t|\,\d\ppi(\gamma)\qquad{\rm a.e.}\ t\in[0,1],
\end{equation}
where the exceptional set can be chosen to be independent on $f$.

Moreover, if $(f_t)\in AC([0,1],L^2(\Y))\cap L^\infty([0,1],W^{1,2}(\Y))$, then the map $t\mapsto\int f_t\circ\e_t\,\d\ppi$ is also absolutely continuous and 
\[
\frac\d{\d s}\Big(\int f_s\circ\e_s\,\d\ppi\Big)\restr{s=t}=\int \big(\frac\d{\d s}f_s\restr{s=t}\big)\circ\e_t\,\d\ppi+\frac\d{\d s}\Big(\int f_t\circ\e_s\,\d\ppi\Big)\restr{s=t}\qquad {\rm a.e.}\ t\in[0,1].
\]
\end{Lemma}

\begin{proof}
The absolute continuity of $t \mapsto \int f\circ\e_t\,\d\ppi$ and the bound \eqref{eq:bder1} are trivial consequences of the definitions of test plans and Sobolev functions. The fact that the exceptional set can be chosen independently on $f$ follows from the separability of $W^{1,2}(\Y)$ and standard approximation procedures, carried out, for instance, in \cite{Gigli14}.

For the second part, we start noticing that the second derivative in the right-hand side exists for a.e.\ $t$ thanks to what we have just proved, so that the claim makes sense. The absolute continuity follows from the fact that for any $t_0,t_1 \in [0,1]$, $t_0 < t_1$ it holds
\[
\begin{split}
\Big|\int f_{t_1}\circ\e_{t_1}-f_{t_0}\circ\e_{t_0}\,\d\ppi\Big| & \leq \Big|\int f_{t_1}\circ\e_{t_1} - f_{t_1}\circ\e_{t_0}\,\d\ppi\Big| + \Big|\int f_{t_1} - f_{t_0}\,\d(\e_{t_0})_*\ppi\Big| \\
& \leq \iint_{t_0}^{t_1}|\d f_{t_1}|(\gamma_t)|\dot\gamma_t|\,\d t\,\d\ppi(\gamma) + \iint_{t_0}^{t_1}\Big|\ddt f_t\Big|\,\d t\,\d(\e_{t_0})_*\ppi
\end{split}
\]
and our assumptions on $(f_t)$ and $\ppi$. Now fix a point $t$ of differentiability for $(f_t)$ and observe that the fact that $\frac{f_{t+h}-f_t}{h}$ strongly converges in $L^2(\Y)$ to $\ddt f_t$ and $(\e_{t+h})_*\ppi$ weakly converges to $(\e_t)_*\ppi$ as $h \to 0$ and the densities are equibounded is sufficient to get
\[
\lim_{h\to 0}\int\frac{f_{t+h} - f_t}{h}\circ\e_{t+h}\,\d\ppi = \int \ddt f_t\circ\e_t\,\d\ppi = \lim_{h \to 0} \int \frac{f_{t+h}-f_t}{h}\circ\e_{t}\,\d\ppi.
\]
Hence the conclusion comes dividing by $h$ the trivial identity
\[
\begin{split}
\int f_{t+h}\circ\e_{t+h}-f_{t}\circ\e_{t}\,\d\ppi = & \int f_{t}\circ\e_{t+h}-f_{t}\circ\e_{t}\,\d\ppi + \int f_{t+h}\circ\e_{t} - f_{t}\circ\e_{t}\,\d\ppi + \\
& \qquad + \int (f_{t+h}-f_t)\circ\e_{t+h}-(f_{t+h}-f_t)\circ\e_{t}\,\d\ppi
\end{split}
\]
and letting $h\to 0$.
\end{proof}

We now prove that in the limit the potentials evolve according to the Hopf-Lax semigroup (recall formula \eqref{eq:hli}).

\begin{Proposition}[Limit curve and potentials]\label{pro:9} With the same assumptions and notations as in Setting \ref{set} the following holds.

The limit curve $(\mu_t)$ given by Proposition \ref{lem:61} is unique (i.e.\ independent on the sequence $\eps_n \downarrow0$) and is the only $W_2$-geodesic connecting $\mu_0$ to $\mu_1$.

For any $\bar{x} \in \X$, $M>0$ and any limit curve $(\varphi_t)$ given by Proposition \ref{lem:62}, $(\varphi_t)$ is in $AC_{loc}((0,1],C(\X,e^{-V})) \cap L^\infty_{loc}((0,1],W^{1,2}(\X,e^{-V}\mm))$, where $V := M\sfd^2(\cdot,\bar{x})$, and for any $t_0,t_1 \in (0,1]$, $t_0<t_1$ we have
\begin{subequations}
\begin{align}
\label{eq:hl1}
-\varphi_{t_1}&=Q_{t_1-t_0}(-\varphi_{t_0})\\
\label{eq:costointermedio}
\int \varphi_{t_0}\,\d\mu_{t_0}- \int \varphi_{t_1}\,\d\mu_{t_1}&= \frac1{2(t_1-t_0)}W_2^2(\mu_{t_0},\mu_{t_1})
\end{align}
\end{subequations}
and $-(t_1-t_0)\varphi_{t_1}$ is a Kantorovich potential from $\mu_{t_1}$ to $\mu_{t_0}$.
Similarly, for $V$ as above and any limit curve $(\psi_t)$ given by Proposition \ref{lem:62}, $(\psi_t)$ belongs to $AC_{loc}([0,1),C(\X,e^{-V})) \cap L^\infty_{loc}([0,1),W^{1,2}(\X,e^{-V}\mm))$ and for every $ t_0,t_1 \in [0,1)$, $t_0<t_1$ we have
\begin{subequations}
\begin{align}
\label{eq:hl2}
-\psi_{t_0}&= Q_{t_1-t_0}(-\psi_{t_1})\\
\label{eq:costointermedio2}
\int \psi_{t_1}\,\d\mu_{t_1}- \int \psi_{t_0}\,\d\mu_{t_0}&= \frac1{2(t_1-t_0)}W_2^2(\mu_{t_0},\mu_{t_1})
\end{align}
\end{subequations}
and $-(t_1-t_0)\psi_{t_0}$ is a Kantorovich potential from $\mu_{t_0}$ to $\mu_{t_1}$.
\end{Proposition}

\begin{proof}

\noindent{\bf Inequality $\leq$ in \eqref{eq:hl1}}. Pick $x,y \in \X$, $r>0 $, define
\[
\nu^r_x := \frac{1}{\mm(B_r(x))}\mm\restr{B_r(x)} \qquad\qquad\qquad\nu^r_y := \frac{1}{\mm(B_r(y))}\mm\restr{B_r(y)}
\]
and $\ppi^r$ as the only lifting of the only $W_2$-geodesic from $\nu^r_x$ to $\nu^r_y$ (recall point $(i)$ of Theorem \ref{thm:bm}). Since $\nu^r_x,\nu^r_y$ have compact support and $\ppi^r \in \gopt(\nu^r_x,\nu^r_y)$, there exist $\bar{x} \in \X$ and $R > 0$ sufficiently large such that
\begin{equation}\label{eq:13}
\supp((\e_t)_*\ppi^r) \subset B_R(\bar{x}), \qquad \forall t \in [0,1].
\end{equation}
Let $\nchi$ be a Lipschitz cut-off function with bounded support such that $\nchi \equiv 1$ in $B_R(\bar{x})$. Then, let $\eps \in (0,1)$ and $0 < t_0 < t_1 \leq 1$, put $\tilde{\varphi}_t^\eps := \nchi\varphi_t^\eps$ and observe that $(\tilde{\varphi}_t^\eps) \in AC_{loc}((0,1],L^2(\X)) \cap L^{\infty}_{loc}((0,1],W^{1,2}(\X))$ by Proposition \ref{pro:7} and the compactness of the support of $\nchi$; thus, by Lemma \ref{lem:dermiste} applied to $\ppi^r$ and $t \mapsto \tilde{\varphi}^\eps_{(1-t)t_0+tt_1}$, we get
\[
\begin{split}
\ddt\int \tilde{\varphi}^\eps_{(1-t)t_0+tt_1}\circ\e_t\,\d\ppi^r\geq \int(t_1-t_0)\frac{\d}{\d s}\tilde{\varphi}^\eps_s\restr{s={(1-t)t_0+tt_1}}(\gamma_t)-|\d \tilde{\varphi}^\eps_{(1-t)t_0+tt_1}|(\gamma_t)|\dot\gamma_t|\,\d\ppi^r(\gamma).
\end{split}
\]
As \eqref{eq:13} implies that $\nchi(\gamma_t) = 1$ for all $t \in [0,1]$ for $\ppi^r$-a.e. $\gamma$, $\tilde{\varphi}^\eps$ can be replaced by $\varphi^\eps$ in the inequality above and, recalling the expression for $\ddt \varphi^\eps_t$ and using Young's inequality, we obtain
\[
\ddt\int \varphi^\eps_{(1-t)t_0+tt_1}\circ\e_t\,\d\ppi^r\geq \int\eps\frac{t_1-t_0}2\Delta\varphi^\eps_{(1-t)t_0+tt_1}(\gamma_t)-\frac1{2(t_1-t_0)}|\dot\gamma_t|^2\,\d\ppi^r(\gamma).
\]
Integrating in time  and recalling that $\ppi^r$ is optimal we get
\[
\int\varphi^\eps_{t_1}\,\d\nu^r_y-\int\varphi_{t_0}^\eps\,\d\nu^r_x\geq -\frac1{2(t_1-t_0)}W_2^2(\nu^r_y,\nu^r_x)+\iint_0^1\eps\frac{t_1-t_0}2\Delta\varphi^\eps_{(1-t)t_0+tt_1}\circ\e_t\,\d t\,\d\ppi^r.
\]
Let $\eps\downarrow0$ along the sequence $(\eps_n)$ for which $(\varphi^{\eps_n}_t)$ converges to our given $(\varphi_t)$ in the sense of Proposition \ref{lem:62} and use the uniform bound \eqref{eq:lapcontr2} and the fact that $\ppi^r$ has bounded compression to deduce that
\[
\int\varphi_{t_1}\,\d\nu^r_y-\int\varphi_{t_0}\,\d\nu^r_x\geq -\frac1{2(t_1-t_0)}W_2^2(\nu^r_y,\nu^r_x)
\]
and finally letting $r\downarrow 0$  we conclude from the arbitrariness of $x\in\X$ that
\begin{equation}
\label{eq:o1}
-\varphi_{t_1}(y)\leq Q_{t_1-t_0}(-\varphi_{t_0})(y)\qquad\forall y\in \X.
\end{equation}

\noindent{\bf Inequality $\geq$ in \eqref{eq:hl1}}. To prove the opposite inequality we fix $\bar{x} \in \X$, $r > 1$, again $0 < t_0 < t_1 \leq 1$ and let $\tilde R>r$ to be fixed later. Let $\nchi_{\tilde R} \in \testi\X$ be given by Lemma \ref{lem:cutoff}, define the vector field $X_t^\eps := \nchi_{\tilde R}\nabla\varphi_t^\eps$ and apply Theorem \ref{thm:RLF} to $((t_1-t_0)X^\eps_{(1-t)t_1+tt_0})$: the inequality
\[
{\rm div}X_t^\eps \geq \nchi_{\tilde R}\Delta\varphi_t^\eps -|\nabla\nchi_{\tilde R}||\nabla\varphi_t^\eps|
\]
and the bounds \eqref{eq:lipcontr}, \eqref{eq:lapcontr} on $\nabla\varphi^\eps_t,\Delta\varphi^\eps_t$ ensure that the theorem is applicable and we obtain existence of the regular Lagrangian flow $F^\eps$. Notice that from \eqref{eq:lipcontr} we know that $|X_t^\eps| \leq C'(1+\sfd(\cdot,\bar x))$ for all $t \in [t_0,1]$ for some $C' < \infty$ independent of $\tilde R,\eps$, therefore for $\mm$-a.e.\ $x$ we have 
\[
\frac{\d}{\d t}\sfd(F^\eps_t(x),\bar x)\leq {\rm ms}_t({F^\eps_{\cdot}}(x))\stackrel{\eqref{eq:quants}}=(t_1-t_0)|X_{(1-t)t_1+tt_0}^\eps|(F^\eps_t(x))\leq C'(1+\sfd(F^\eps_t(x),\bar x))
\]
for a.e.\ $t$ and thus Gronwall's Lemma implies the existence of $R$ independent of $\tilde R,\eps$ such that for $\mm$-a.e.\ $x$ it holds
\begin{equation}
\label{eq:dagron}
x\in B_r(\bar x)\qquad\Rightarrow\qquad F_t(x)\in B_R(\bar x)\quad\forall t\in[0,1].
\end{equation}
We now fix $\tilde R:=R$ and put $\ppi^\eps:=\mm({B_{r}(\bar{x})})^{-1}(F^\eps_\cdot)_*\mm\restr{B_{r}(\bar{x})}$, where $F^\eps_\cdot:\X\to C([0,1],\X)$ is the $\mm$-a.e.\ defined map which sends $x$ to $t\mapsto F_t^\eps(x)$, and observe that the bound \eqref{eq:quantm} and the identity \eqref{eq:quants} provided by Theorem \ref{thm:RLF} coupled with the estimates \eqref{eq:lipcontr}, \eqref{eq:lapcontr} on $\nabla\varphi^\eps_t,\Delta\varphi^\eps_t$ and the fact that $\nchi_R \in \testi\X$ ensure that $\ppi^\eps$ is a test plan with
\begin{equation}
\label{eq:uniformtest}
\sup_{\eps\in(0,1)}\iint_0^1|\dot\gamma_t|^2\,\d t\,\d\ppi^\eps(\gamma)<\infty\qquad\text{and}\qquad (\e_t)_*\ppi^\eps \leq C\mm \qquad \forall t\in[0,1],\ \eps\in(0,1),
\end{equation}
for some $C < \infty$. Now put  $\tilde{\varphi}_t^\eps := {\nchi_R}\varphi_t^\eps$ and notice that the definition of $\ppi^\eps$ and \eqref{eq:dagron} ensures that for every $t\in[0,1]$ we have $\tilde\varphi_t^\eps=\varphi_t^\eps$ $(\e_t)_*\ppi$-a.e. Moreover we have  $(\tilde{\varphi}_t^\eps) \in AC_{loc}((0,1],L^2(\X)) \cap L^{\infty}_{loc}((0,1],W^{1,2}(\X))$, thus by Lemma \ref{lem:dermiste} applied to $\ppi^\eps$ and $t \mapsto \tilde{\varphi}^\eps_{(1-t)t_1+tt_0}$   we obtain
\[
\begin{split}
\ddt\int &\varphi^\eps_{(1-t)t_1+tt_0}\circ\e_t\,\d\ppi^\eps = \ddt\int \tilde{\varphi}^\eps_{(1-t)t_1+tt_0}\circ\e_t\,\d\ppi^\eps\\
& = \int(t_0-t_1)\frac{\d}{\d s}\tilde{\varphi}^\eps_s\restr{s={(1-t)t_1+tt_0}}\circ\e_t\,\d\ppi^\eps + \frac\d{\d s}\int\tilde{\varphi}^\eps_{(1-t)t_1+tt_0}\circ\e_s\,\d\ppi^\eps\restr{s=t}\\
& = \int(t_0-t_1)\frac{\d}{\d s}\varphi^\eps_s\restr{s={(1-t)t_1+tt_0}}\circ\e_t\,\d\ppi^\eps + (t_1 - t_0) \int\d\tilde{\varphi}^\eps_{(1-t)t_1+tt_0}(X_t^\eps) \circ\e_t\,\d\ppi^\eps\\
& = \int\Big( \frac{t_0-t_1}{2}|\d\varphi^\eps_{(1-t)t_1+tt_0}|^2 + \eps\frac{t_0-t_1}{2}\Delta\varphi^\eps_{(1-t)t_1+tt_0} +(t_1-t_0)|\d\varphi^\eps_{(1-t)t_1+tt_0}|^2\Big)\circ\e_t\,\d\ppi^\eps\\
& = \int\Big( \frac{t_1-t_0}{2}|\d\varphi^\eps_{(1-t)t_1+tt_0}|^2 + \eps\frac{t_0-t_1}{2}\Delta\varphi^\eps_{(1-t)t_1+tt_0} \Big)\circ\e_t\,\d\ppi^\eps.
\end{split}
\]
Integrating in time and recalling \eqref{eq:quants} we deduce
\begin{equation}
\label{eq:k6}
\int \varphi_{t_0}^\eps\circ\e_1- \varphi_{t_1}^\eps\circ\e_0\,\d\ppi^\eps=\iint_0^1\frac1{2(t_1-t_0)}|\dot\gamma_t|^2+\eps\frac{t_0-t_1}{2}\Delta\varphi^\eps_{(1-t)t_1+tt_0}(\gamma_t)\,\d t\,\d\ppi^\eps(\gamma).
\end{equation}
Now, as before, we let $\eps\downarrow0 $  along the sequence $(\eps_n)$ for which $(\varphi^{\eps_n}_t)$ converges to our given $(\varphi_t)$ in the sense of Proposition \ref{lem:62}: the first in \eqref{eq:uniformtest} grants that $(\ppi^\eps)$ is tight in $\prob{C([0,1],\X)}$ (because $\gamma\mapsto\int_0^1|\dot\gamma_t|^2\,\d t$ has locally compact sublevels and $(\e_0)_*\ppi^\eps=\mm(B_r(\bar x))^{-1}\mm\restr{B_r(\bar x)}$) and thus  up to pass to a subsequence, not relabeled, we can assume that $(\ppi^{\eps_n})$ weakly converges to some $\ppi \in \prob{C([0,1],\X)}$. The second in \eqref{eq:uniformtest} and the bound \eqref{eq:lapcontr2} grant that the term with the Laplacian in \eqref{eq:k6} vanishes in the limit and thus taking into account the lower semicontinuity of the 2-energy we deduce that
\[
\int \varphi_{t_0}\circ\e_1- \varphi_{t_1}\circ\e_0\,\d\ppi\geq \frac1{2(t_1-t_0)}\iint_0^1|\dot\gamma_t|^2\,\d t\,\d\ppi\geq  \frac1{2(t_1-t_0)}\int\sfd^2(\gamma_0,\gamma_1)\,\d\ppi(\gamma).
\]
Now notice that \eqref{eq:o1} implies that 
\begin{equation}
\label{eq:hjcurve}
\frac{\sfd^2(\gamma_0,\gamma_1)}{2(t_1-t_0)}\geq \varphi_{t_0}(\gamma_1)-\varphi_{t_1}(\gamma_0)
\end{equation} 
for any curve $\gamma$, hence the above gives
\[
\int \varphi_{t_0}\circ\e_1- \varphi_{t_1}\circ\e_0\,\d\ppi\geq  \frac1{2(t_1-t_0)}\int\sfd^2(\gamma_0,\gamma_1)\,\d\ppi(\gamma)\geq \int \varphi_{t_0}\circ\e_1- \varphi_{t_1}\circ\e_0\,\d\ppi
\]
thus forcing the inequalities to be equalities. In particular, equality in \eqref{eq:hjcurve} holds for $\ppi$-a.e.\ $\gamma$ and since $(\e_0)_*\ppi = \mm\restr{B_{r}(\bar{x})}$, this is the same as to say that for $\mm$-a.e.\ $y \in B_{r}(\bar{x})$ equality holds in \eqref{eq:o1}. Since both sides of $\eqref{eq:o1}$ are continuous in $y$, we deduce that equality holds for any $y \in B_{r}(\bar{x})$ and the arbitrariness of $r$ allows to conclude  that equality actually holds for any $y \in \X$.

\noindent{\bf Other properties of $\varphi_t$}. From Proposition \ref{lem:62} we already know that, for any $\bar{x} \in \X$ and $M>0$, $(\varphi_t) \in AC_{loc}((0,1],L^1(\X,e^{-V}\mm)) \cap L^\infty_{loc}((0,1],W^{1,2}(\X,e^{-V}\mm))$, where $V := M\sfd^2(\cdot,\bar{x})$. Since $\varphi_t$ is a real-valued function for all $t \in (0,1]$, \eqref{eq:hl1} tells us that for all $x \in \X$, $t \mapsto \varphi_t(x)$ satisfies \eqref{eq:55} for a.e.\ $t \in (0,1]$; taking \eqref{eq:lipcontr2} into account, this yields that for all $\delta \in (0,1)$ and $t_0,t_1 \in [\delta,1]$ with $t_0 < t_1$
\[
\|\varphi_{t_1} - \varphi_{t_0}\|_{C(\X,e^{-V})} \leq \sup_{x \in \X} \int_{t_0}^{t_1}\frac{\big(\lip(\varphi_t)(x)\big)^2}{2}\dt \leq \Big( \sup_{x \in \X}C_{\delta}\big(1+\sfd(x,\bar{x})\big)e^{-V(x)}\Big)|t_1 - t_0|
\]
whence $(\varphi_t) \in AC_{loc}((0,1],C(\X,e^{-V})) \cap L^\infty_{loc}((0,1],W^{1,2}(\X,e^{-V}\mm))$.

Up to extract a further subsequence - not relabeled - we can assume that the curves $(\mu^{\eps_n}_t)$ converge to a limit curve $(\mu_t)$ as in Proposition \ref{lem:61}. We claim that for any $t_0,t_1 \in (0,1]$, $t_0 < t_1$ it holds 
\begin{equation}\label{eq:perpot}
-\int \varphi_{t_1}\,\d\mu_{t_1}+\int \varphi_{t_0}\,\d\mu_{t_0}\geq \frac1{2(t_1-t_0)}W_2^2(\mu_{t_0},\mu_{t_1})
\end{equation}
and start observing that since  $(\varphi_t) \in C((0,1],C(\X,e^{-V})) $ and $(\mu_t)\in AC([0,1],(\probt{\X},W_2))$,  by the uniform estimates  \eqref{eq:contquadr} we see that both sides in \eqref{eq:perpot} are continuous in $t_0,t_1$, hence it is sufficient to prove \eqref{eq:perpot} for $t_0,t_1\in(0,1)$.

Now fix $\bar{x} \in \X$ and $R > 0$, let $\nchi_R \in \testi\X$ be a cut-off function as in Lemma \ref{lem:cutoff}  and observe that by Proposition \ref{pro:7} $t \mapsto \int\nchi_R\varphi^\eps_t\rho^\eps_t\,\d\mm$ belongs to $C((0,1]) \cap AC_{loc}((0,1))$ with
\begin{equation}
\label{eq:pert0t1}
\begin{split}
-\ddt\int\nchi_R\varphi^\eps_t\rho^\eps_t\,\d\mm = & \int\nchi_R\Big(-\frac{|\nabla\varphi^\eps_t|^2}2 - \frac\eps2\Delta\varphi^\eps_t - \la\nabla\varphi^\eps_t,\nabla\vartheta^\eps_t \ra\Big)\rho^\eps_t+ \varphi_t^\eps\langle\nabla\nchi_R,\nabla\vartheta_t^\eps\rangle\rho_t^\eps\,\d\mm \\
=& \int\nchi_R\Big(\frac{|\nabla\vartheta^\eps_t|^2}2-\frac{\eps^2}8|\nabla\log\rho^\eps_t|^2-\frac\eps2\Delta\varphi^\eps_t\Big)\rho^\eps_t+  \varphi_t^\eps\langle\nabla\nchi_R,\nabla\vartheta_t^\eps\rangle\rho_t^\eps\,\d\mm
\end{split}
\end{equation}
for a.e.\ $t\in(0,1)$, having also used the identity $\varphi^\eps_t = \frac\eps2\log\rho^\eps_t - \vartheta^\eps_t$. By \eqref{eq:contquadr}, \eqref{eq:boundmomento} and Lemma \ref{lem:contprov2} it is readily verified that $\int\nchi_R\varphi^\eps_t\rho^\eps_t\,\d\mm \to \int\varphi^\eps_t\rho^\eps_t\,\d\mm $ as $R\to\infty$ for any $t\in(0,1)$ and that the rightmost side of \eqref{eq:pert0t1} passes to the limit as $R\to\infty$ locally uniformly in $t\in(0,1)$. Hence after an integration in $t$ and letting $R\to\infty$ in  \eqref{eq:pert0t1} we obtain
\[
-\int \varphi^\eps_{t_1}\,\d\mu^\eps_{t_1} + \int \varphi^\eps_{t_0}\,\d\mu^\eps_{t_0} = \iint_{t_0}^{t_1} \Big(\frac{|\nabla\vartheta^\eps_t|^2}2-\frac{\eps^2}8|\nabla\log\rho^\eps_t|^2-\frac\eps2\Delta\varphi^\eps_t\Big)\rho^\eps_t\,\d t\,\d\mm
\]
As already noticed in the proof of Proposition \ref{lem:61}, $(\mu^\eps_t)$ and $(\nabla\vartheta^\eps_t)$ satisfy the assumptions of Theorem \ref{thm:GH}, thus from such theorem we deduce that
\[
\iint_{t_0}^{t_1}\frac{|\nabla\vartheta^\eps_t|^2}2\rho^\eps_t\,\d t\,\d\mm=\frac12\int_{t_0}^{t_1}|\dot\mu^\eps_t|^2\,\d t\geq\frac1{2(t_1-t_0)}W_2^2(\mu^\eps_{t_0},\mu^\eps_{t_1}).
\]
Therefore
\[
-\int \varphi^\eps_{t_1}\,\d\mu^\eps_{t_1} + \int \varphi^\eps_{t_0}\,\d\mu^\eps_{t_0} \geq \frac1{2(t_1-t_0)}W_2^2(\mu^\eps_{t_0},\mu^\eps_{t_1}) + \iint_{t_0}^{t_1}\Big(-\frac{\eps^2}8|\nabla\log\rho^\eps_t|^2-\frac\eps2\Delta\varphi^\eps_t\Big)\rho^\eps_t\,\d t\,\d\mm.
\]
We now pass to the limit in $\eps = \eps_n\downarrow0$: we know from Proposition \ref{lem:61} that $W_2(\mu^{\eps_n}_t,\mu_t)\to 0$ and together with \eqref{eq:contquadr} this also grants that the left-hand side trivially converges to the left-hand side of \eqref{eq:perpot}. The contribution of the term with $|\nabla\log\rho^\eps_t|$ vanishes by \eqref{eq:b2} and so does the one with $\Delta\varphi_t^\eps$ by \eqref{eq:lapcontr2} and \eqref{eq:tail}. Hence \eqref{eq:perpot} is proved.

Now notice that \eqref{eq:hl1} can be rewritten as
\[
-(t_1-t_0)\varphi_{t_1}=\big((t_1-t_0)\varphi_{t_0}\big)^c,
\]
so that in particular $-(t_1-t_0)\varphi_{t_1}$ is $c$-concave and $(-(t_1-t_0)\varphi_{t_1})^c\geq (t_1-t_0)\varphi_{t_0}$. Hence both \eqref{eq:costointermedio} and the fact that $-(t_1-t_0)\varphi_{t_1}$ is a Kantorovich potential follow from
\[
\begin{split}
\frac12W_2^2(\mu_{t_0},\mu_{t_1})&\geq \int -(t_1-t_0)\varphi_{t_1}\,\d\mu_{t_1}+\int (-(t_1-t_0)\varphi_{t_1})^c\,\d\mu_{t_0}\\
& \geq \int -(t_1-t_0)\varphi_{t_1}\,\d\mu_{t_1}+\int (t_1-t_0)\varphi_{t_0}\,\d\mu_{t_0}\stackrel{\eqref{eq:perpot}}\geq \frac12W_2^2(\mu_{t_0},\mu_{t_1})
\end{split}
\]
Then \eqref{eq:costointermedio2} and the other  claims about $(\psi_t)$ are proved in the same way.

\noindent{\bf $(\mu_t)$ is a geodesic}. Let $[t_0,t_1] \subset (0,1)$, pick $t \in [0,1]$ and put $t_0' := (1-t)t_1 + tt_0$. We know that $-(t_1-t_0)\varphi_{t_1}$ and $-t(t_1-t_0)\varphi_{t_1}$ are Kantorovich potentials from $\mu_{t_1}$ to $\mu_{t_0}$ and from $\mu_{t_1}$ to $\mu_{t_0'}$ respectively and thus by point $(ii)$ of Theorem \ref{thm:bm} we deduce
\[
W_2^2(\mu_{t_0},\mu_{t_1}) = \int|\d((t_1-t_0)\varphi_{t_1})|^2\,\d\mu_{t_1} = \frac1{t^2}\int|\d((t_1-t_0')\varphi_{t_1})|^2\,\d\mu_{t_1} = \frac{(t_1-t_0)^2}{(t_1-t_0')^2}W_2^2(\mu_{t_1},\mu_{t_0'}).
\]
Swapping the roles of $t_0,t_1$ and using the $\psi$'s in place of the $\varphi$'s we then get
\[
W_2(\mu_{t_1'},\mu_{t_0'})=\frac{t_1'-t_0'}{t_1-t_0}W_2(\mu_{t_1},\mu_{t_0})\qquad\forall\, [t_0',t_1']\subset [t_0,t_1]\subset (0,1).
\]
This grants that the restriction of $(\mu_t)$ to any interval $[t_0,t_1] \subset (0,1)$ is a constant speed geodesic. Since $(\mu_t)$ is continuous on the whole $[0,1]$, this gives the conclusion. Since in this situation the $W_2$-geodesic connecting $\mu_0$ to $\mu_1$ is unique (recall point $(i)$ of Theorem \ref{thm:bm}), by the arbitrariness of the subsequences chosen we also proved the uniqueness of the limit curve $(\mu_t)$.
\end{proof}

\begin{Remark}[The vanishing viscosity limit]{\rm
The part of this last proposition concerning the properties of the $\varphi^\eps_t$'s is valid in a context wider than the one provided by Schr\"odinger problem: we could restate the result by saying that if $(\varphi^\eps_t)$ solves
\begin{equation}
\label{eq:hjvv}
\ddt\varphi^\eps_t = \frac12{|\nabla\varphi_t^\eps|^2} + \frac\eps2\Delta\varphi^\eps_t
\end{equation}
and $\varphi^\eps_0$ uniformly converges to some $\varphi_0$, then $\varphi^\eps_t$ uniformly converges to $\varphi_t := -Q_t(-\varphi_0)$.

In this direction, it is worth recalling that in \cite{AF14} and \cite{GS15} it has been developed a theory of viscosity solutions for some first-order Hamilton-Jacobi equations on metric spaces. This theory applies in particular to the equation 
\begin{equation}
\label{eq:hjv}
\ddt\varphi_t=\frac12\lip(\varphi_t)^2
\end{equation}
whose only viscosity solution is given by the formula $\varphi_t := -Q_t(-\varphi_0)$.

Therefore, we have just proved that if one works not only on a metric space, but on a metric measure space which is an $\RCD^*(K,N)$ space, then the solutions of the viscous approximation \eqref{eq:hjvv} converge to the unique viscosity solution of \eqref{eq:hjv}, in accordance with the classical case.
}\fr
\end{Remark}

\begin{Remark}{\rm
It is not clear whether the `full' families $\varphi^\eps_t,\psi^\eps_t$ converge as $\eps\downarrow0$ to a unique limit. This is related to the non-uniqueness of the  Kantorovich potentials in the classical optimal transport problem.
}\fr
\end{Remark}

We shall now make use of the following lemma. It could be directly deduced from the results obtained by Cheeger in \cite{Cheeger00}, however, the additional regularity assumptions on both the space and the function allow for a `softer' argument based on the metric Brenier's theorem, which we propose.

\begin{Lemma}\label{lem:12}
Let $(\Y,\sfd_\Y,\mm_\Y)$ be an $\RCD^*(K,N)$ space with $K \in \mathbb{R}$ and $N \in [1,\infty)$ and let $\phi : \X \to \R\cup\{-\infty\}$ be a $c$-concave function not identically $-\infty$. Let $\Omega$ be the interior of the set $\{\phi>-\infty\}$. Then $\phi$ is locally Lipschitz on $\Omega$ and
\[
\lip\,\phi = |\d\phi|, \quad \mm\textrm{-a.e. on }\Omega.
\]
\end{Lemma}

\begin{proof}
Lemma 3.3 in \cite{GigliRajalaSturm13} grants that $\phi$ is locally Lipschitz on $\Omega$ and that $\partial^c\phi(x) \neq \emptyset$ for every $x \in \Omega$. The same lemma also grants that for $K \subset \Omega$ compact, the set $\cup_{x\in K}\partial^c\phi(x)$ is bounded. Recalling that $\partial^c\phi$ is the set of $(x,y) \in \Y^2$ such that
\[
\phi(x) + \phi^c(y) = \frac12\sfd^2(x,y)
\]
and that $\phi,\phi^c$ are upper semicontinuous, we see that $\partial^c\phi$ is closed. Hence for $K\subset \Omega$ compact the set $\cup_{x\in K}\partial^c\phi(x)$ is compact and not empty and thus by the Kuratowski--Ryll-Nardzewski Borel selection theorem we deduce the existence of a Borel map $T:\Omega\to \Y$ such that $T(x)\in\partial^c\phi(x)$ for every $x\in \Omega$.  

Pick $\mu \in \probt\Y$ with $\supp(\mu)\subset\subset\Omega$ and $\mu\leq C\mm$ for some $C>0$ and set  $\nu := T_{*}\mu$. By construction, $\mu,\nu$ have both bounded support,  $T$ is an optimal map and $\phi$ is a Kantorovich potential from $\mu$ to $\nu$. 

Hence point $(iii)$ of Theorem \ref{thm:bm}  applies and since $\lip\,\phi=\max\{|D^+\phi|,|D^-\phi|\}$, by the arbitrariness of $\mu$ to conclude it is sufficient to show that $|D^+\phi|=|D^-\phi|$ $\mm$-a.e. This easily follows from the fact that $\mm$ is doubling and $\phi$ Lipschitz, see Proposition 2.7 in \cite{AmbrosioGigliSavare11}.
\end{proof}

With this said, we can now show that the weighted energies of the Schr\"odinger potentials converge  to the weighted energy of the limit ones:
\begin{Proposition}\label{thm:10}
With the same assumptions and notations as in Setting \ref{set} the following holds.

Let $\eps_n \downarrow 0$ be a sequence such that $(\varphi^{\eps_n}_t),(\psi^{\eps_n}_t)$ converge to limit curves $(\varphi_t),(\psi_t)$ as in Proposition \ref{lem:62} and let $V := M\sfd^2(\cdot,\bar{x})$ with $\bar{x} \in \X$ and $M>0$ arbitrary. Then for every $\delta\in(0,1)$ we have
\begin{equation}
\label{eq:energialimite}
\begin{split}
\lim_{n \to \infty}\iint_{\delta}^1 e^{-V}|\d \varphi^{\varepsilon_n}_t|^2\,\d t\,\d\mm & = \iint_{\delta}^1 e^{-V}|\d\varphi_t|^2\,\d t\,\d\mm,\\
\lim_{n \to \infty}\iint_0^{1-\delta} e^{-V}|\d \psi^{\varepsilon_n}_t|^2\,\d t\,\d\mm & = \iint_0^{1-\delta} e^{-V}|\d\psi_t|^2\,\d t\,\d\mm.
\end{split}
\end{equation}
\end{Proposition}
\begin{proof}
Fix $\delta \in (0,1)$ and argue as in the proof of Proposition \ref{lem:62} to obtain that $t\mapsto \int e^{-V}\varphi^\eps_t\,\d\mm$ is absolutely continuous in $[\delta,1]$ (see in particular \eqref{eq:lipl1}) and that
\[
\int e^{-V}\big(\varphi^\eps_1-\varphi^\eps_\delta\big)\d\mm = \frac12\iint_\delta^1 e^{-V}\big(|\d \varphi^\eps_t|^2+\eps\Delta\varphi^\eps_t\big)\d t\,\d\mm.
\]
Choosing $\eps:=\eps_n$, letting $n\to\infty$ and using the uniform bounds \eqref{eq:lapcontr2}, \eqref{eq:contquadr} and the volume growth estimate \eqref{eq:volgrowth} we obtain
\begin{equation}
\label{eq:k7}
\lim_{n\to\infty} \frac12\iint_\delta^1 e^{-V}|\d \varphi^{\eps_n}_t|^2\,\d t\,\d\mm = \lim_{n\to\infty} \int e^{-V}\big(\varphi^{\eps_n}_1-\varphi^{\eps_n}_\delta\big)\d\mm = \int e^{-V}\big(\varphi_1-\varphi_\delta\big)\d\mm.
\end{equation}
Combining \eqref{eq:55} and \eqref{eq:hl1} we see that for any $x \in \X$ it holds
\[
\ddt \varphi_t(x) = \frac{1}{2}\big((\lip\,\varphi_t)(x)\big)^2\qquad a.e.\ t\in[0,1].
\]
By Fubini's theorem, the same identity holds for $\mathscr{L}^1 \otimes \mm$-a.e.\ $(t,x) \in [\delta,1] \times \X$.  The identity \eqref{eq:hl1}  also grants that $\varphi_t$ is a multiple of a $c$-concave function, thus the thesis of Lemma \ref{lem:12} is valid for $\varphi_t$ and recalling that $(\varphi_t)\in AC_{loc}((0,1],L^1(\X,e^{-V}\mm))$ by Proposition \ref{lem:62} we deduce that  
\[
\int e^{-V}\big(\varphi_1-\varphi_\delta\big)\d\mm = \int_\delta^1 \ddt\int e^{-V}\varphi_t\d\mm\,\d t = \iint_{\delta}^1 e^{-V}\frac{|\d\varphi_t|^2}{2}\,\d t\,\d\mm,
\]
which together with \eqref{eq:k7} gives the first in \eqref{eq:energialimite}. The proof of the second is analogous.
\end{proof}

As a direct consequence of the limit \eqref{eq:energialimite} and the local equi-Lipschitz bounds \eqref{eq:lipcontr} we obtain the following result. In order to state it, let us introduce the module $L^2(T^*\X,e^{-V}\mm)$ as $\{\omega \in L^0(T^*\X) \,:\, |\omega| \in L^2(\X,e^{-V}\mm)\}$; an analogous definition can be given for $L^2((T^*)^{\otimes 2}\X)$.

\begin{Corollary}\label{cor:convd}
With the same assumptions and notations as in Setting \ref{set} the following holds.

Let $\eps_n\downarrow 0$ be a sequence such that $(\varphi^{\eps_n}_t), (\psi^{\eps_n}_t)$ converge to limit curves $(\varphi_t),(\psi_t)$ as in Proposition \ref{lem:62}. Then for every $\delta \in (0,1)$, $\bar{x} \in \X$ and $M>0$ we have
\begin{equation}
\label{eq:limited}
\begin{array}{rlll}
(\d\varphi^{\eps_n}_t) \quad & \to \quad (\d\varphi_t) && \text{ in }\quad L^2([\delta,1],L^2(T^*\X,e^{-V}\mm))\\
(\d\psi^{\eps_n}_t) \quad & \to \quad (\d\psi_t) && \text{ in }\quad L^2([0,1-\delta],L^2(T^*\X,e^{-V}\mm))\\
(\d\varphi^{\eps_n}_t \otimes \d\varphi^{\eps_n}_t) \quad & \to \quad (\d\varphi_t \otimes \d\varphi_t) && \text{ in } \quad L^2([\delta,1],L^2((T^*)^{\otimes 2}\X,e^{-V}\mm))\\
(\d\psi^{\eps_n}_t \otimes \d\psi^{\eps_n}_t) \quad & \to \quad (\d\psi_t \otimes \d\psi_t) && \text{ in }\quad L^2([0,1-\delta],L^2((T^*)^{\otimes 2}\X,e^{-V}\mm))\\
(\d\varphi^{\eps_n}_t \otimes \d\psi^{\eps_n}_t) \quad & \to \quad (\d\varphi_t \otimes \d\psi_t) && \text{ in }\quad L^2([\delta,1-\delta],L^2((T^*)^{\otimes 2}\X,e^{-V}\mm))
\end{array}
\end{equation}
where $V := M\sfd^2(\cdot,\bar{x})$.
\end{Corollary}

\begin{proof}
Let $V$ be as in the statement and start noticing that the closure of the differential grants that $\d\varphi^{\eps_n}_t \weakto \d\varphi_t$ in $L^2(T^*\X,e^{-V}\mm)$ for all $t\in(0,1]$. This and the fact that $(\d\varphi^{\eps_n}_t)$ is equibounded in $L^2([\delta,1],L^2(T^*\X,e^{-V}\mm))$, as a direct consequence of \eqref{eq:lipcontr}, are sufficient to ensure that $(\d\varphi^{\eps_n}_t) \weakto (\d\varphi_t)$ in $L^2([\delta,1],L^2(T^*\X,e^{-V}\mm))$. Given that the first in \eqref{eq:energialimite} grants convergence of the $L^2([\delta,1],L^2(T^*\X,e^{-V}\mm))$-norms, we deduce strong convergence. This establishes the first limit.

Now observe that for every $\omega \in L^2([\delta,1],L^2(T^*\X,e^{-V}\mm))$ the fact that $e^{-V}|\d\varphi^{\eps_n}_t|$ is uniformly bounded in $L^\infty([\delta,1] \times \X)$ for every $M>0$ in the definition of $V$ and the strong $L^2$-convergence just proved ensure that $\la \d\varphi^{\eps_n}_t,\omega_t\ra \to \la \d\varphi_t,\omega_t\ra$ in $L^2([\delta,1]\times \X,\d t \otimes e^{-V}\mm)$. It follows that for any $\omega_1,\omega_2\in L^2([\delta,1],L^2(T^*\X,e^{-V}\mm))$ we have
\[
\iint_\delta^1 e^{-V}\la \d\varphi^{\eps_n}_t,\omega_{1,t}\ra \la \d\varphi^{\eps_n}_t,\omega_{2,t}\ra\,\d t\,\d\mm \quad \to \quad \iint_\delta^1 e^{-V}\la \d\varphi_t,\omega_{1,t}\ra \la \d\varphi_t,\omega_{2,t}\ra\,\d t\,\d\mm
\]
and thus to conclude it remains to prove that
\[
\iint_\delta^1 e^{-V}|\d\varphi^{\eps_n}_t \otimes \d\varphi^{\eps_n}_t|_\HS^2\,\d t\,\d\mm \quad \to \quad \iint_\delta^1 e^{-V}|\d\varphi_t \otimes \d\varphi_t|_\HS^2\,\d t\,\d\mm.
\]
Since $|v\otimes v|_\HS^2=|v|^4$ this is a direct consequence of the fact that $|\d\varphi^{\eps_n}_t|$ is uniformly bounded and converge to $|\d\varphi_t|$ in $L^2([\delta,1]\times \X,\d t \otimes e^{-V}\mm)$. Hence also the third limit is established.

The other claims  follow by analogous arguments.
\end{proof}

The estimates that we have for the functions $\varphi$'s tell nothing about their regularity as $t \downarrow 0$ and similarly little we know so far about the $\psi$'s for $t \uparrow 1$. We now see in which sense limit functions $\varphi_0,\psi_1$ exist. This is not needed for the proof of our main result, but we believe it is relevant on its own.

Thus let us fix $\eps_n \downarrow 0$ so that $\varphi^{\eps_n}_t\to \varphi_t$ for $t\in(0,1]$ and $\psi^{\eps_n}_t\to \psi_t$ for $t\in[0,1)$ as in Proposition \ref{lem:62}. Then define the functions $\varphi_0,\psi_1:\X\to\R\cup\{-\infty\}$ as
\begin{equation}
\label{eq:varphi0}
\begin{split}
\varphi_0(x):=\inf_{t\in(0,1]}\varphi_t(x)=\lim_{t\downarrow0}\varphi_t(x),\\
\psi_1(x):=\inf_{t\in[0,1)}\psi_t(x)=\lim_{t\uparrow1}\psi_t(x).
\end{split}
\end{equation}
Notice that the fact that the $\inf$ are equal to the stated limits is a consequence of formulas \eqref{eq:hl1}, \eqref{eq:hl2}, which directly imply that for every $x\in \X$ the maps $t\mapsto\varphi_t(x)$ and $t\mapsto\psi_{1-t}(x)$ are non-decreasing.

The main properties of $\varphi_0,\psi_1$ are collected in the following proposition:

\begin{Proposition}\label{pro:v0}
With the same assumptions and notations as in Setting \ref{set} and for $\varphi_0,\psi_1$ defined by \eqref{eq:varphi0} the following holds.

\begin{itemize}
\item[i)] The functions $-\varphi_t$ (resp. $-\psi_t$) $\Gamma$-converge to $-\varphi_0$ (resp. $-\psi_1$) as $t\downarrow0$ (resp. $t\uparrow 1$).
\item[ii)] For every $t\in(0,1]$ we have
\[
Q_t(-\varphi_0)=-\varphi_t\qquad\qquad Q_t(-\psi_1)=-\psi_{1-t}.
\]
\item[iii)] It holds
\[
\varphi_0(x)=
\left\{\begin{array}{ll}
-\psi_0(x)&\quad\text{if }x\in\supp(\rho_0)\\
-\infty&\quad\text{otherwise}
\end{array}
\right.
\qquad\quad
\psi_1(x)=
\left\{\begin{array}{ll}
-\varphi_1(x)&\quad\text{if }x\in\supp(\rho_1)\\
-\infty&\quad\text{otherwise}
\end{array}
\right.
\]
\item[iv)] We have
\[
\int\varphi_0\rho_0\,\d\mm+\int\psi_1\rho_1\,\d\mm=\frac12W_2^2(\mu_0,\mu_1).
\]
\item[v)] Define $\varphi^\eps_0$ on $\supp(\rho_0)$ as $\varphi^\eps_0:=\eps\log(f^\eps)$ and let $\eps_n\downarrow0$ be such that $\varphi^{\eps_n}_t,\psi^{\eps_n}_t$ converge to $\varphi_t,\psi_t$ as $n\to\infty$ as in Proposition \ref{lem:62}. 

Then the functions $\rho_0\varphi^{\eps_n}_0$, set to be 0 on $\X \setminus \supp(\rho_0)$, converge to $\rho_0\varphi_0$ in $L^\infty(\X)$ as $n \to \infty$. 

With the analogous definition of $\rho_1\psi^{\eps_n}_1$ we have that these converge to $\rho_1\psi_1$ in $L^\infty(\X)$ as $n\to\infty$. 
\end{itemize}
\end{Proposition}

\begin{proof}
We shall prove the claims for $\varphi_0$ only, as those for $\psi_1$ follow along similar lines.\\
\noindent{\bf (i)} For the $\Gamma-\lims$ inequality we simply observe that by definition $-\varphi_0(x)=\lim_{t\downarrow0}-\varphi_t(x)$. To prove the $\Gamma-\limi$ inequality, use the fact that $-\varphi_t\geq-\varphi_s$ for $0<t\leq s$ and the continuity of $\varphi_s$:  for given $(x_t)$ converging to $x$ we have
\[
\limi_{t\downarrow0}-\varphi_t(x_t)\geq\limi_{t\downarrow0}-\varphi_s(x_t)=-\varphi_s(x)\qquad\forall s>0.
\]
The conclusion follows letting $s\downarrow0$.

\noindent{\bf (ii)} From $-\varphi_0 \geq -\varphi_s$ we deduce that 
\[
Q_t(-\varphi_0) \geq Q_t(-\varphi_s) \stackrel{\eqref{eq:hl1}}= -\varphi_{t+s} \qquad \forall s \in (0,1]
\]
and thus letting $s \downarrow 0$ and using the continuity of $(0,1] \ni t \mapsto \varphi_t(x)$ for all $x \in \X$ we obtain $Q_t(-\varphi_0)(x) \geq -\varphi_t(x)$ for all $x \in \X$. For the opposite inequality, notice that the second \eqref{eq:weightbound} gives
\begin{equation}
\label{eq:psialto}
\varphi_t^\eps \leq \eps\log C_4 - \eps\log v_{\eps t/2} + \eps\log\|f^\eps\|_{L^1(\X)} - \frac{C_5 \sfd^2(\cdot,\bar x)}{t}+ \frac{C_6}t
\end{equation}
for all $t \in (0,1]$ with $C_4,C_5,C_6$ depending on $K,N,\rho_0,\rho_1,\bar x$ only and $v_{\eps t/2}$ as in \eqref{eq:vV}. We now claim that for every $\eps\in(0,1)$ we have
\begin{equation}
\label{eq:volneg}
\begin{split}
\eps\log v_\eps & \geq -C,\qquad\qquad\qquad
\eps \log\|f^\eps\|_{L^1(\X)} \leq C 
\end{split}
\end{equation}
for some constant $C$ depending on $K,N,\rho_0,\rho_1,\bar x$ only. Indeed, from \eqref{eq:doubling} we see that letting $D$ be the diameter of $\supp(\rho_0)$ and $c=c(D)$ a constant depending only on $D$ we have
\[
\mm(B_{\sqrt\eps}(x)) \geq c^{\log_2(D/\sqrt\eps)+1}\mm(\supp(\rho_0)) \qquad \forall x \in \supp(\rho_0).
\]
Thus $v_\eps \geq C^{\log_2(D/\sqrt\eps)+1}\mm(\supp(\rho_0))$ and thus the first in \eqref{eq:volneg} follows. For the second we start noticing that the first inequality in \eqref{eq:gaussest}, the identity  $\int f^\eps\otimes g^\eps\,\d \hR^{\eps/2}=1$ and the fact that the supports of $f^\eps,g^\eps$ coincide with those of $\rho_0,\rho_1$ respectively give
\begin{equation}
\label{eq:fgeps}
\eps\log\Big(\|f^\eps\|_{L^1(\X)}\|g^\eps\|_{L^1(\X)}\Big)=\eps\log\int_{\supp(\rho_0)\times\supp(\rho_1)} f^\eps\otimes g^\eps\,\d\mm \leq \eps\log(C_1\mm(B))+\bar D^2+C_2 \eps^2,
\end{equation}
for every $\eps\in(0,1)$, where $\bar D$ is the diameter of $\supp(\rho_0)\cup\supp(\rho_1)$ and $B$ is the 1-neighbourhood of $\supp(\rho_0)\cup\supp(\rho_1)$. Then recall the normalization \eqref{eq:normalization}, the identity $\log\rho_1=\log g^\eps+\log(\h_{\eps/2}f^\eps)$ and use Jensens' inequality for $-\log$ to obtain
\[
H(\mu_1 \,|\, \mm) = \int \rho_1\log\rho_1\,\d\mm = \int \log(g^\eps)\rho_1\,\d\mm \leq \log\int g^\eps\rho_1\,\d\mm \leq \log\Big(\|g^\eps\|_{L^1(\X)}\|\rho_1\|_{L^{\infty}(\X)}\Big),
\]
whence $\log\|g^\eps\|_{L^1(\X)} \geq H(\mu_1 \,|\, \mm) -\log\|\rho_1\|_{L^{\infty}(\X)}$ for all $\eps \in (0,1)$, which together with \eqref{eq:fgeps} gives the second in \eqref{eq:volneg}.

Therefore passing to the limit in \eqref{eq:psialto} as $\eps=\eps_n\downarrow0$ and recalling the local uniform convergence of $\varphi^{\eps_n}_t$ to $\varphi_t$ give $-\varphi_t\geq -\frac{\tilde C}t+\frac{C_5\sfd^2(\cdot,\bar x)}t$ for every $t\in(0,1]$, where $\tilde C\geq 0$ depends on $K,N,\rho_0,\rho_1,\bar x$ only. It follows that
\begin{equation}
\label{eq:coerc}
-\varphi_t\geq \frac{C_5\sfd^2(\cdot,\bar x)}{2t}\geq {\frac{C_5}2\sfd^2(\cdot,\bar x)},\qquad\forall t\in(0,1],\ x\notin B_{\sqrt{\frac{2C_5}{\tilde C}}}(\bar x).
\end{equation}
Now  fix $x \in \X$ and a sequence $t_n \downarrow 0$: the bound \eqref{eq:coerc} grants that there are $y_n \in \X$ such that
\[
Q_t(-\varphi_{t_n})(x) = \frac{\sfd^2(x,y_n)}{2t} - \varphi_{t_n}(y_n)
\]
and that these $y_n$ range in a bounded set. Thus up to pass to a subsequence we can assume that $y_n \to y$ for some $y \in \X$, so that taking into account the $\Gamma-\limi$ inequality previously proved we get
\[
\frac{\sfd^2(x,y)}{2t}-\varphi_{0}(y)\leq \limi_{n\to\infty}\frac{\sfd^2(x,y_n)}{2t}-\varphi_{t_n}(y_n)=\limi_{n\to\infty}Q_t(-\varphi_{t_n})(x)\stackrel{\eqref{eq:hl1}}=\limi_{n\to\infty}-\varphi_{t_n+t}(x)=-\varphi_t(x)
\]
which shows that  $Q_t(-\varphi_0)(x)\leq -\varphi_t(x)$, as desired.

\noindent{\bf (iii)} For any $t \in (0,1]$ we have
\[
\varphi_0\leq \varphi_t\stackrel{\eqref{eq:63}}\leq-\psi_t
\]
so that letting $t \downarrow 0$ and using the continuity of $[0,1) \ni t \mapsto \psi_t(x)$ for all $x \in \X$ we deduce that 
\[
\varphi_0\leq -\psi_0\qquad{\rm on}\ \X.
\]
Now notice that the fact that $-\varphi_0 \leq \Gamma-\limi(-\varphi_t)$ implies that
\begin{equation}
\label{eq:magg}
\varphi_0(\gamma_0)\geq\lims_{t\downarrow 0}\varphi_t(\gamma_t)\qquad\forall \gamma\in C([0,1],\X).
\end{equation}
Let $\ppi$ be the lifting  of the $W_2$-geodesic $(\mu_t)$ (recall point $(i)$ of Theorem \ref{thm:bm}); taking into account that the evaluation maps $\e_t:C([0,1],\X)\to \X$ are continuous and that $\supp(\ppi)$ is a compact subset of $C([0,1],\X)$, because given by constant speed geodesics running from the compact set $\supp(\rho_0)$ to the compact $\supp(\rho_1)$, it is easy to see that for every $\gamma \in \supp(\ppi)$ and $t \in [0,1]$ we have $\gamma_t \in \supp(\mu_t)$ and viceversa for every $x \in \supp(\mu_t)$ there is $\gamma \in \supp(\ppi)$ with $\gamma_t=x$.  

Thus let $x \in \supp(\rho_0) = \supp(\mu_0)$ and find $\gamma \in \supp(\ppi)$ with $\gamma_0 = x$: from the fact that $\gamma_t \in \supp(\mu_t)$ and \eqref{eq:63} we get
\[
\varphi_0(x)\stackrel{\eqref{eq:magg}}\geq\lims_{t\downarrow 0}\varphi_t(\gamma_t)=\lims_{t\downarrow0}-\psi_t(\gamma_t).
\]
and since the continuity of $[0,1)\ni t\mapsto \psi_t\in L^1(\X,e^{-V})$ and the uniform local Lipschitz continuity of the $\psi_t$'s (both coming from Proposition \ref{lem:62}) imply local uniform convergence of $\psi_t$ to $\psi_0$, we conclude $\varphi_0(x)\geq \psi_0(x)$.

Thus it remains to prove that $\varphi_0 = -\infty$ outside $\supp(\rho_0)$. To this aim, we notice again that the supports of $f^\eps,g^\eps$ coincide with those of $\rho_0,\rho_1$ and use the second in \eqref{eq:gaussest} to get
\[
\begin{split}
f^\eps_t(x)&=\h_{\eps t/2}f^\eps(x)=\int f^\eps(y)\hr_{\eps t/2}(x,y)\,\d\mm(y)\leq\frac{c_1}{v_{\eps t/2}}e^{-c_2\frac{\sfd^2(x,\supp(\rho_0))}{3\eps t}+c_3\eps t}\int f^\eps\,\d\mm,\\
g^\eps_t(x)&=\h_{\eps (1-t)/2}g^\eps(x)=\int g^\eps(y)\hr_{\eps(1-t)/2}(x,y)\,\d\mm(y)\leq\frac{c_4}{v_{\eps (1-t)/2}}\int g^\eps\,\d\mm,
\end{split}
\]
for every $t \in (0,1)$ and constants $c_i>0$ depending on  $K,N,\rho_0,\rho_1,\bar x$ only. From these bounds, the identity $\rho^\eps_t=f^\eps_tg^\eps_t$ and the estimates \eqref{eq:fgeps} and \eqref{eq:volneg} we deduce that
\begin{equation}
\label{eq:altdens}
\lims_{\eps\downarrow0}\eps\log(\rho^\eps_t(x))\leq c_5-c_6\frac{\sfd^2(x,\supp(\rho_0))}{t}\qquad\forall x\in\X,\ t\in(0,1).
\end{equation}
Now let  $\eps_n\downarrow0$ be the sequence such that $\varphi^{\eps_n}_t,\psi^{\eps_n}_t$ converge to $\varphi_t,\psi_t$ as in Proposition \ref{lem:62} and put $S(x) := \sup_{\eps\in(0,1),t\in[0,1/2]}|\psi^\eps_t(x)| < \infty$ (recall \eqref{eq:contquadr}). The inequality
\[
\varphi_t(x)=\lim_{n\to\infty}\varphi^{\eps_n}_t(x)\leq \lims_{n\to\infty}\eps_n\log(\rho^{\eps_n}_t(x)) - \lim_{n\to\infty}\psi^{\eps_n}_t(x) \stackrel{\eqref{eq:altdens}}\leq S(x) +c_5-c_6\frac{\sfd^2(x,\supp(\rho_0))}{t}
\]
shows that if $x \notin \supp(\rho_0)$ we have $\varphi_0(x) = \lim_{t\downarrow0}\varphi_t(x) = -\infty$, as desired.

\noindent{\bf (iv)} By the point $(iii)$ just proven we have
\[
\int\varphi_0\rho_0\,\d\mm+\int\psi_1\rho_1\,\d\mm=-\int\psi_0\rho_0\,\d\mm-\int\varphi_1\rho_1\,\d\mm
\]
so that taking into account the weak continuity of $t \mapsto \mu_t$, the fact that the measures $\mu_t$ have equibounded supports and the  continuity of $t \mapsto \varphi_t$ (resp. $t \mapsto \psi_t$) for $t$ close to 1 (resp. close to 0) in the topology of local uniform convergence (direct consequence of the continuity in $L^1(\X,e^{-V}\mm)$ and the uniform local Lipschitz estimates provided by Proposition \ref{lem:62}), we get
\[
\begin{split}
\int\varphi_0\rho_0\,\d\mm+\int\psi_1\rho_1\,\d\mm & \stackrel{\phantom{\eqref{eq:63}}} = \lim_{t\downarrow0} - \int\psi_t\rho_t\,\d\mm - \int\varphi_{1-t}\rho_{1-t}\,\d\mm \\
& \stackrel{\eqref{eq:63}} = \lim_{t\downarrow0}\int\varphi_t\rho_t\,\d\mm - \int\varphi_{1-t}\rho_{1-t}\,\d\mm \stackrel{\eqref{eq:costointermedio}} = \frac12W_2^2(\mu_0,\mu_1).
\end{split}
\] 
\noindent{\bf (v)} Since $\rho_0 \in L^\infty(\X)$, we also have $\rho_0\log(\rho_0) \in L^\infty(\X)$. The claim then follows from the identity $\rho_0\varphi^\eps_0 = \eps\rho_0\log\rho_0 - \rho_0\psi^\eps_0$, the compactness of $\supp(\rho_0)$, the local uniform convergence of $\psi^{\eps_n}_0$ to $\psi_0$ as $n\to\infty$ and the fact that $\psi_0 = -\varphi_0$ on $\supp(\rho_0)$.
\end{proof}

\subsection{Proof of the main theorem}

We start with the following simple continuity statement:

\begin{Lemma}\label{lem:percon} 
With the same assumptions and notation as in Setting \ref{set}, let $t\mapsto\mu_t=\rho_t\mm$ be the $W_2$-geodesic from $\mu_0$ to $\mu_1$ and  $(\varphi_t)_{t\in(0,1]}$ and $(\psi_t)_{t\in[0,1)}$ any couple of limit functions given by Proposition \ref{lem:62}.

Then the maps
\[
\begin{array}{rll}
(0,1] \ni t \quad & \mapsto \quad \rho_t\,\d \varphi_t &\in L^2(T^*\X)\\
\ [0,1) \ni t \quad & \mapsto \quad \rho_t\,\d \psi_t &\in L^2(T^*\X)\\
(0,1] \ni t \quad & \mapsto \quad \rho_t\,\d \varphi_t \otimes \d\varphi_t&\in L^2((T^*)^{\otimes2}\X)\\
\ [0,1) \ni t \quad & \mapsto \quad \rho_t\,\d \psi_t \otimes \d\psi_t&\in L^2((T^*)^{\otimes2}\X)
\end{array}
\]
are all continuous w.r.t.\ the strong topologies.
\end{Lemma}

\begin{proof}
By Lemma \ref{lem:8} we know that for any $p < \infty$ we have $\rho_s\to\rho_t$ in $L^p(\X)$ as $s\to t$ and thus in particular $\sqrt{\rho_s}\to\sqrt{\rho_t}$ as $s\to t$. Moreover, the compactness of the supports of $\rho_0$ and $\rho_1$ implies that there exist $\bar{x} \in \X$ and $R > 0$ such that $\supp(\rho_t) \subset B_R(\bar{x})$ for all $t \in [0,1]$. Consider a Lipschitz cut-off function $\nchi$ with support in $B_{R+1}(\bar{x})$ such that $\nchi \equiv 1$ in $B_R(\bar{x})$. The closure of the differential and the fact that $\varphi_s \to \varphi_t$ weakly in $W^{1,2}(\X,e^{-V}\mm)$ as $s\to t>0$ (as a consequence of $(\varphi_t) \in C((0,1],C(\X,e^{-V}))\cap L^\infty_{loc}((0,1),W^{1,2}(\X,e^{-V}\mm))$, see Proposition \ref{pro:9} and the notation therein) grant that $\d\varphi_s \to \d\varphi_t$ weakly in $L^2(T^*\X,e^{-V}\mm)$ and thus $\nchi\d\varphi_s \to \nchi\d\varphi_t$ in $L^2(T^*\X)$. Together with the previous claim about the densities, the fact that the latter are uniformly bounded in $L^\infty(\X)$ and how $\nchi$ is constructed, this is sufficient to conclude that $t \mapsto \sqrt{\rho_t}\d\varphi_t\in L^2(T^*\X)$ is weakly continuous.

We now claim that $t\mapsto \sqrt{\rho_t}\d\varphi_t\in L^2(T^*\X)$ is strongly continuous and to this aim we show that their $L^2(T^*\X)$-norms are constant. To see this, recall that by Proposition \ref{pro:9} we know that for $t\in(0,1]$ the function $-(1-t)\psi_t$ is a Kantorovich potential from $\mu_t$ to $\mu_1$ while from \eqref{eq:63} and the locality of the differential we get that $|\d \varphi_t|=|\d\psi_t|$ $\mu_t$-a.e., thus by point $(iii)$ in Theorem \ref{thm:bm} we have that
\[
\int|\d\varphi_t|^2\rho_t\,\d\mm=\frac1{(1-t)^2}\int |\d(1-t)\psi_t|^2\rho_t\,\d\mm=\frac1{(1-t)^2}W_2^2(\mu_t,\mu_1)=W_2^2(\mu_0,\mu_1).
\]
Multiplying the $\sqrt{\rho_t}\d\varphi_t$'s by $\sqrt{\rho_t}$ and using again the $L^2(\X)$-strong continuity of $\sqrt{\rho_t}$ and the uniform $L^\infty(\X)$-bound we conclude that $t \mapsto {\rho_t}\d\varphi_t\in L^2(T^*\X)$ is strongly continuous, as desired.

To prove the strong continuity of $t \mapsto \rho_t\,\d \varphi_t\otimes\d\varphi_t\in  L^2((T^*)^{\otimes2}\X)$ we argue as in Corollary \ref{cor:convd}: the strong continuity of $t\mapsto \sqrt{\rho_t}\d\varphi_t\in L^2(T^*\X)$ and the fact that these are, locally in $t\in(0,1]$, uniformly bounded (thanks again to $\supp(\rho_t) \subset B_R(\bar{x})$ for all $t \in [0,1]$), grant both that $t\mapsto \|\rho_t\d\varphi_t\otimes\d\varphi_t\|_{L^2((T^*)^{\otimes 2}\X)}$ is continuous and that $t\mapsto \rho_t\d\varphi_t\otimes\d \varphi_t\in L^2((T^*)^{\otimes 2}\X)$ is weakly continuous.

The claims about the $\psi_t$'s follow in the same way.
\end{proof}

We now have all the tools needed to prove our main result. Notice that we shall not make explicit use of Theorem \ref{thm:1i} but rather reprove it for (the restriction to $[\delta,1-\delta]$ of) entropic interpolations.

\begin{Theorem}\label{thm:main}
Let $(\X,\sfd,\mm)$ be an $\RCD^*(K,N)$ space with $K \in \R$ and $N \in [1,\infty)$. Let $\mu_0,\mu_1 \in \probt \X$ be such that $\mu_0,\mu_1\leq C\mm$ for some $C>0$, with compact supports and let $(\mu_t)$ be the unique $W_2$-geodesic connecting $\mu_0$ to $\mu_1$. Also, let $h\in H^{2,2}(\X)$.

Then the map 
\[
[0,1]\ni \ t\quad \mapsto\quad \int h\,\d\mu_t\ \in\R
\]
belongs to $C^2([0,1])$ and the following formulas hold for every $t\in[0,1]$:
\begin{equation}
\label{eq:derivate}
\begin{split}
\ddt \int h\,\d\mu_t & = \int\la\nabla h,\nabla\phi_t\ra\,\d\mu_t,\\
\frac{\d^2}{\d t^2}\int h\,\d\mu_t & = \int \H h(\nabla\phi_t,\nabla\phi_t)\,\d\mu_t,
\end{split}
\end{equation}
where $\phi_t$ is any function such that for some $s\neq t$, $s\in[0,1]$, the function $-(s-t)\phi_t$ is a Kantorovich potential from $\mu_t$ to $\mu_s$.
\end{Theorem}

\begin{proof}
For the given $\mu_0,\mu_1$ introduce the notation of Setting \ref{set} and then find $\eps_n\downarrow0$ such that $(\varphi^{\eps_n}_t),(\psi^{\eps_n}_t)$ converge to limit curves $(\varphi_t),(\psi_t)$ as in Proposition \ref{lem:62}. 

By Lemma \ref{lem:gradpot} we know that the particular choice of the $\phi_t$'s as in the statement does not affect the right-hand sides in \eqref{eq:derivate}, we shall therefore prove that such formulas hold for the choice $\phi_t:=\psi_t$, which is admissible thanks to Proposition \ref{pro:9} whenever $t<1$. The case $t=1$ can be achieved swapping the roles of $\mu_0,\mu_1$ or, equivalently, with the choice $\phi_t = -\varphi_t$ which is admissible for $t>0$. 

Fix $h\in H^{2,2}(\X)$ with compact support and for $t\in[0,1]$ set
\[
\hh_n(t):=\int h\,\d\mu^{\eps_n}_t\qquad\qquad \hh(t):=\int h\,\d\mu_t.
\]
The bound \eqref{eq:linftybound} grants that the $\hh_n$'s are uniformly bounded and the convergence in \eqref{eq:66} that $\hh_n(t)\to \hh(t)$ for any $t\in[0,1]$.

Since $(\rho^{\eps_n}_t)\in AC_{loc}((0,1),W^{1,2}(\X))$ we have that $\hh_n\in AC_{loc}((0,1))$ and, recalling the formula for $\ddt\rho^\eps_t$ given by Proposition \ref{pro:7}, that
\begin{equation}
\label{eq:der1}
\ddt \hh_n(t) = \int h\ddt\rho^{\eps_n}_t\,\d\mm = -\int h\,{\rm div}(\rho^{\eps_n}_t\nabla\vartheta^{\eps_n}_t) \d\mm = \int \la\nabla h,\nabla\vartheta_t^{\eps_n}\ra\rho^{\eps_n}_t\,\d\mm.
\end{equation}
The fact that $\vartheta_t = \frac{\psi_t-\varphi_t}2$, the compactness of $\supp(h)$ and the bounds \eqref{eq:linftybound} and \eqref{eq:lipcontr} ensure that $\big|\ddt \hh_n(t)\big|$ is uniformly bounded in $n$ and $t\in[t_0,t_1]\subset(0,1)$ and the compactness of $\supp(h)$ also allows us to use the convergence properties \eqref{eq:limited} and \eqref{eq:66}, which grant that 
\[
\iint_{t_0}^{t_1} \la\nabla h,\nabla\vartheta_t^{\eps_n}\ra\rho^{\eps_n}_t\,\d t\,\d\mm\quad\to\quad\iint_{t_0}^{t_1}  \la\nabla h,\nabla\vartheta_t\ra\rho_t\,\d t\,\d\mm.
\]
This is sufficient to pass  to the limit in the distributional formulation of $\ddt \hh_n(t)$ and taking into account that $\hh\in C([0,1])$ we have just proved that $\hh\in AC_{loc}((0,1))$ with
\begin{equation}
\label{eq:der11}
\ddt \hh(t)=\int \la\nabla h,\nabla\vartheta_t\ra\rho_t\,\d\mm
\end{equation}
for a.e.\ $t\in[0,1]$. Recalling that $\vartheta_t=\frac{\psi_t-\varphi_t}2$, \eqref{eq:63} and the locality of the differential we see that
\begin{equation}
\label{eq:gradt}
\nabla\vartheta_t=\nabla\psi_t\quad\rho_t\mm\ae\qquad\forall t\in[0,1),
\end{equation} 
and thus by Lemma \ref{lem:percon} we see that the right-hand side of \eqref{eq:der11} is continuous in $t\in[0,1)$, which then implies  that $\hh\in C^1([0,1))$ and that the first in \eqref{eq:derivate} holds for any $t \in [0,1)$.

For the second derivative we assume in addition that $h \in \testi\X$. Then we recall that $(\rho^{\eps_n}_t) \in AC_{loc}((0,1),W^{1,2}(\X))$ and $(\vartheta_t^{\eps_n}) \in AC_{loc}((0,1),W^{1,2}(\X,e^{-V}\mm))$ with $V = M\sfd^2(\cdot,\bar{x})$ for some $\bar{x} \in \X$ and $M>0$ sufficiently large. Consider the rightmost side of \eqref{eq:der1} to get that $\ddt I_n(t) \in AC_{loc}((0,1))$ and
\[
\frac{\d^2}{\d t^2}\hh_n(t) = \int\langle\nabla h,\nabla\ddt \vartheta^{\eps_n}_t\rangle\rho^{\eps_n}_t+\la\nabla h,\nabla \vartheta^{\eps_n}_t\ra\ddt\rho^{\eps_n}_t\,\d\mm
\]
for a.e.\ $t$, so that defining the `acceleration' $a^\eps_t$ as
\[
a^\eps_t := -\Big(\frac{\eps^2}{4} \Delta\log\rho^\eps_t + \frac{\eps^2}{8}|\nabla\log\rho^\eps_t|^2\Big)
\]
and recalling the formula for $\ddt\vartheta^\eps_t$ given by Proposition \ref{pro:7} we have
\[
\begin{split}
\frac{\d^2}{\d t^2}\hh_n(t) &= \int\langle\nabla h,\nabla\Big(-\frac12|\nabla\vartheta^{\eps_n}_t|^2 +  a^{\eps_n}_t\Big)\rangle\rho^{\eps_n}_t - \la\nabla h,\nabla \vartheta^{\eps_n}_t\ra{\rm div}(\rho^{\eps_n}_t\nabla\vartheta^{\eps_n}_t)\,\d \mm\\
& = \int\Big(-\frac12\langle\nabla h,\nabla |\nabla\vartheta^{\eps_n}_t|^2\rangle + \langle\nabla(\la\nabla h,\nabla \vartheta^{\eps_n}_t\ra),\nabla\vartheta^{\eps_n}_t\rangle + \la\nabla h,\nabla a^{\eps_n}_t\ra\Big)\rho^{\eps_n}_t\,\d\mm\\
(\text{by \eqref{eq:leibh}})\quad & = \int \H h(\nabla\vartheta^{\eps_n}_t,\nabla\vartheta^{\eps_n}_t)\rho^{\eps_n}_t\,\d\mm - \int \big(\Delta h + \la\nabla h,\nabla\log\rho^{\eps_n}_t\ra\big) a^{\eps_n}_t\rho^{\eps_n}_t\,\d\mm.
\end{split}
\]
Since $\vartheta^\eps_t = \frac{\psi^\eps_t - \varphi^\eps_t}2$ and $\H h\in L^2(T^{*\otimes 2}\X)$ with compact support, by    \eqref{eq:limited} and   \eqref{eq:66}   we see that 
\[
\int \H h(\nabla\vartheta^{\eps_n}_t,\nabla\vartheta^{\eps_n}_t)\rho^{\eps_n}_t\,\d\mm\quad\stackrel{n\to\infty}\to \quad\int \H h(\nabla\vartheta_t,\nabla\vartheta_t)\rho_t\,\d\mm\qquad\text{ in $L^1_{loc}(0,1)$}
\]
and since $|\nabla h|,\Delta h\in L^\infty(\X)$, by Lemma \ref{lem:vanish} we deduce that
\[
\int \big(\Delta h + \la\nabla h,\nabla\log\rho^{\eps_n}_t\ra\big) a^{\eps_n}_t\rho^{\eps_n}_t\,\d\mm\quad\to\quad 0\qquad\text{ in $L^1_{loc}(0,1)$}.
\]
Hence we can pass to the limit in the distributional formulation of $\frac{\d^2}{\d t^2}\hh_n$ to obtain that $\ddt I\in AC_{loc}((0,1))$ and
\begin{equation}
\label{eq:der2}
\frac{\d^2}{\d t^2}\hh(t)=\int \H h(\nabla\vartheta_t,\nabla\vartheta_t)\rho_t\,\d\mm
\end{equation}
for a.e.\ $t$. Using again \eqref{eq:gradt} and Lemma \ref{lem:percon} we conclude that the right-hand side of \eqref{eq:der2} is continuous on $[0,1)$, so that $\hh\in C^2([0,1))$ and the second in \eqref{eq:derivate} holds for every $t\in[0,1)$.

It remains to remove the assumption that $h \in \testi\X$ and has compact support. To this aim we claim that functions in $\testi\X$ with compact support are dense in $H^{2,2}(\X)$. To see this, let $\nchi_R$ be as in Lemma \ref{lem:cutoff} and notice that the Leibniz rules for the gradient and the Laplacian easily give that $\nchi_Rh\to h$ in $W^{1,2}(\X)$ and $\Delta (\nchi_Rh)\to\Delta h$ in $L^2(\X)$ as $R\to\infty$ for every $h\in\testi\X$. Hence by \eqref{eq:heslap} we also have $\nchi_Rh\to h$ in $H^{2,2}(\X)$. Taking into account that $\testi\X$ is dense in $H^{2,2}(\X)$ (recall \eqref{eq:testidenso}), our claim follows.

Now let $h\in H^{2,2}(\X)$ be arbitrary and $(h_k)\subset\testi\X $ with bounded support be $H^{2,2}$-converging to $h$. Notice that we can choose the $\phi_t$'s to be uniformly Lipschitz (e.g.\ by taking $\phi_t := \psi_t$ for $t \geq 1/2$, $\phi_t := -\varphi_t$ for $t<1/2$ and recalling Proposition \ref{pro:1} and using a cut-off argument). The uniform $L^\infty$ estimates \eqref{eq:linftyrcd}, the fact that all the densities $\rho_t$ are supported in a compact set $B$ independent of $t \in [0,1]$ and the $L^2$-convergence of ${h_k},\nabla {h_k},\H {h_k}$ to ${h},\nabla {h},\H {h}$ respectively grant that as $k \to \infty$ we have that
\[
\begin{split}
\int h_k\,\d\mu_t&\qquad\to\qquad\int h\,\d\mu_t\\
\int\la\nabla h_k,\nabla\phi_t\ra\,\d\mu_t&\qquad\to\qquad \int\la\nabla h,\nabla\phi_t\ra\,\d\mu_t\\
\int \H {h_k}(\nabla\phi_t,\nabla\phi_t)\,\d\mu_t&\qquad\to\qquad \int \H {h}(\nabla\phi_t,\nabla\phi_t)\,\d\mu_t
\end{split}
\]
uniformly in $t\in[0,1]$. This is sufficient to pass to the limit from the formulas \eqref{eq:derivate} for the $h_k$'s to the one for $h$ and also ensures the $C^2$ regularity of $t\mapsto \int h\,\d\mu_t$. 
\end{proof}

\subsection{Related differentiation formulas}

In this last part we collect some direct consequences of Theorem \ref{thm:main}. For the notion of covariant derivative and the related calculus rules we refer to \cite{Gigli14}.
\begin{Theorem}
Let $(\X,\sfd,\mm)$ be an $\RCD^*(K,N)$ space with $K \in \R$ and $N \in [1,\infty)$. Then the following holds:
\begin{itemize}
\item[(i)] Let $\ppi$ be an optimal geodesic test plan with bounded support and $h \in H^{2,2}(\X)$. Then the map $[0,1]\ni t\mapsto h\circ\e_t\in L^2(\ppi)$ is in $C^2([0,1],L^2(\ppi))$ and we have
\[
\begin{split}
\frac{\d}{\d t} \big( h\circ\e_t\big)&=\la\nabla h,\nabla\phi_t\ra \circ\e_t,\\
\frac{\d^2}{\d t^2} \big( h\circ\e_t\big)&=\H{h}(\nabla\phi_t,\nabla\phi_t)\circ\e_t,
\end{split}
\]
for every $t \in [0,1]$, where $\phi_t$ is any function such that for some $s \neq t$, $s \in [0,1]$, the function $-(s-t)\phi_t$ is a Kantorovich potential from $(\e_t)_*\ppi$ to $(\e_s)_*\ppi$.
\item[(ii)] Let $\ppi$ be an optimal geodesic test plan with bounded support and $W \in H^{1,2}_C(\X)$. Then the map $[0,1] \ni t \mapsto \la W,\nabla \phi_t\ra \circ \e_t \in L^2(\ppi)$ is in $C^1([0,1],L^2(\ppi))$ and we have
\[
\ddt \big(  \la W,\nabla \phi_t\ra\circ\e_t\big) =\big( \nabla W : (\nabla\phi_t \otimes \nabla\phi_t)\big )\circ\e_t,
\]
for every $t \in [0,1]$, where $\phi_t$ is as in $(i)$.
\item[(iii)] Let $\mu_0,\mu_1,\phi_t$ be as in Theorem \ref{thm:main} and $W \in H^{1,2}_C(T\X)$. Then the map 
\[
[0,1] \ni \ t \quad \mapsto \quad \int \langle W,\nabla\phi_t\rangle \d\mu_t\ \in\R
\]
belongs to $C^1([0,1])$ and the following formula holds for every $t \in [0,1]$
\[
\ddt \int \langle W,\nabla\phi_t\rangle \d\mu_t = \int  \nabla W : (\nabla\phi_t \otimes \nabla\phi_t) \, \d\mu_t.
\]
\end{itemize}
\end{Theorem}
\begin{proof}

\noindent{\bf (i)} Start observing that for $\ppi$ as in the assumptions and $\Gamma\subset C([0,1],\X)$ Borel with $\ppi(\Gamma)>0$, the curve $t\mapsto \ppi(\Gamma)^{-1}(\e_t)_*\ppi\restr\Gamma$ fulfills the assumptions of Theorem \ref{thm:main} with the same $\phi_t$'s as in the current setting (and in particular, the $\phi_t$'s can be chosen independently of $\Gamma$). Then, taking into account Lemma \ref{lem:percon} it is easy to check that the maps $[0,1]\ni t\mapsto h\circ\e_t,\la\nabla h,\nabla\phi_t\ra\circ\e_t,(\H{h}(\nabla\phi_t,\nabla\phi_t)\circ\e_t)\in L^2(\ppi)$ are all continuous, and in particular with uniformly, in $t\in[0,1]$, bounded $L^2(\ppi)$-norms.

Also, Theorem \ref{thm:main} applied to $t\mapsto \ppi(\Gamma)^{-1}(\e_t)_*\ppi\restr\Gamma$ gives, after integration and an application of Fubini's theorem, that for every $t,s\in[0,1]$, $t<s$ we have
\[
\begin{split}
\int \eta\big(h\circ\e_s-h\circ\e_t\big)\,\d\ppi&=\int\eta\int_t^s\la\nabla h,\nabla\phi_r\ra\circ\e_r\,\d r\,\d\ppi\\
\int \eta\big(\la\nabla h,\nabla\phi_s\ra\circ\e_s-\la\nabla h,\nabla\phi_t\ra\circ\e_t\big)\,\d\ppi&=\int\eta\int_t^s\H{h}(\nabla\phi_r,\nabla\phi_r)\circ\e_r\,\d r\,\d\ppi\\
\end{split}
\]
for every $\eta$ of the form $\eta=\nchi_\Gamma$ with $\Gamma\subset C([0,1],\X)$ Borel, where here and below the integral are intended in the Bochner sense. Then the fact that the linear span of such $\eta$'s is dense in $L^2(\ppi)$ forces the equalities
\[
\begin{split}
 h\circ\e_s-h\circ\e_t&=\int_t^s\la\nabla h,\nabla\phi_r\ra\circ\e_r\,\d r\\
\la\nabla h,\nabla\phi_s\ra\circ\e_s-\la\nabla h,\nabla\phi_t\ra\circ\e_t&=\int_t^s\H{h}(\nabla\phi_r,\nabla\phi_r)\circ\e_r\,\d r
\end{split}
\]
which is the claim.

\noindent{\bf (ii)} By $(i)$ and the Leibniz rule for the covariant derivative (see \cite{Gigli14}) we see that the claim holds for $W=\sum_{i=1}^nf_i\nabla g_i$, with $n\in\N$ and $(f_i),(g_i)\in\testi\X$. These vector fields are dense in $H^{1,2}_C(T\X)$, hence the claim follows noticing that if $W_n\to W$ in $H^{1,2}_C(T\X)$ and the $\phi_t$'s are chosen uniformly Lipschitz (which as discussed in the proof of Theorem \ref{thm:main} is always admissible) then $\la W_n,\nabla\phi_t\ra \to \la W,\nabla\phi_t\ra $ and $\nabla W_n : (\nabla\phi_t \otimes \nabla\phi_t)\to \nabla W : (\nabla\phi_t \otimes \nabla\phi_t)$ in $L^2(\X)$ as $n\to\infty$. Therefore, since $(\e_t)_*\ppi\leq C\mm$ for every $t\in[0,1]$ and some $C>0$, we have that
\[
\begin{split}
\la W_n,\nabla\phi_t\ra \circ\e_t&\to \la W,\nabla\phi_t\ra\circ\e_t\\
\big(\nabla W_n : (\nabla\phi_t \otimes \nabla\phi_t)\big)\circ\e_t&\to\big( \nabla W : (\nabla\phi_t \otimes \nabla\phi_t)\big)\circ \e_t
\end{split}
\]
in $L^2(\ppi)$ uniformly in $t\in[0,1]$. The conclusion follows.

\noindent{\bf (iii)} Direct consequence of $(ii)$ and an integration w.r.t.\ $\ppi$.
\end{proof}

\appendix

\section{Reminders about analysis on $\RCD$ spaces}\label{sec:2}
In this appendix we recall the basic definitions and properties of the various objects that we used in the body of the paper. We also provide detailed bibliographical references.

\subsection{Sobolev calculus on \texorpdfstring{$\RCD$}{RCD} spaces}

By $C([0,1],(\X,\sfd))$, or simply $C([0,1],\X)$, we denote the space of continuous curves with values on the metric space $(\X,\sfd)$ and for $t\in[0,1]$ the {\bf evaluation map} $\e_t:C([0,1],(\X,\sfd))\to \X$ is defined as $\e_t(\gamma):=\gamma_t$.  For the notion of {\bf absolutely continuous curve} in a metric space and of {\bf metric speed} see for instance Section 1.1 in \cite{AmbrosioGigliSavare08}. The collection of absolutely continuous curves on $[0,1]$ is denoted $AC([0,1],(\X,\sfd))$, or simply by $AC([0,1],\X)$.

By $\prob\X$ we denote the space of Borel probability measures on $(\X,\sfd)$ and by $\probt\X \subset \prob\X$ the subclass of those with finite second moment.

\medskip

Let $(\X,\sfd,\mm)$ be a complete and separable metric measure space endowed with a Borel non-negative measure which is finite on bounded sets. 

For the definition of {\bf test plans}, of the {\bf Sobolev class} $S^2(\X)$ and of {\bf minimal weak upper gradient} $|D f|$ see \cite{AmbrosioGigliSavare11} (and the previous works \cite{Cheeger00}, \cite{Shanmugalingam00} for alternative - but equivalent - definitions of Sobolev functions). The local counterpart of $S^2(\X)$ is introduced as follows: $L^2_{loc}(\X)$ is defined as the space of functions $f \in L^0(\X)$ such that for all compact set $\Omega \subset \X$ there exists a function $g \in L^2(\X)$ such that $f = g$ $\mm$-a.e.\ in $\Omega$ and the local Sobolev class $S^2_{loc}(\X)$ is then defined as
\begin{equation}
\label{eq:3}
S^2_{loc}(\X) := \{ f \in L^0(\X) \,:\, \forall \Omega \subset\subset \X \,\,\exists g \in S^2(\X) \textrm{ s.t. } f = g \,\,\mm\textrm{-a.e. in } \Omega\}.
\end{equation}
The local minimal weak upper gradient of a function $f \in S^2_{loc}(\X)$ is denoted by $|Df|$, omitting the locality feature, and defined for all $\Omega \subset\subset \X$ as $|Df| := |Dg|$ $\mm$-a.e.\ in $\Omega$, where $g$ is as in (\ref{eq:3}). The definition does depend neither on $\Omega$ nor on the choice of $g$ associated to it by locality of the minimal weak upper gradient.

The Sobolev space $W^{1,2}(\X)$ (resp.\ $W^{1,2}_{loc}(\X)$) is defined as $L^2(\X)\cap S^2(\X)$ (resp.\ $L^2_{loc} \cap S^2_{loc}(\X)$). When endowed with the norm $\|f\|_{W^{1,2}}^2:=\|f\|_{L^2}^2+\||Df|\|_{L^2}^2$, $W^{1,2}(\X)$ is a Banach space. The {\bf Cheeger energy} is the convex and lower-semicontinuous functional $E:L^2(\X)\to[0,\infty]$ given by
\[
E(f):=\left\{\begin{array}{ll}
\displaystyle{\frac12\int|D f|^2\,\d\mm}&\qquad \text{for }f\in W^{1,2}(\X)\\
+\infty&\qquad\text{otherwise}
\end{array}\right.
\]
$(\X,\sfd,\mm)$ is {\bf infinitesimally Hilbertian} (see \cite{Gigli12}) if $W^{1,2}(\X)$ is Hilbert. In this case $E$ is a Dirichlet form and its infinitesimal generator $\Delta$, which is a closed self-adjoint operator on $L^2(\X)$, is called {\bf Laplacian} on $(\X,\sfd,\mm)$ and its domain denoted by $D(\Delta)\subset W^{1,2}(\X)$. The flow $(\h_t)$ associated to $E$ is called {\bf heat flow} (see \cite{AmbrosioGigliSavare11}), and  for any $f\in L^2(\X)$ the curve $t\mapsto\h_tf\in L^2(\X)$ is continuous on $[0,\infty)$, locally absolutely continuous on $(0,\infty)$ and the only solution of
\[
\ddt\h_tf=\Delta\h_tf\qquad\h_tf\to f\text{ as }t\downarrow0.
\]
If moreover $(\X,\sfd,\mm)$ is an $\RCD(K,\infty)$ space (see \cite{AmbrosioGigliSavare11-2}), the following a priori estimates hold true for every $f \in L^2(\X)$ and $t > 0$:
\begin{subequations}
\begin{align}
\label{eq:apriori1}
E(\h_t f) & \leq \frac{1}{4t}\|f\|_{L^2(\X)}^2, \\
\label{eq:apriori2}
\|\Delta\h_t f\|_{L^2(\X)}^2 & \leq \frac{1}{2t^2}\|f\|_{L^2(\X)}^2.
\end{align}
\end{subequations}
Still within the $\RCD$ framework, there exists the {\bf heat kernel}, namely a function 
\begin{equation}
\label{eq:hk}
(0,\infty)\times \X^2\ni (t,x,y)\quad\mapsto\quad \hr_t[x](y)=\hr_t[y](x)\in (0,\infty)
\end{equation}
such that
\begin{equation}
\label{eq:rapprform}
\h_tf(x)=\int f(y)\hr_t[x](y)\,\d\mm(y)\qquad\forall t>0
\end{equation}
for every $f\in L^2(\X)$. For every $x\in \X$ and $t>0$, $r_t[x]$ is a probability density and thus \eqref{eq:rapprform} can be used to extend the heat flow to $L^1(\X)$ and shows that the  flow is {\bf mass preserving} and satisfies the {\bf maximum principle}, i.e.
\[
f\leq c\quad\mm-a.e.\qquad\qquad\Rightarrow \qquad\qquad\h_tf\leq c\quad\mm\ae,\ \forall t>0.
\]
For finite-dimensional $\RCD^*(K,N)$ spaces (introduced in \cite{Gigli12}) the heat kernel satisfies {\bf Gaussian estimates}, i.e.\ for every $\delta>0$ there are positive constants $C_1=C_1(K,N,\delta)$ and $C_2=C_2(K,N,\delta)$ such that for every $x,y\in \X$ and $t>0$ it holds
\begin{equation}
\label{eq:gaussest}
\frac{1}{C_1\mm(B_{\sqrt t}(y))}\exp\Big(-\frac{\sfd^2(x,y)}{(4-\delta)t} - C_2 t\Big)\leq \hr_t[x](y)\leq \frac{C_1}{\mm(B_{\sqrt t}(y))}\exp\Big(-\frac{\sfd^2(x,y)}{(4+\delta)t} + C_2 t\Big),
\end{equation}
see \cite{JLZ15}, which adapts the approach of \cite{Sturm92}, \cite{Sturm96II} to the $\RCD$ setting.

\bigskip

For general metric measure spaces, the {\bf differential} is a well defined linear map $\d$ from $S^2(\X)$ with values in the {\bf cotangent module} $L^2(T^*\X)$ (see \cite{Gigli14}) which is a closed operator when seen as unbounded operator on $L^2(\X)$. If $(\X,\sfd,\mm)$ is infinitesimally Hilbertian, which from now on we shall always assume, the cotangent module is canonically isomorphic to its dual, the {\bf tangent module} $L^2(T\X)$, and the isomorphism sends the differential $\d f$ to the gradient $\nabla f$. Replacing the language of $L^2$-normed modules with the $L^0$'s one (see \cite{Gigli14}), the differential can be extended to $\d : S^2_{loc}(\X) \to L^0(T^*\X)$, where $L^0(T^*\X)$ denotes the family of (measurable) 1-forms. The dual of $L^0(T^*\X)$ as $L^0$-normed module is denoted by $L^0(T\X)$, it is canonically isomorphic to $L^0(T^*\X)$ and its elements are called vector fields. With this said, $L^2_{loc}(T^*\X) \subset L^0(T^*\X)$ (resp.\ $L^2_{loc}(T\X) \subset L^0(T\X)$) is defined as the collection of the 1-forms $\omega$ such that $|\omega| \in L^2_{loc}(\X)$ (resp.\ the vector fields $v$ such that $|v| \in L^2_{loc}(\X)$).

The {\bf divergence} of a vector field is defined as (minus) the adjoint of the differential, i.e.\ we say that $v \in L^2(T\X)$ (resp.\ $v \in L^2_{loc}(T\X)$) has a divergence in $L^2(\X)$ (resp.\ in $L^2_{loc}(\X)$), and write $v\in D({\rm div})$ (resp.\ $v \in D({\rm div}_{loc})$), provided there is a function $g\in L^2(\X)$ (resp.\ $g \in L^2_{loc}(\X)$) such that
\[
\int fg\,\d\mm=-\int \d f(v)\,\d\mm\qquad\forall f\in W^{1,2}(\X).
\]
(resp.\ for all Lipschitz functions $f$ with bounded support). In this case $g$ is unique and is denoted ${\rm div}(v)$. A function $f \in W^{1,2}_{loc}(\X)$ has Laplacian in $L^2_{loc}(\X)$, and we shall write $f \in D(\Delta_{loc})$, if there exists $g \in L^2_{loc}(\X)$ such that
\[
\int \phi g\d\mm = -\int \langle \nabla\phi,\nabla f\rangle\d\mm, \qquad \forall \phi \textrm{ Lipschitz with bounded support}
\]
and in this case, since $g$ is unique, we set $\Delta f := g$. It can be verified that
\[
f\in D(\Delta_{loc})\ \text{ if and only if }\ \nabla f\in D({\rm div}_{loc})\text{ and in this case }\Delta f={\rm div}(\nabla f),
\]
in accordance with the smooth case.

\medskip

As regards the properties of $\d, {\rm div},\Delta$, the differential satisfies the following calculus rules which we shall use extensively without further notice:
\begin{align*}
|\d f|&=|D f|\quad\mm\ae&&\forall f\in S^2(\X)\\
\d f&=\d g\qquad\mm\ae\ \text{\rm on}\ \{f=g\} &&\forall f,g\in S^2(\X)\\
\d(\varphi\circ f)&=\varphi'\circ f\,\d f&&\forall f\in S^2(\X),\ \varphi:\R\to \R\ \text{Lipschitz}\\
\d(fg)&=g\,\d f+f\,\d g&&\forall f,g\in L^\infty\cap S^2(\X)
\end{align*}
where it is part of the properties the fact that $\varphi\circ f,fg\in S^2(\X)$ for $\varphi,f,g$ as above. For the divergence, the formula
\[
{\rm div}(fv)=\d f(v)+f{\rm div}(v)\qquad\forall f\in W^{1,2}(\X),\ v\in D({\rm div}),\ \text{such that}\ |f|,|v|\in L^\infty(\X)
\]
holds, where it is intended in particular that $fv\in D({\rm div})$ for $f,v$ as above, and for the Laplacian
\[
\begin{split}
\Delta(\varphi\circ f)&=\varphi''\circ f|\d f|^2+\varphi'\circ f\Delta f\\
\Delta(fg)&=g\Delta f+f\Delta g+2\la\nabla f,\nabla g\ra
\end{split}
\]
where in the first equality we assume that $f\in D(\Delta),\varphi\in C^2(\R)$ are such that $f,|\d f|\in L^\infty(\X)$ and $\varphi',\varphi''\in L^\infty(\R)$ and in the second that $f,g\in D(\Delta)\cap L^\infty(\X)$ and $|\d f|,|\d g|\in L^\infty(\X)$ and it is part of the claims that $\varphi\circ f,fg$ are in $D(\Delta)$. On $S^2_{loc}(\X)$ as well as on $D({\rm div}_{loc})$ and $D(\Delta_{loc})$ the same calculus rules hold with slight adaptations. For sake of information, we present the chain rule for differetial and Laplacian, as they will be widely exploited without further mention.

\begin{Lemma}[Calculus rules]\label{lem:loccalculus}
Let $(\X,\sfd,\mm)$ be an $\RCD^*(K,N)$ with $K \in \R$ and $N \in [1,\infty)$. Then:
\begin{itemize}
\item[(i)] for all $f \in S^2_{loc}(\X)$ and $\varphi : \R \to \R$ such that for all $K \subset\subset \X$ there exists $I_K \subset\subset \R$ in such a way that $\mathscr{L}^1(f(K) \setminus I_K) = 0$ and $\varphi\restr{I_K}$ is Lipschitz it holds
\[
\d(\varphi\circ f) = \varphi'\circ f,
\]
where it is part of the statement the fact that $\varphi\circ f \in S^2_{loc}(\X)$ for $\varphi,f$ as above; analogous statements hold for the gradient;
\item[(ii)] for all $f \in D(\Delta_{loc})$ and $\varphi : \R \to \R$ such that $f,|\d f| \in L^\infty_{loc}(\X)$ and $\varphi',\varphi'' \in L^\infty(\R)$ it holds
\[
\Delta(\varphi\circ f) = \varphi''\circ f|\d f|^2 + \varphi'\circ f\Delta f
\]
where it is part of the claims that $\varphi\circ f \in D(\Delta_{loc})$.
\end{itemize}
\end{Lemma}
The proof is based on the locality  of such differentiation operators and the analogous properties of their global counterparts defined on $S^2(\X),D(\Delta)$.

Beside this notion of $L^2$-valued Laplacian, we shall also need that of measure-valued Laplacian (\cite{Gigli12}). A function $f\in W^{1,2}(\X)$ is said to have measure-valued Laplacian, and in this case we write $f\in D(\bd)$, provided there exists a Borel (signed) measure $\mu$ whose total variation is finite on bounded sets and such that
\[
\int g\,\d\mu=-\int\la\nabla g,\nabla f\ra\,\d\mm,\qquad\forall g\text{ Lipschitz with bounded support}.
\]
In this case $\mu$ is unique and denoted $\bd f$. This notion is compatible with the previous one in the sense that
\[
f\in D(\bd),\ \bd f\ll\mm\text{ and }\frac{\d\bd f}{\d\mm}\in L^2(\X)\qquad\Leftrightarrow\qquad f\in D(\Delta)\text{ and in this case }\Delta f=\frac{\d\bd f}{\d\mm}.
\]

\bigskip

On $\RCD(K,\infty)$ spaces, the vector space of `test functions' (see \cite{Savare13})
\[
\testi\X := \Big\{ f \in D(\Delta) \cap L^{\infty}({\X}) \ :\ |\nabla f| \in L^{\infty}(\X),\ \Delta f \in L^{\infty}\cap W^{1,2}(\X) \Big\}
\]
 is an algebra dense in $W^{1,2}(\X)$ for which it holds
\begin{equation}
\label{eq:stabcomp}
\text{$f\in\testi\X$   and $\varphi\in C^\infty(\R)$}\qquad\Rightarrow\qquad\text{$\varphi\circ f\in\testi\X$.} 
\end{equation}

Combining  the Gaussian estimates on  $\RCD^*(K,N)$ spaces, $N<\infty$, with the results in \cite{Savare13} we see that 
\begin{equation}
\label{eq:regflow}
f\in L^2 \cap L^{\infty}(\X),\  t>0\qquad\Rightarrow\qquad \h_t(f)\in\testi \X.
\end{equation}
The fact that $\testi\X$ is an algebra  is based on the property
\begin{equation}
\label{eq:reggrad}
\begin{split}
f\in\testi\X \qquad \Rightarrow \qquad &|\d f|^2\in W^{1,2}(\X)\quad\text{ with }\\
&\int|\d(|\d f|^2)|^2\,\d\mm\leq \||\d f|\|^2_{L^\infty}\Big(\||\d f|\|_{L^2}\||\d \Delta f|\|_{L^2}+|K|\||\d f|\|_{L^2}^2\Big)
\end{split}
\end{equation}
and actually a further regularity property of test functions is that
\[
f\in\testi\X \qquad \Rightarrow \qquad |\d f|^2\in D(\bd),
\]
so that it is possible to introduce the {\bf measure-valued $\Gamma_2$ operator} (\cite{Savare13}) as
\[
\Ggamma_2(f):=\bd\frac{|\d f|^2}{2}-\la\nabla f,\nabla\Delta f\ra\mm\qquad\forall f\in\testi \X.
\]
By construction, the assignment $f\mapsto \Ggamma_2(f)$ is a quadratic form.

An important property of the heat flow on $\RCD(K,\infty)$ spaces is the {\bf Bakry-\'Emery contraction estimate} (see \cite{AmbrosioGigliSavare11-2}):
\begin{equation}
\label{eq:be}
|\d\h_tf|^2\leq e^{-2Kt}\h_t(|\d f|^2)\qquad\forall f\in W^{1,2}(\X),\ t\geq 0.
\end{equation}
We also recall that $\RCD(K,\infty)$ spaces have the {\bf Sobolev-to-Lipschitz} property (\cite{AmbrosioGigliSavare11-2}, \cite{Gigli13}), i.e.
\begin{equation}
\label{eq:sobtolip}
f\in W^{1,2}(\X),\ |\d f|\in L^\infty(\X)\qquad\Rightarrow\qquad \exists \tilde f=f\ \mm-a.e.\ \text{ with }\Lip(\tilde f)\leq\||\d f|\|_{L^\infty},
\end{equation}
and thus we shall typically identify Sobolev functions with bounded differentials with their Lipschitz representative; in particular this will be the case for functions in $\testi \X$.

\bigskip

A well-known consequence of lower Ricci curvature bounds (see e.g. \cite{Cheeger-Colding97I}, \cite{Cheeger-Colding97II}, \cite{Cheeger-Colding97III}) is the existence of `{\bf good cut-off functions}', typically intended as cut-offs with bounded Laplacian; for our purposes the following result will be sufficient:
\begin{Lemma}\label{lem:cutoff}
Let $(\X,\sfd,\mm)$ be an $\RCD^*(K,N)$ space with $K \in \R$ and $N \in [1,\infty)$. Then for all $R>0$ and $x \in \X$ there exists a function $\nchi_R : \X \to \R$ satisfying:
\begin{itemize}
\item[(i)] $0 \leq \nchi_R \leq 1$, $\nchi_R \equiv 1$ on $B_R(x)$ and $\supp(\nchi_R) \subset B_{R+1}(x)$;
\item[(ii)] $\nchi_R \in \testi\X$.
\end{itemize}
Moreover, there exist constants $C,C' > 0$ depending on $K,N$ only such that
\begin{equation}\label{eq:cutbound}
\| |\nabla\nchi_R| \|_{L^{\infty}(\X)} \leq C, \qquad\qquad \|\Delta\nchi_R\|_{L^{\infty}(\X)} \leq C'.
\end{equation}
\end{Lemma}
The proof can be obtained following verbatim the arguments given in Lemma 3.1 of \cite{Mondino-Naber14} (inspired by  \cite{AmbrosioMondinoSavare13-2}, see also  \cite{Gigli-Mosconi14} for an alternative approach): there the authors are interested in cut-off functions such that $\nchi \equiv 1$ on $B_R(x)$ and $\supp(\nchi) \subset B_{2R}(x)$: for this reason they fix $R>0$ and then claim that for all $x \in \X$ and $0<r<R$ there exists a cut-off function $\nchi$ satisfying $(i)$, $(ii)$ and \eqref{eq:cutbound} with $C,C'$ also depending on $R$. However, as far as one is concerned with cut-off functions $\nchi$ where the distance between $\{\nchi = 0\}$ and $\{\nchi = 1\}$ is always equal to 1, the proof of \cite{Mondino-Naber14} in the case $R=1$ applies and does not affect \eqref{eq:cutbound}.

\medskip

A direct consequence of the existence of such cut-off functions is that 
\begin{equation}
\label{eq:testlocal}
\begin{split}
\{ f \in L^2_{loc}(\X) \,:&\, \forall \Omega \subset\subset \X \,\,\exists g \in \testi\X \textrm{ s.t. } f = g \,\,\mm\textrm{-a.e. in } \Omega\}\\
&=\Big\{ f \in D(\Delta_{loc}) \cap L^{\infty}_{loc}({\X}) \,:\, |\nabla f| \in L^{\infty}_{loc}(\X),\, \Delta f \in W^{1,2}_{loc}(\X) \Big\}.
\end{split}
\end{equation}
Indeed the `$\subset$' inclusion is obvious, while for the opposite one if $f$ belongs to the second set and $\Omega \subset \X$ is a bounded open set, consider  a cut-off function $\nchi \in \testi\X$ with compact support and $\nchi \equiv 1$ on $\Omega$: it is clear that $\nchi f \in \testi\X$ and $\nchi f \equiv f$ on $\Omega$.  We shall call the set in \eqref{eq:testlocal} the space of \emph{local} test functions and denote it $\testl\X$.

The existence of the space of test functions and the language of $L^2$-normed $L^\infty$-modules allow to introduce the spaces $W^{2,2}(\X)$ and $W^{1,2}_C(T\X)$ as follows (see \cite{Gigli14}). We first consider the tensor product $L^2((T^*)^{\otimes 2}\X)$ of $L^2(T^*\X)$ with itself. The pointwise norm on such module is denoted $|\cdot|_\HS$ (as in the smooth case it coincides with the Hilbert-Schmidt one) and $:$ is the scalar product inducing it. Then we say that a function $f \in W^{1,2}(\X)$ belongs to $W^{2,2}(\X)$ provided there exists $A \in L^2((T^*)^{\otimes 2}\X)$ symmetric, i.e.\ such that $A(v_1,v_2) = A(v_2,v_1)$ $\mm$-a.e.\ for every $v_1,v_2\in L^2(T\X)$, for which it holds
\[
\int h A(\nabla g,\nabla g)\,\d\mm=\int-\la\nabla f,\nabla g\ra{\rm div}(h\nabla g)-h\langle\nabla f,\nabla\frac{|\nabla g|^2}2\rangle\,\d\mm \qquad\forall g,h\in\testi \X.
\]
In this case $A$ is unique, called {\bf Hessian} of $f$ and denoted by $\H f$. The space $W^{2,2}(\X)$ endowed with the norm
\[
\|f\|^2_{W^{2,2}(\X)}:=\|f\|^2_{L^2(\X)}+\|\d f\|^2_{L^2(T^*\X)}+\|\H f\|^2_{L^2((T^*)^{\otimes 2}\X)}
\]
is a separable Hilbert space which contains $\testi\X$ and in particular is dense in $W^{1,2}(\X)$. It is proved in \cite{Gigli14} that $D(\Delta) \subset W^{2,2}(\X)$ with
\begin{equation}
\label{eq:heslap}
\int|\H f|_{\HS}^2\,\d\mm\leq \int (\Delta f)^2-K|\nabla f|^2\,\d\mm\qquad\forall f\in D(\Delta).
\end{equation}
The space $H^{2,2}(\X)$ is defined as the closure of $D(\Delta)$ in $W^{2,2}(\X)$ and following the arguments in Proposition 4.3.18 in \cite{Gigli14} it is not difficult to see that
\begin{equation}
\label{eq:testidenso}
\text{$\testi\X$ is dense in $H^{2,2}(\X)$.}
\end{equation}
It is unknown whether $H^{2,2}(\X)=W^{2,2}(\X)$ or not. We recall that 
\begin{equation}
\label{eq:leibh} 
\d\la\nabla f,\nabla g\ra = \H f(\nabla g,\cdot)+\H g(\nabla f,\cdot)\qquad\forall f,g\in\testi \X
\end{equation}
and that the Hessian is a local operator, i.e.\ $\H f=\H g$ $\mm$-a.e.\ on $\{f=g\}$. Using this latter fact, for $f\in\testi\X$ we can define $\H f$ as the element in the $L^0$-completion of $L^2((T^*)^{\otimes 2}\X)$ defined by
\[
\H f:=\H g\quad\mm-a.e.\ on\ \{f=g\}\qquad\forall g\in\testi\X.
\]

\medskip

The {\bf Bochner inequality} on $\RCD(K,\infty)$ spaces takes the form of an inequality between measures (\cite{Gigli14} - see also the previous contributions \cite{Savare13}, \cite{Sturm14}):
\begin{equation}
\label{eq:bochhess}
\Ggamma_2(f)\geq \big(|\H f|^2_\HS+K|\d f|^2\big)\mm\qquad\forall f\in\testi \X,
\end{equation}
and if the space is $\RCD^*(K,N)$ for some finite $N$ it also holds (\cite{Erbar-Kuwada-Sturm13}, \cite{AmbrosioMondinoSavare13}):
\begin{equation}
\label{eq:bochlap}
\Ggamma_2(f)\geq\Big( \frac{(\Delta f)^2}N+K|\d f|^2\Big)\mm\qquad\forall f\in\testi \X.
\end{equation}
Notice that since the Laplacian is in general not the trace of the Hessian, the former does not trivially imply the latter (in connection to this, see \cite{Han14}).

\bigskip

As regards the geometric features of finite-dimensional $\RCD^*(K,N)$ spaces, we recall the {\bf Bishop-Gromov inequality} in the form we shall need (see \cite{Sturm06I}, \cite{Sturm06II}): for any $x \in \supp(\mm)$ and for any $0 < r \leq R < \infty$ it holds
\begin{equation}\label{eq:bishop}
\frac{\mm(B_r(x))}{\mm(B_R(x))} \geq \frac{\int_0^r \sinh(t\sqrt{-K/(N-1)})^{N-1}\d t}{\int_0^R \sinh(t\sqrt{-K/(N-1)})^{N-1}\d t} \qquad \frac{\mathfrak{s}_r(x)}{\mathfrak{s}_R(x)} \geq \bigg(\frac{\sinh(r\sqrt{-K/(N-1)})}{\sinh(R\sqrt{-K/(N-1)})}\bigg)^{N-1}
\end{equation}
(with standard adaptations and caveat if $K \geq 0$) where
\[
\mathfrak{s}_r(x) := \limsup_{\delta \downarrow 0}\frac{1}{\delta}\mm(\overline{B_{r+\delta}(x)} \setminus B_r(x)).
\]
A couple of interesting consequences are the following: $\mm$ is uniformly locally doubling with an explicit expression for the local doubling costant, i.e.\ for all $x \in \X$ and $r > 0$ it holds
\begin{equation}\label{eq:doubling}
\mm(B_{2r}(x)) \leq 2^N\cosh\Big(2\sqrt{\frac{-K}{N-1}}r\Big)^{N-1}\mm(B_r(x));
\end{equation}
and there exists a constant $C > 0$ such that the following volume growth condition is satisfied
\begin{equation}\label{eq:volgrowth}
\mm(B_r(x)) \leq Ce^{Cr}, \qquad \forall x \in \X,\, r > 0.
\end{equation}

\bigskip

We conclude the section recalling the notion of Regular Lagrangian Flow, introduced by Ambrosio-Trevisan in \cite{Ambrosio-Trevisan14} as the generalization to $\RCD$ spaces of the analogous concept existing on $\R^d$ as proposed by Ambrosio in \cite{Ambrosio04}:

\begin{Definition}[Regular Lagrangian Flow]
Given $(v_t)\in L^1([0,1],L^2(T\X))$, the function $F:[0,1]\times \X\to \X$ is  a Regular Lagrangian Flow for $(v_t)$ provided:
\begin{itemize}
\item[i)] $[0,1]\ni t \mapsto F_t(x)$ is continuous for every $x\in \X$
\item[ii)] for every $f\in\testi \X$ and $\mm$-a.e.\ $x$ the map $t\mapsto f(F_t(x))$ belongs to $W^{1,1}([0,1])$ and
\[
\ddt  f(F_t(x))=\d f(v_t)(F_t(x))\qquad {\rm a.e.}\ t\in[0,1].
\]
\item[iii)] it holds
\[
(F_t)_*\mm\leq C\mm\qquad\forall t\in[0,1]
\]
for some constant $C>0$.
\end{itemize}
\end{Definition}

In \cite{Ambrosio-Trevisan14} the authors prove that under suitable assumptions on the $v_t$'s, Regular Lagrangian Flows exist and are unique. We shall use the following formulation of their result (weaker than the one provided in \cite{Ambrosio-Trevisan14}):

\begin{Theorem}\label{thm:RLF}
Let $(\X,\sfd,\mm)$ be an $\RCD(K,\infty)$ space and $(v_t)\in L^1([0,1],L^2(T\X))$ be such that $v_t\in D({\rm div})$ for a.e.\ $t$ and
\[
{\rm div}(v_t) \in L^1([0,1],L^2(\X)) \qquad ({\rm div}(v_t))^- \in L^1([0,1],L^\infty(\X)).
\]
Then there exists a unique, up to $\mm$-a.e.\ equality, Regular Lagrangian Flow $F$ for $(v_t)$.

For such flow, the quantitative bound
\begin{equation}
\label{eq:quantm}
(F_t)_*\mm\leq \exp\Big(\int_0^1\|({\rm div}(v_t))^-\|_{L^\infty(\X)}\,\d t\Big)\mm
\end{equation}
holds for every $t\in[0,1]$ and for $\mm$-a.e.\ $x$ the curve $t\mapsto F_t(x)$ is absolutely continuous and its metric speed ${\rm ms}_t({F_{\cdot}}(x))$ at time $t$ satisfies
\begin{equation}
\label{eq:quants}
{\rm ms}_t({F_{\cdot}}(x))=|v_t|(F_t(x))\qquad {\rm a.e.}\ t\in[0,1].
\end{equation}
\end{Theorem}
To be precise,  \eqref{eq:quants} is not explicitly stated in \cite{Ambrosio-Trevisan14}; its proof is anyway not hard and can be obtained, for instance, following the arguments in \cite{Gigli14}. 

\subsection{Optimal transport on \texorpdfstring{$\RCD$}{RCD} spaces}

It is well known that on $\R^d$, curves of measures which are $W_2$-absolutely continuous are in correspondence with appropriate solutions of the {\bf continuity equation} (\cite{AmbrosioGigliSavare08}). It has been proved in \cite{GigliHan13} that the same connection holds on arbitrary metric measure spaces $(\X,\sfd,\mm)$, provided the measures are controlled by $C\mm$ for some $C>0$, the formulation of such result which we shall need is:

\begin{Theorem}[Continuity equation and $W_2$-AC\ curves]\label{thm:GH}
Let $(\X,\sfd,\mm)$ be infinitesimally Hilbertian, $(\mu_t)\subset \prob \X$ be weakly continuous and $t \mapsto X_t \in L^0(T\X)$ be a family of vector fields, possibly defined only for a.e.\ $t\in[0,1]$. Assume that the map $t \mapsto \int |X_t|^2\d\mu_t$ is Borel and:
\begin{subequations}
\begin{align}
\label{eq:bih}
\mu_t&\leq C\mm\qquad\text{ $\forall t\in[0,1]$ for some $C>0$}\\
\label{eq:bkh}
\int_0^1\int |X_t|^2\,\d\mu_t\,\d t&<\infty
\end{align}
\end{subequations}
and that  the continuity equation
\[
\ddt\mu_t+{\rm div}(X_t\mu_t)=0,
\]
is satisfied in the following sense: for any $f\in W^{1,2}(\X)$ the map $[0,1]\ni t\mapsto\int f\,\d\mu_t$ is absolutely continuous and it holds
\[
\ddt\int f\,\d\mu_t=\int \d f(X_t)\,\d\mu_t\qquad {\rm a.e.}\ t.
\]
Then $(\mu_t)\in AC([0,1],(\prob \X,W_2))$ and
\[
|\dot\mu_t|^2=\int|X_t|^2\,\d\mu_t\qquad{\rm a.e.}\ t\in[0,1].
\]
\end{Theorem}

Recall that given $f : \X \to \R$ the  upper and lower slopes $|D^+f|,|D^-f| : \X \to [0,\infty]$ are defined as 0 on isolated points and otherwise
\[
|D^+f|(x):=\lims_{y\to x}\frac{(f(y)-f(x))^+}{\sfd(x,y)}\qquad\qquad|D^-f|(x):=\lims_{y\to x}\frac{(f(y)-f(x))^-}{\sfd(x,y)}.
\]
Similarly, the {\bf local Lipschitz constant} $\lip(f):\X\to[0,\infty]$ is defined as 0 on isolated points and otherwise as
\[
\lip f(x) :=\max\{|D^+f|(x),|D^-f|(x)\}= \limsup_{y \to x}\frac{|f(x) - f(y)|}{d(x,y)}.
\]
We also recall that the $c$-transform $\varphi^c:X\to\R\cup\{-\infty\}$ of  a function $\varphi:\X\to\R\cup\{-\infty\}$ is defined as
\[
\varphi^c(x):=\inf_{y\in X}\frac{\sfd^2(x,y)}{2}-\varphi(y)
\]
and that $\varphi$ is said to be {\bf $c$-concave} provided $\varphi=\psi^c$ for some $\psi$. Also, given $\mu_0,\mu_1\in\probt X$, a function $\varphi:\X\to\R\cup\{-\infty\}$ is called {\bf Kantorovich potential} from $\mu_0$ to $\mu_1$ provided it is $c$-concave and
\[
\int\varphi\,\d\mu_0+\int\varphi^c\,\d\mu_1=\frac12W_2^2(\mu_0,\mu_1).
\]
It is worth recalling that on general complete and separable metric spaces $(X,\sfd)$ we have that for $\mu_0,\mu_1\in\prob \X$ with bounded support there exists a Kantorovich potential from $\mu_0$ to $\mu_1$ which is Lipschitz and bounded.

This can be obtained starting from an arbitrary Kantorovich potential $\psi$ and then defining
\[
\varphi(x):=\min\Big\{C,\inf_{y\in X}\frac{\sfd^2(x,y)}2-\psi^c(y)\Big\}
\]
for $C$ sufficiently big.

With this said, we recall the following version of   Brenier-McCann theorem on $\RCD$ spaces ($(i)$ comes from \cite{Gigli12a} and \cite{RajalaSturm12}, $(ii)$ from \cite{AmbrosioGigliSavare11-2} and \cite{Gigli12}, $(iii)$ from \cite{AmbrosioGigliSavare11} and  $(iv)$ from \cite{GigliRajalaSturm13}).

\begin{Theorem}\label{thm:bm}
Let $(\X,\sfd,\mm)$ be an $\RCD(K,\infty)$ space and $\mu_0,\mu_1 \in \probt\X$ with bounded support and such that $\mu_0,\mu_1 \leq C\mm$ for some $C>0$. Also, let $\varphi$ be a Kantorovich potential for the couple $(\mu_0,\mu_1)$ which is locally Lipschitz on a neighbourhood of $\supp(\mu_0)$. Then:
\begin{itemize}
\item[i)] There exists a unique geodesic $(\mu_t)$ from $\mu_0$ to $\mu_1$, it satifies
\begin{equation}
\label{eq:linftyrcd}
\mu_t\leq C'\mm\qquad\forall t\in[0,1]\text{ for some }C'>0
\end{equation}
and there is a unique \emph{lifting} $\ppi$ of it, i.e.\ a unique measure $\ppi\in\prob{C([0,1],X)}$ such that $(\e_t)_*\ppi=\mu_t$ for every $t\in[0,1]$ and $\iint_0^1|\dot\gamma_t|^2\,\d t\,\d\ppi(\gamma)=W_2^2(\mu_0,\mu_1)$.
\item[ii)] For every $f\in W^{1,2}(\X)$ the map $t\mapsto \int f\,\d\mu_t$ is differentiable at $t=0$ and
\[
\ddt\int f\,\d\mu_t\restr{t=0}=-\int \d f(\nabla\varphi)\,\d\mu_0.
\]
\item[iii)] The identity 
\[
|\d\varphi|(\gamma_0)=|D^+\varphi|(\gamma_0)=\sfd(\gamma_0,\gamma_1)
\]
holds for $\ppi$-a.e.\ $\gamma$.
\item[iv)] If the space is $\RCD^*(K,N)$ for some $N<\infty$, then $(i),(ii),(iii)$ holds with $\mu_1$ only assumed to be with bounded support, with the caveat that \eqref{eq:linftyrcd} holds in the form: for every $\delta\in(0,1/2)$ there is $C_\delta>0$ so that  $\mu_t\leq C'_\delta\mm$ for every $t\in[0,1-\delta]$.
\end{itemize}
\end{Theorem}

A property related to the above is the fact that although the Kantorovich potentials are not uniquely determined by the initial and final measures, their gradients are. This is expressed by the following result, which also says that if we sit in the intermediate point of a geodesic and move to one extreme or the other, then the two corresponding velocities are one the opposite of the other (see Lemma 5.8 and Lemma 5.9 in \cite{Gigli13} for the proof):

\begin{Lemma}\label{lem:gradpot}
Let $(\X,\sfd,\mm)$ be an $\RCD(K,\infty)$ space with $K \in \R$ and $(\mu_t) \subset \probt\X$ a $W_2$-geodesic such that $\mu_t \leq C\mm$ for every $t \in [0,1]$ for some $C>0$. For $t \in [0,1]$ let $\phi_t,\phi_t' : \X \to \R$ be locally Lipschitz functions such that for some $s,s'\neq t$ the functions $-(s-t)\phi_t$ and $-(s'-t)\phi_t'$ are Kantorovich potentials from $\mu_t$ to $\mu_s$ and from $\mu_t$ to $\mu_{s'}$ respectively.

Then
\[
\nabla\phi_t = \nabla\phi_{t'} \qquad \mu_t\ae
\]
\end{Lemma}

On $\RCD$ spaces, $W_2$-geodesics  made of measures with bounded density also have the weak continuity property of the densities expressed by the following lemma. The proof  follows by a simple argument involving Young's measures and the continuity of the entropy along a geodesic (see Corollary 5.7 in \cite{Gigli13}):

\begin{Lemma}\label{lem:8}
Let $(\X,\sfd,\mm)$ be an $\RCD(K,\infty)$ space with $K \in \mathbb{R}$ and $(\mu_t)\subset \probt X$ a $W_2$-geodesic such that $\mu_t\leq C\mm$ for every $t\in[0,1]$ for some $C>0$. Let $\rho_t$ be the density of $\mu_t$.

Then for any $t \in [0,1]$ and any sequence $(t_n)_{n \in \N} \subset [0,1]$ converging to $t$ there exists a subsequence $(t_{n_k})_{k \in \N}$ such that
\[
\rho_{t_{n_k}} \to \rho_t, \quad \mm\ae
\]
as $k \to \infty$.
\end{Lemma}

We conclude recalling some properties of the {\bf Hopf-Lax semigroup} in metric spaces, also in connection with optimal transport. For $f : \X \to \R \cup \{+\infty\}$ and $t > 0$ we define the function $Q_tf : \X \to \R \cup \{-\infty\}$ as
\begin{equation}
\label{eq:hli}
Q_tf(x):= \inf_{y\in \X}\frac{\sfd^2(x,y)}{2t}+f(y)
\end{equation}
and set $t_* := \sup\{t > 0 \,:\, Q_tf(x) > -\infty \textrm{ for some } x \in \X\}$; it is worth saying that $t_*$ does not actually depend on $x$, since if $Q_tf(x) > -\infty$, then $Q_sf(y) > -\infty$ for all $s \in (0,t)$ and all $y \in \X$. With this premise we have the following result (see \cite{AmbrosioGigliSavare11}):

\begin{Proposition}\label{pro:11}
Let $(\X,\sfd)$ be a length space and $f : \X \to \R \cup \{+\infty\}$. Then for all $x \in \X$ the map $(0,t_*) \ni t \mapsto Q_tf(x)$ is locally Lipschitz and
\begin{equation}\label{eq:55}
\ddt Q_tf(x) + \frac{1}{2}\Big(\lip Q_tf(x)\Big)^2 = 0\qquad {\rm a.e.}\ t \in (0,t_*)
\end{equation}
\end{Proposition}

\bibliographystyle{siam}
\bibliography{biblio}

\end{document}